\documentclass[aos]{imsart}

\RequirePackage{amsthm,amsmath,amsfonts,amssymb}
\RequirePackage[numbers]{natbib}

\RequirePackage[colorlinks,citecolor=blue,urlcolor=blue]{hyperref}

\usepackage{amsfonts}
\usepackage{amsgen,amsmath,amsthm,amstext,amsbsy,amsopn,amssymb,subfigure}
\usepackage[dvips]{graphicx}
\usepackage{comment}
\usepackage{enumitem}
\usepackage{natbib}
\usepackage{booktabs}
\usepackage{blkarray}
\usepackage{algorithm}
\usepackage{algpseudocode}
\usepackage{url}
\usepackage{longtable}

\usepackage[usenames]{color}

\newcommand{\calG}{\mathcal{G}}

\allowdisplaybreaks

\newtheorem{Theorem}{Theorem}

\newtheorem{Lemma}{Lemma}
\newtheorem{Remark}{Remark}

\newtheorem{Proposition}{Proposition}

\newcommand{\be}{\begin{equation}}
	\newcommand{\ee}{\end{equation}}
\newcommand{\bea}{\begin{eqnarray}}
	\newcommand{\eea}{\end{eqnarray}}
\newcommand{\beas}{\begin{eqnarray*}}
	\newcommand{\eeas}{\end{eqnarray*}}

\newcommand{\X}{{\mathbf{X}}}
\newcommand{\B}{{\mathbf{B}}}

\newcommand{\Var}{{\rm Var}}

\newcommand{\Cov}{{\rm Cov}}

\newcommand{\Y}{{\mathbf{Y}}}
\newcommand{\Z}{{\mathbf{Z}}}
\newcommand{\bbP}{{\mathbb{P}}}
\newcommand{\bbE}{{\mathbb{E}}}

\newcommand{\rank}{{\rm rank}}

\newcommand{\tr}{{\rm tr}}

\newcommand{\diag}{{\rm diag}}

\newcommand{\SVD}{{\rm SVD}}

\newcommand{\argmin}{\mathop{\rm arg\min}}
\newcommand{\argmax}{\mathop{\rm arg\max}}

\newcommand{\sigmax}{\sigma_{\textrm{max}}}
\newcommand{\sigsum}{\sigma_{\textrm{sum}}}
\newcommand{\sigmaxch}{\check{\sigma}_{\textrm{max}}}
\newcommand{\sigsumch}{\check{\sigma}_{\textrm{sum}}}

\usepackage{prettyref}

\newrefformat{eq}{(\ref{#1})}
\newrefformat{chap}{Chapter~\ref{#1}}
\newrefformat{sec}{Section~\ref{#1}}
\newrefformat{alg}{Algorithm~\ref{#1}}
\newrefformat{al}{Algorithm~\ref{#1}}
\newrefformat{fig}{Fig.~\ref{#1}}
\newrefformat{tab}{Table~\ref{#1}}
\newrefformat{rmk}{Remark~\ref{#1}}
\newrefformat{clm}{Claim~\ref{#1}}
\newrefformat{def}{Definition~\ref{#1}}
\newrefformat{cor}{Corollary~\ref{#1}}
\newrefformat{lmm}{Lemma~\ref{#1}}
\newrefformat{lm}{Lemma~\ref{#1}}
\newrefformat{prop}{Proposition~\ref{#1}}
\newrefformat{pr}{Proposition~\ref{#1}}
\newrefformat{app}{Appendix~\ref{#1}}
\newrefformat{hyp}{Hypothesis~\ref{#1}}
\newrefformat{th}{Theorem~\ref{#1}}
\newrefformat{ass}{Assumption~\ref{#1}}

\newcommand{\stepa}[1]{\overset{\rm (a)}{#1}}
\newcommand{\stepb}[1]{\overset{\rm (b)}{#1}}
\newcommand{\stepc}[1]{\overset{\rm (c)}{#1}}
\newcommand{\stepd}[1]{\overset{\rm (d)}{#1}}

\allowdisplaybreaks

\startlocaldefs

\endlocaldefs

\begin{document}

\begin{frontmatter}
\title{Heteroskedastic PCA: Algorithm, Optimality, and Applications}
\runtitle{Heteroskedastic PCA}

\begin{aug}
\author[A]{\fnms{Anru R.} \snm{Zhang}\ead[label=e1]{anruzhang@stat.wisc.edu}},
\author[B]{\fnms{T. Tony} \snm{Cai}\ead[label=e2]{tcai@wharton.upenn.edu}}
\and
\author[C]{\fnms{Yihong} \snm{Wu}\ead[label=e3]{yihong.wu@yale.edu}}
\address[A]{Department of Statistics, University of Wisconsin-Madison and Department of Biostatistics \& Bioinformatics, Duke University, \printead{e1}}

\address[B]{Department of Statistics, The Wharton School, University of Pennsylvania, \printead{e2}}

\address[C]{Department of Statistics and Data Science, Yale University, \printead{e3}}
\end{aug}

\begin{abstract}
A general framework for principal component analysis (PCA) in the presence of heteroskedastic noise is introduced. We propose an algorithm called HeteroPCA, which involves iteratively imputing the diagonal entries of the sample covariance matrix to remove estimation bias due to heteroskedasticity. This procedure is computationally efficient and provably optimal under the  generalized spiked covariance model. A key technical step is a deterministic robust perturbation analysis on singular subspaces, which can be of independent interest. The effectiveness of the proposed algorithm is demonstrated in a suite of problems in high-dimensional statistics, including singular value decomposition (SVD) under heteroskedastic noise, Poisson PCA, and SVD for heteroskedastic and incomplete data.
\end{abstract}



\footnotetext{
The research of Anru Zhang was supported in part by NSF CAREER award DMS-1944904, NSF grant DMS-1811868, and NIH grant R01-GM131399-01. The research of Tony Cai was supported in part by NSF grants DMS-1712735 and  DMS-2015259 and NIH grants R01-GM129781 and R01-GM123056. The research of Yihong Wu was supported in part by the NSFgGrant CCF-1527105, an NSF CAREER award CCF-1651588, and an Alfred Sloan fellowship.}

\end{frontmatter}

\begin{sloppypar} 
	\section{Introduction}\label{sec:intro}
	
	Principal component analysis (PCA) is a ubiquitous tool in statistics, econometrics, machine learning, and applied mathematics. The central aim of PCA is to extract hidden low-rank structures from noisy observations. The spiked covariance model has been well studied and used as a baseline for both methodological and theoretical developments for PCA \citep{johnstone2001distribution,baik2005phase,baik2006eigenvalues,paul2007asymptotics,donoho2018optimal,nadler2008finite}. Under this model, one observes $Y_1,\ldots, Y_n \overset{iid}{\sim}N\left(\mu, \Sigma_0 + \sigma^2 I_p\right)$, where $\Sigma_0 = U\Lambda U^\top$ is a symmetric low-rank matrix and $I_p$ is a $p$-dimensional identity matrix. The spiked covariance model can be equivalently written as 
	\begin{equation}\label{eq:spikecov}
		Y_k = X_k + \varepsilon_k,\quad  X_k\overset{iid}{\sim} N(\mu, \Sigma_0),\quad \varepsilon_k\overset{iid}{\sim} N(0, \sigma^2 I_p), \quad k=1, \ldots, n.
	\end{equation}
	The goal is either to recover $\Sigma_0$, $\Lambda$, or $U$, and it is often done through the sample covariance matrix of $Y_1,\ldots, Y_n$, i.e.
	\begin{equation}\label{eq:sample-covariance}
		\widehat{\Sigma} = \frac{1}{n-1} (\Y - \bar{Y}1_n^\top)(\Y - \bar{Y}1_n^\top)^\top = \frac{1}{n-1}\sum_{k=1}^n (Y_k-\bar{Y})(Y_k-\bar{Y})^\top,
	\end{equation}
	where $\Y=[Y_1, \ldots, Y_n]$ and $\bar{Y} = \frac{1}{n}\sum_{k=1}^n Y_k.$ The asymptotic properties of eigenvalues and eigenvectors of $\widehat{\Sigma}$ have been well established in literature and their estimation based on the eigendecomposition of $\widehat{\Sigma}$ has been introduced and studied. A key assumption in these analyses is homoskedasticity, in the sense that each $\varepsilon_k$ is assumed to be spherically symmetric Gaussian.

	\subsection{Heteroskedastic PCA}
	\label{sec:intro-HPCA}
	
	In many applications, the noise term $\varepsilon_k$ can be highly heteroskedastic in the sense that the magnitude of noise entries varies significantly in the data matrix. Heteroskedastic noise is especially common in datasets with different types of variables. For example, in various biological sequencing and photon imaging data, the observations are discrete counts that are commonly modeled by Poisson, multinomial, or negative binomial distributions \citep{salmon2014poisson,cao2020multisample} and are naturally heteroskedastic. In network analysis and recommender systems, the observations are usually binary or ordinal, which are heteroskedastic as well.

	Motivated by these applications, it is natural to relax the homoskedasticity assumption in \prettyref{eq:spikecov} and consider the following generalized spiked covariance model \citep{bai2012sample,yao2015sample}:
	\begin{equation}\label{eq:multivariate-normal}
		\begin{split}
			& Y = X + \varepsilon,\quad \mathbb{E} X = \mu, \quad \Cov(X) = \Sigma_0,\\
			& \mathbb{E}\varepsilon = 0, \quad \Cov(\varepsilon_j) = \sigma_j^2, \\ 
			& \varepsilon = (\varepsilon_1,\ldots, \varepsilon_p)^\top;\quad  X, \varepsilon_1,\ldots, \varepsilon_p \text{ are independent}.
		\end{split}
	\end{equation}
	Here, $\Sigma_0$ is rank-$r$ and admits eigendecomposition $\Sigma_0 = U\Lambda U^\top$ with $U\in \mathbb{R}^{p\times r}$ and $\Lambda\in \mathbb{R}^{r\times r}$. $\sigma_1^2,\ldots, \sigma_p^2$ are unknown and not necessarily identical. This model is also widely used as the standard model in the literature of \emph{factor analysis}  (see, e.g., \cite{tipping1999probabilistic,ghosh2009default} and the references therein). Given i.i.d.~copies $Y_1,\ldots, Y_n$ drawn from \prettyref{eq:multivariate-normal} and the rank $r$, the goal is to estimate $U$. 
	
	Performing the classical PCA on data with  heteroskedastic noise can often lead to inconsistent estimates. The estimation of $U$ using the classical PCA is equivalent to the estimation of eigenvectors of the sample covariance matrix $\widehat{\Sigma}$.
	Since $\mathbb{E}\widehat{\Sigma} = \Sigma_0 + \diag(\sigma_1^2,\ldots, \sigma_p^2)$, the top eigenvectors of $\mathbb{E}\widehat{\Sigma}$ and $\Sigma_0$ will coincide when $\sigma_1^2,\ldots, \sigma_p^2$ are the same. But in the case of heteroskedastic noise, the differences in the bias terms $\sigma_1^2,\ldots, \sigma_p^2$ can lead to significant difference between the principal components of $\mathbb{E}\widehat{\Sigma}$ and those of $\Sigma_0$. Similar phenomena appear in other problems with heteroskedastic noise (see Section \ref{sec:applications} for details). 
	
	To cope with the bias on the diagonal elements of covariance matrix, \cite{florescu2016spectral} introduced the \emph{diagonal-deletion SVD} in the context of bipartite stochastic block model. The idea is to set the diagonal of the sample covariance matrix to zero before performing singular value decomposition.  However, it is a priori unclear whether zeroing out the diagonals is always the best choice, because it may change the singular subspace entirely. 
	
	In this paper, we introduce \emph{HeteroPCA}, a novel method for heteroskedastic principal component analysis. Instead of zeroing out the diagonal entries of the sample covariance/Gram matrix, we propose to iteratively update the diagonal entries based on the off-diagonals, so that the bias incurred on the diagonal is significantly reduced and more accurate estimation can be achieved. The performance of the proposed procedure is studied both theoretically and numerically. By establishing matching minimax upper and lower bounds, we show that HeteroPCA achieves the optimal rate of convergence for a range of settings under the generalized spiked covariance model. 
	
	Classic perturbation bounds, such as Davis-Kahan and Wedin's theorems \citep{davis1970rotation,wedin1972perturbation}, play key roles in the theoretical analysis of various PCA methods.
	These tools may not be suitable for the analysis of heteroskedastic PCA due to the aforementioned bias on the diagonal entries of the sample covariance matrix. To tackle this difficulty, we develop a new deterministic subspace perturbation bound (Theorem \ref{th:diagonal-less}), which provides the key technical tool for analyzing HeteroPCA and may be of independent interest.

	In addition to heteroskedastic PCA for the generalized spiked covariance model, the proposed HeteroPCA algorithm is applicable to a collection of high-dimensional problems with heteroskedastic data. Several applications are discussed in detail in Section \ref{sec:applications}, including \emph{SVD under heteroskedastic noise}, \emph{Poisson PCA}, and \emph{SVD for heteroskedastic and incomplete data}. Our results can also be useful in heteroskedastic canonical correlation analysis, heteroskedastic tensor SVD, exponential family PCA, and community detection in bipartite stochastic network.

	\subsection{Related Literature}
	
	\cite{bai2012sample,yao2015sample} extended the theory for regular spiked covariance model \eqref{eq:spikecov} to the generalized spiked covariance model and studied the limiting distribution of eigenvalues of the sample covariance matrix. \cite{hong2016towards,hong2018asymptotic,hong2018optimally} introduced an alternative model for heteroskedastic data, where the noise is non-uniform \emph{across} different samples but uniform \emph{within} each sample. Under this model, \cite{hong2016towards,hong2018asymptotic} studied the asymptotic performance of PCA and \cite{hong2018optimally} developed the optimal weights for weighted PCA with theoretical guarantees. \cite{vaswani2016correlated,vaswani2017finite-allerton,vaswani2018pca,vaswani2020fast} provide a comprehensive study on the methodology and theory of PCA where data-dependent, non-isotropic, or correlated noise and missing values may appear. The detailed comparison of our results and \cite{vaswani2018pca,vaswani2020fast} are given in later Remarks \ref{rm:implication-vaswani} and \ref{rm:vaswani-compare}.
	
	Our work is also closely related to a substantial body of literature on factor model analysis \citep{thomson1939factorial,lawley1962factor,tipping1999probabilistic,ghosh2009default,bai2012statistical,owen2016bi,wang2017asymptotics}. There have been various approaches developed to estimate the principal components in factor models, such as the regression method \citep{thomson1939factorial}, weighted least squares \citep{bartlett1937statistical}, EM \citep{tipping1999probabilistic}, and Bayesian MCMC \citep{ghosh2009default}. The asymptotic theory for factor model analysis was also extensively studied (e.g., \cite{bai2012statistical,wang2017asymptotics} and the references therein). Different from the previous results, this paper mainly concerns a non-asymptotic framework, providing algorithms with provable guarantees and allowing heteroskedastic noise within each sample in the high-dimensional regime that $n, p, r$ can all grow. 
	
	Matrix denoising \citep{donoho2014minimax,nadakuditi2014optshrink,gavish2017optimal,donoho2018optimal}, where the central goal is to estimate low-rank matrices from noisy observations, is closely related to this work. In order to get an accurate estimation of overall low-rank matrices from observations perturbed by random noise, the singular value thresholding \citep{donoho2014minimax,chatterjee2015matrix} and the singular value shrinkage \citep{nadakuditi2014optshrink,gavish2017optimal,donoho2018optimal} were proposed and widely studied recently. Departing from these previous results, this paper focuses on estimating the singular subspace instead of the overall matrix, which achieves better performance in singular subspace estimation than denoising the whole matrix by previous methods and then performing a rank-$r$ SVD. 
	
	In Section \ref{sec:matrix-completion}, we discuss the application of HeteroPCA to SVD for heteroskedastic and incomplete data. This problem is related to a body of literature on matrix completion that we will give a review in Section \ref{sec:matrix-completion}.

	\subsection{Organization of the Paper}
	
	After a brief introduction of notation and definitions (Section \ref{sec:notation}), we focus on the generalized spiked covariance model, present the HeteroPCA algorithm (Section \ref{sec:method-heteroPCA}), and develop matching minimax upper and lower bounds of the estimation error (Section \ref{sec:theory-heteroPCA}). Then, we introduce a deterministic robust perturbation analysis that serves as a key technical step in our analysis (Section \ref{sec:deterministic}). We also illustrate main proof ideas in \prettyref{sec:proof-sketch}. In Section \ref{sec:applications}, we discuss the applications of established results. Numerical results are given in Section \ref{sec:simu}. The proofs of main results are given in Section \ref{sec:proof}. The additional proofs and technical lemmas are provided in the supplementary materials \citep{zhang2018supplement}.
	
	\section{Optimal Heteroskedastic Principal Component Analysis}\label{sec:pca}
	
	\subsection{Notation and Preliminaries}
	\label{sec:notation}
	
	We use lowercase letters, e.g., $x, y$, to denote scalars or vectors and use uppercase letters, e.g, $U, M$ to denote matrices. For sequences of positive numbers $\{a_k\}$ and $\{b_k\}$, we write $a_k \lesssim b_k$ or $b_k\gtrsim a_k$ if there exists a uniform constant $C>0$ such that $a_k \leq Cb_k$ for all $k$. We also write $a_k \asymp b_k$ if $a_k\lesssim b_k$ and $a_k\gtrsim b_k$ both hold. For any matrix $M \in \mathbb{R}^{p_1\times p_2}$, let $\lambda_k(M)$ be the $k$-th largest singular value. Then, the SVD of $M$ can be written as $M = \sum_{k=1}^{p_1\wedge p_2}\lambda_k(M) u_kv_k^\top$. Let $\SVD_r(M) = [u_1 ~ \cdots u_r]$ be the collection of leading $r$ left singular vectors and ${\rm QR}(M)$ be the Q part of QR orthogonalization of $M$. The matrix spectral norm and Frobenius norm are defined as
	$\|M\| = \sup_{\|u\|_2=1} \|Mu\|_2 = \lambda_1(M)$ and $\|M\|_F = (\sum_{i,j}M_{ij}^2)^{1/2} = (\sum_k \lambda_k^2(M))^{1/2}$. Let $I_r$, $0_{m\times n}$, and $1_{m\times n}$ be the $r$-by-$r$ identity, $m\times n$ zero, and $m\times n$ all-one matrices, respectively. Also let $0_{m}$ and $1_{m}$ denote the $m$-dimensional zero and all-one column vectors. Denote $\mathbb{O}_{p, r} = \{U\in \mathbb{R}^{p\times r}: U^\top U=I_r\}$ as the set of all $p$-by-$r$ matrices with orthonormal columns. For $U\in \mathbb{O}_{p, r}$, we note $U_{\perp}\in \mathbb{O}_{p, p-r}$ as the orthogonal complement so that $[U ~ U_{\perp}]\in \mathbb{R}^{p\times p}$ is a complete orthogonal matrix. 
	
	Motivated by incoherence condition, a widely used assumption in the matrix completion literature \citep{candes2009exact}, we define \emph{incoherence constant} of $U\in \mathbb{O}_{p, r}$ as
	\begin{equation}
		I(U) = (p/r) \max_{i \in [p]} \|e_i^\top U\|_2^2.
		\label{eq:incoherence}
	\end{equation}
	For any $U_1, U_2\in \mathbb{O}_{p, r}$, we define the $\sin\Theta$ distance $\|\sin\Theta(U_1, U_2)\| \triangleq \|U_{1\perp}^\top U_2\| = \|U_{2\perp}^\top U_1\|$. For any square matrix $A$, let $\Delta(A)$ be $A$ with all diagonal entries set to zero and $D(A)$ be $A$ with all off-diagonal entries set to zero. Then $A = \Delta(A) + D(A)$. We define the Orlicz-$\phi_2$ norm of any random variable $Y$ as $\|Y\|_{\psi_2} = \sup_{q\geq 1}q^{-1/2}(\mathbb{E}|Y|^q)^{1/q}$. A random variable $Y$ is called $\sigma^2$-sub-Gaussian if $\|Y/\sigma\|_{\psi_2}\leq C$ for some constant $C>0$; a random vector $X$ is called $\Sigma$-sub-Gaussian if $\max_{q\geq 1, v\in \mathbb{R}^p}\|v^\top \Lambda^{-1/2} U^\top Y \|_{\psi_2}\leq C$ (here, $\Sigma = U\Lambda U^\top$ is the eigendecomposition of $\Sigma$) for some constant $C>0$. We use $C, C_1, \ldots, c, c_1, \cdots $ to respectively represent generic large and small constants, whose values may differ in different lines.
	
	\subsection{Methods for Heteroskedastic PCA}\label{sec:method-heteroPCA}
	
	Suppose one observes i.i.d. copies $Y_1,\ldots, Y_n$ of $Y$ from the generalized spiked covariance model \eqref{eq:multivariate-normal}. Let $\widehat{\Sigma}$ be the sample covariance matrix defined as \eqref{eq:sample-covariance}. The regular SVD estimator $\widetilde{U} = \SVD_r(\widehat{\Sigma})$, i.e., the leading $r$ left singular vectors of $\widehat{\Sigma}$, is the natural estimator of $U$, the leading singular vectors of $\Sigma_0$. 
	An important variant of Davis-Kahan's theorem \citep{davis1970rotation} given by Yu, Wang, and Samworth \citep{yu2014useful} yields
	\begin{equation}\label{ineq:tilde-U}
		\left\|\sin\Theta(\widetilde{U}, U)\right\| \lesssim \frac{\|\widehat{\Sigma} - (\Sigma_0 + \beta I_p)\|}{\lambda_r(\Lambda)}\wedge 1,
	\end{equation}
	which holds for any scalar $\beta\geq 0$ and cannot be improved in general. As briefly discussed earlier, since $\mathbb{E}\widehat{\Sigma} = \Sigma_0 + \diag(\sigma_1^2,\ldots, \sigma_p^2)$, when $\sigma_1^2,\ldots, \sigma_p^2$ have different values, the diagonal entries of the perturbation matrix $\widehat{\Sigma} - (\Sigma_0 + \beta I_p)$ may be significantly larger than the rest. As a result, $\widetilde{U}$ can be a suboptimal estimator for $U$.
	
	To achieve a more accurate estimate of $U$, we propose the following Algorithm \ref{al:heteroPCA} named HeteroPCA. The central idea is to iteratively impute the diagonal entries of the sample covariance matrix $\widehat{\Sigma}$ by the diagonals of its low-rank approximation.
	\begin{algorithm}[!h]
		\caption{HeteroPCA}
		\begin{algorithmic}[1]
			\State Input: matrix $\widehat{\Sigma}$, rank $r$, maximum number of iterations $T$.
			\State Initialize by setting the diagonal of $\widehat{\Sigma}$ to zero: $N^{(0)} = \Delta(\widehat{\Sigma})$, $t = 0$.
			\Repeat
			\State Perform SVD on $N^{(t)}$ and let $\widetilde{N}^{(t)}$ be its best rank-$r$ approximation:
			\begin{equation*}
				\begin{split}
					& N^{(t)} = U^{(t)}\Sigma^{(t)}(V^{(t)})^\top = \sum_{i} \lambda_i^{(t)} u_i^{(t)} (v_i^{(t)})^\top, \quad \lambda_1^{(t)}\geq \cdots\geq \lambda_{m}^{(t)} \geq 0,\\ 
					& \widetilde{N}^{(t)} = \sum_{i=1}^r \lambda_i^{(t)} u^{(t)}_i (v^{(t)}_i)^\top. 
				\end{split}
			\end{equation*}
			\State Update $N^{(t+1)} = D(\widetilde{N}^{(t)}) + \Delta(N^{(t)})$, i.e., replace the diagonal entries of $N^{(t)}$ by those in $\widetilde{N}^{(t)}$:
			\begin{equation}
				N^{(t+1)}_{ij} = \left\{\begin{array}{ll}
					N^{(t)}_{ij} = \widetilde{N}^{(t)}_{ij}, & i=j;\\
					\widehat{\Sigma}_{ij}, & i\neq j.
				\end{array}\right.
			\end{equation}
			\State $t = t+1$.
			\Until{convergence or maximum number of iterations reached.}
			\State Output: $\widehat{U} = U^{(T)} = [u_1^{(T)} ~ \cdots ~ u_r^{(T)}]$.
		\end{algorithmic}
		\label{al:heteroPCA}
	\end{algorithm}
	In Algorithm \ref{al:heteroPCA}, since $\widehat{\Sigma}, N^{(t)}$ are symmetric, we have $u_i^{(t)} = v_i^{(t)}$ or $-v_i^{(t)}$. In contrast to most previous work on matrix completion and robust PCA, where the entries to be imputed are missing at random, here our goal is to impute the diagonal entries. Moreover, HeteroPCA can be interpreted as the projection gradient descent (PGD) for the following rank-constrained (non-convex) optimization problem: 
	\begin{equation}
		\min_{\rank(\tilde{N})\leq r} \left\|\Delta(\hat{\Sigma} - \tilde{N})\right\|_F^2.
		\label{eq:rankconst}
	\end{equation}
	To see this connection, we first note that $\tilde{N}^{(t)}$ is the best rank-$r$ approximation of $N^{(t)}$, which correspond to the projection step in PGD; next, $\nabla_{\tilde{N}}\|\Delta(\hat{\Sigma} - \tilde{N})\|_F^2 = 2\Delta(\tilde{N} - \hat{\Sigma})$ and the operator norm of $\nabla^2_{\tilde{N}}\|\Delta(\hat{\Sigma} - \tilde{N})\|_F^2$ is 2, where $\nabla_{\tilde{N}}$ and $\nabla_{\tilde{N}}^2$ are the gradient and Hessian with respect to $\tilde{N}$, respectively. Based on the theory of PGD (see, e.g., Section 3.3 in \cite{bubeck2015convex}), the smoothness parameter of the loss function $\beta= 2$ and the update,
	\begin{equation*}
		\begin{split}
			N^{(t+1)} = & \tilde{N}^{(t)} - (1/\beta)\nabla_{\tilde{N}^{(t)}}\|\Delta(\hat{\Sigma}-\tilde{N}^{(t)})\|_F^2 = \tilde{N}^{(t)} - \Delta(\tilde{N}^{(t)}-\hat{\Sigma}) \\
			= & D(\tilde{N}^{(t)}) + \Delta(\hat{\Sigma}) = D(\tilde{N}^{(t)}) + \Delta(N^{(t)}),
		\end{split}
	\end{equation*}
	corresponds to the gradient descent step in PGD.
	Due to the non-convexity of \eqref{eq:rankconst}, existing convergence results for PGD do not apply to Algorithm \ref{al:heteroPCA}, for which we provide a direct analysis next.
	
	\subsection{Theoretical Analysis}\label{sec:theory-heteroPCA}
	Denote 
	\begin{equation*} 
		\sigmax^2 \triangleq \max_{i} \sigma_i^2, \quad \sigsum^2 \triangleq \sum_i \sigma_i^2.
		\label{eq:sigmasigma}
	\end{equation*}
	Recall that $\Sigma_0 = U\Lambda U^\top$ is rank-$r$, so $\lambda_r(\Lambda)$ is the smallest non-zero eigenvalue of $\Sigma_0$. We have the following theoretical guarantee for Algorithm \ref{al:heteroPCA}. 
	\begin{Theorem}[Heteroskedastic PCA: upper bound]\label{th:heterogeneous-PCA}	
		Consider the generalized spiked covariance model \eqref{eq:multivariate-normal}, where $X_i$ and $\varepsilon_i$ are $\Sigma$-sub-Gaussian and $\sigma_i$-Gaussian, respectively. 
		Let $Y_1,\ldots, Y_n$ be i.i.d.~samples from \eqref{eq:multivariate-normal}. Assume $n\geq (C_0r) \wedge C_0\log\left(\sigma_r(\Lambda)/\sigsum^2\right)$ and $\|\Lambda\|/\lambda_r(\Lambda) \leq C_0$ for constant $C_0>0$. There exists constant $c_I>0$ such that if the incoherence constant $I(U)$ (defined in \eqref{eq:incoherence}) satisfies $I(U) \leq c_I p/r$, then the output $\widehat U$ of Algorithm \ref{al:heteroPCA} applied to the sample covariance matrix $\widehat{\Sigma}$ with the number of iterations $T = \Omega\left(\log \left(n\lambda_r(\Lambda)/\sigsum^2\right) \vee 1\right)$ satisfies:
		\begin{equation}\label{ineq:upper-bound-hetero-PCA}
			\begin{split}
				\mathbb{E}\left\|\sin\Theta(\widehat{U}, U)\right\| \leq \frac{C}{\sqrt{n}}\left(\frac{\sigsum + r^{1/2}\sigmax}{\lambda_r^{1/2}(\Lambda)} + \frac{\sigsum \sigmax}{\lambda_r(\Lambda)}\right)\wedge 1.
			\end{split}
		\end{equation}
		Here, the constant $C$ relies on $c_I, C_0$, but not $X$, $\sigma_i, p, r, n$. 
	\end{Theorem}
	\begin{Remark}{\rm 
			Let $\widetilde{p} = \sigsum^2/\sigmax^2$. Then \eqref{ineq:upper-bound-hetero-PCA} can be rewritten as
			\begin{equation}
				\mathbb{E}\left\|\sin\Theta(\widehat{U}, U)\right\| \lesssim \left(\sqrt{\frac{\widetilde{p}\vee r}{n}}\frac{\sigmax}{\lambda_r^{1/2}(\Lambda)} + \sqrt{\frac{\widetilde{p}}{n}}\frac{ \sigmax^2}{\lambda_r(\Lambda)}\right)\wedge 1.
				\label{eq:heteropca}
			\end{equation}
			Consider the homoskedastic PCA setting where $\sigma_1^2 = \cdots = \sigma_p^2 = \sigmax^2$. This special case of Theorem \ref{th:heterogeneous-PCA} yields:
			\begin{equation}
				\mathbb{E}\left\|\sin\Theta(\widehat{U}, U)\right\| \lesssim \sqrt{\frac{p}{n}}\left(\frac{\sigmax}{\lambda_r^{1/2}(\Lambda)} + \frac{\sigmax^2}{\lambda_r(\Lambda)}\right)\wedge1.\
				\label{eq:regularpca}
			\end{equation}	
			Comparing \prettyref{eq:heteropca} 	with \prettyref{eq:regularpca}, we see that 
			a weighted average between $\widetilde{p}\vee r$ and $\widetilde{p}$ can be viewed as the ``effective dimension" for heteroskedastic PCA. 
		}
	\end{Remark}
	
	\begin{Remark}\label{rm:implication-vaswani}\rm
		Recently, \cite{vaswani2017finite-allerton,vaswani2020fast} studied the PCA for matrix data with non-isotropic and data-dependent noise. In our notation and with some mild regularity conditions,  \cite[Part 2, Corollary 2.7]{vaswani2020fast} shows that the regular SVD estimator $\widetilde{U} = \SVD_r(Y)$ satisfies
		\begin{equation}
			\|\sin\Theta(\widetilde{U}, U)\| \lesssim\left(\sqrt{\frac{p}{n}}\frac{\sigmax}{\lambda_{r}^{1/2}(\Lambda)} +  \sqrt{\frac{p}{n}}\frac{\sigmax^2}{\lambda_{r}(\Lambda)} + \frac{\|U_{\perp}^\top \Sigma_\varepsilon U\|}{\lambda_{\min}(\Lambda)}\right)\wedge 1
		\end{equation}
		with high probability. Here, $\Sigma_\varepsilon$ is the covariance matrix of the noise vector $\varepsilon_k$. Since $\sigsum \leq \sqrt{p} \sigmax$ and $\|U_{\perp}^\top \Sigma_{\varepsilon} U\|\geq 0$, our Theorem \ref{th:heterogeneous-PCA} yields a better estimation error rate.
	\end{Remark}
	
	Next, we establish the optimality of \prettyref{th:heterogeneous-PCA}. Consider the following class of generalized spiked covariance matrices:
	\begin{equation}\label{eq:hetero-PCA-class}
		\begin{split}
			& \mathcal{F}_{p, n, r}(\sigsumch, \sigmaxch, \nu, \kappa) = \Bigg\{\Sigma = U\Lambda U^\top + D:\\ 
			& \quad \begin{array}{l}
				D \text{ is non-negative diagonal}, \sum_iD_{ii} \leq \sigsumch^2, \max_i D_{ii} \leq \sigmaxch^2, \\
				U\in \mathbb{O}_{p, r},  I(U) \leq c_Ip/r , \|\Lambda\|/\lambda_r(\Lambda) \leq \kappa, \lambda_r(\Lambda) \geq \nu
			\end{array}\Bigg\}.
		\end{split}
	\end{equation}
	\begin{Theorem}[Heteroskedastic PCA: lower bound]\label{th:lower-bound-hetero-PCA}
		Suppose $\sqrt{p}\sigmaxch\geq\sigsumch\geq \sigmaxch>0$, $\kappa\geq 1$. There exists constant $C>0$, such that if $p\geq Cr$, we have
		\begin{equation}\label{eq:PCA-lower-bound}
			\inf_{\widehat{U}} \sup_{\Sigma\in \mathcal{F}_{p,n,r}(\sigsumch, \sigmaxch, \nu, \kappa)} \mathbb{E}\left\|\sin\Theta(\widehat{U}, U)\right\| \gtrsim \frac{1}{\sqrt{n}}\left(\frac{\sigsumch + r^{1/2}\sigmaxch}{\nu^{1/2}} + \frac{\sigsumch\sigmaxch}{\nu}\right)\wedge 1.
		\end{equation}
	\end{Theorem}
	
	\begin{Remark}{\rm 
			By combining Theorems \ref{th:heterogeneous-PCA} and \ref{th:lower-bound-hetero-PCA}, the proposed \prettyref{al:heteroPCA} achieves the following optimal estimation error rate in $\mathcal{F}_{p,n,r}(\sigsumch, \sigmaxch, \nu, \kappa)$ when the condition number $\kappa$ is a constant:
			\begin{equation*}
				\inf_{\widehat{U}} \sup_{\Sigma\in \mathcal{F}_{p,n,r}(\sigsumch, \sigmaxch, \nu, C)} \mathbb{E}\left\|\sin\Theta(\widehat{U}, U)\right\| \asymp \frac{1}{\sqrt{n}}\left(\frac{\sigsumch + r^{1/2}\sigmaxch}{\nu^{1/2}} + \frac{\sigsumch\sigmaxch}{\nu}\right)\wedge 1.
			\end{equation*}
		}
	\end{Remark}
	
	Next, we consider the performance of HeteroPCA if the covariance matrix $\Sigma_0$ is approximately low-rank.
	\begin{Proposition}[HeteroPCA for approximately low-rank covariance]\label{pr:heteroPCA-approximate}
		Consider the generalized spiked covariance model \eqref{eq:spikecov}. Suppose $\Sigma_0 = \widetilde{U}\Lambda \widetilde{U}^\top$ is the eigenvalue decomposition, where $\widetilde{U} = [U ~ U_\perp]$ and $U$ is the collection of leading $r$ singular vectors. Assume $X$ and $\varepsilon_i$ are $\Sigma$-sub-Gaussian and $\sigma^2_i$-sub-Gaussian, respectively. 
		Also assume that $n\geq Cr$,
		$n\wedge p \geq C\left(\sigsum^2/\sigma_r(\Lambda)\right)$, 
		and $\|\Lambda\|/\lambda_r(\Lambda)\leq C$ for some constant $C>0$.
		Then there exists some constant $c_I>0$ such that if the incoherence constant $I(U)$ (defined in \eqref{eq:incoherence}) satisfies $I(U)\leq c_Ip/r$, then the output $\widehat U$ of Algorithm \ref{al:heteroPCA} with the input matrix $\widehat{\Sigma}$ and number of iterations $T = \Omega\left(\log \left(n\lambda_r(\Lambda)/\sigsum^2\right) \vee 1\right)$ satisfies
		\begin{equation*}
			\begin{split}
				& \mathbb{E}\left\|\sin\Theta(\widehat{U}, U)\right\| \\
				\lesssim & \left(\frac{\sigsum+\sqrt{r}\sigmax}{n^{1/2}\lambda_r^{1/2}(\Lambda)} + \frac{\sigsum\sigmax}{n^{1/2}\lambda_r(\Lambda)} + \frac{((np)^{1/2}+p)\lambda_{r+1}^{1/2}(\Lambda)}{n\lambda_r^{1/2}(\Lambda)} + \frac{\lambda_{r+1}(\Lambda)}{\lambda_r(\Lambda)}\right)\wedge 1.
			\end{split}
		\end{equation*}
	\end{Proposition}
	Proposition \ref{pr:heteroPCA-approximate} shows that HeteroPCA can estimate the loading matrix $U$ accurately if there exists a significant gap between $\lambda_r(\Sigma_0)$ and $\lambda_{r+1}(\Sigma_0)$.
	
	\subsection{A Deterministic Robust Perturbation Analysis}\label{sec:deterministic}
	
	In this section, we temporarily ignore the randomness of $X_k$s and $\varepsilon_k$s and focus on a more general prototypical model of the heteroskedastic PCA problem in \prettyref{sec:intro-HPCA}. Let $N, M, Z$ be deterministic symmetric matrices (not necessarily positive definite) that satisfy
	\begin{equation}\label{eq:model-1}
		N = M + Z \in \mathbb{R}^{p\times p}.
	\end{equation}
	Here $N$ is the observation, $M$ is the rank-$r$ matrix of interest, and $Z\in \mathbb{R}^{p\times p}$ is the perturbation that possibly has significantly large amplitude in its diagonal entries. In the heteroskedastic PCA model, $N, M, Z$ may represent the sample covariance matrix $\widehat{\Sigma}$, population covariance matrix $\Sigma_0$, and their difference, respectively. 
	Let $U\in \mathbb{O}_{p, r}$ be the first $r$ singular vectors of $M$. As discussed earlier, applying the proposed HeteroPCA (Algorithm \ref{al:heteroPCA}) to matrix $N$ provides an adaptive estimate of $U$. In the following theorem, we demonstrate the theoretical property for the proposed Algorithm \ref{al:heteroPCA} under the general robust perturbation model \eqref{eq:model-1}.

	\begin{Theorem}[Robust $\sin\Theta$ theorem]\label{th:diagonal-less}
		Suppose $M\in \mathbb{R}^{p\times p}$ is a rank-$r$ symmetric matrix and $U\in \mathbb{O}_{p, r}$ consists of the eigenvectors of $M$. Let $\widehat{U}^{(t)} = [u_1^{(t)} \cdots u_r^{(t)}]$ be the intermediate result of Algorithm \ref{al:heteroPCA} with input matrix $N$ after $t$ iterations. There exists a universal constant $c_I>0$ such that if 
		\begin{equation}\label{ineq:incoherence-constant-dl-d-k}
			I(U) \|M\|/\lambda_r(M) \leq c_I p/r,
		\end{equation}
		where $I(U)$ is the incoherence constant defined in \prettyref{eq:incoherence}, then
		$$	\left\|\sin\Theta(\widehat{U}^{(t)}, U)\right\| \leq \frac{4\|\Delta(Z)\| }{\lambda_r(M)} + 2^{-(t+3)}.$$ 
		In particular if $T = \Omega(\log\frac{\lambda_r(M)}{\eta\|\Delta(Z)\|}\vee 1)$, the final outcome $\widehat{U}$ satisfies
		\begin{equation}\label{ineq:dl-d-k-upper-bound}
			\left\|\sin\Theta(\widehat{U}, U)\right\|\lesssim \frac{\|\Delta(Z)\|}{\lambda_r(M)}\wedge 1.
		\end{equation}
	\end{Theorem}
	A matching lower bound and several discussions to the robust $\sin\Theta$ theorem is given in Section \ref{sec:addition-sin-theta} in the supplementary materials.
	\begin{Remark}\rm 
		Distinct from the matrix completion, where most entries are missing from the target matrix, the substantial corrupted entries only lie in the diagonal of the target Gram matrix/sample covariance matrix in our problem. Thus, a much looser condition on incoherence, $I(U) < cp/r$, is sufficient compared to the one required by matrix completion, $I(U)<\mu$ with $\mu$ being a constant.
	\end{Remark}

	\subsection{Proof Sketches of Main Technical Results}\label{sec:proof-sketch}
	
	We briefly discuss the proofs of Theorems \ref{th:heterogeneous-PCA}, \ref{th:lower-bound-hetero-PCA}, and \ref{th:diagonal-less} in this section. 
	
	The proof of Theorem \ref{th:heterogeneous-PCA} consists of three main steps. First, we define $\widehat{\Sigma}_X$ as the sample covariance matrix of signal vectors $X_1,\ldots, X_n$ and $E = [\varepsilon_1 \cdots \varepsilon_n]$ as the noise matrix. We aim to develop a concentration inequality for $\Delta\left((n-1)(\widehat{\Sigma} - \widehat{\Sigma}_X)\right)$, i.e., the off-diagonal part of the perturbation. To this end, we decompose $(n-1)(\widehat{\Sigma} - \widehat{\Sigma}_X)$ into $(XE^\top +EX^\top), (EE^\top), n(\bar{X}\bar{E}^\top + \bar{E}\bar{X}^\top + \bar{E}\bar{E}^\top)$, 
	then bound them separately by heteroskedastic Wishart concentration inequality \citep{cai2020non} and Lemma \ref{lm:orthogonal-projection} in the supplementary materials. Second, we develop a lower bound for  $\lambda_r(\widehat{\Sigma}_X)$, i.e., the least non-trivial singular value of the signal covariance matrix. Finally, we apply the robust $\sin\Theta$ theorem (Theorem \ref{th:diagonal-less}), to complete the proof.

	To show the lower bound in Theorem \ref{th:lower-bound-hetero-PCA}, it suffices to show the two terms in \eqref{eq:PCA-lower-bound} separately; c.f.,~\eqref{ineq:PCA-lower-1} and \eqref{ineq:PCA-lower-2} in the detailed proof. To show each individual lower bound, we construct a series of ``candidate matrices" $\{U^{(k)}, \Sigma^{(k)}\}_{k=1}^N$ in $\mathcal{F}_{p, n, r}(\sigsum, \sigmax, \nu)$ so that $\{U^{(k)}\}_{k=1}^N$ are well-separated while distinguishing them apart based on random sample $Y_1,\ldots, Y_n \sim N(0, \Sigma^{(k)})$ is impossible. This implies the desired lower bound by applying Fano's method.

	The proof of Theorem \ref{th:diagonal-less} is the main technical contribution of this paper. Specifically, we analyze how the estimation error $K_t = \|N^{(t)} - M\|$ decays at each iteration. We first obtain an initialization error bound. Then for each $t$, we decompose $K_t$ into four terms, bound them separately, and obtain an inequality that relates $K_t$ to $K_{t-1}$ (see \eqref{eq:Ktt}). By induction, this recursive inequality leads to the exponential decay of $K_t$ and implies the desired upper bound. Note that Algorithm \ref{al:heteroPCA} can be viewed as successive compositions involving the projection operator $P_U(\cdot)$ and the diagonal-deletion operator $D(\cdot)$. We thus introduce Lemma \ref{lm:diagonal-projection-condense} to give sharp operator norm upper bounds for compositions of $P_U(\cdot)$ and $D(\cdot)$. At the heart of the proof of Theorem \ref{th:diagonal-less}, this lemma is useful for bounding the error at both the initialization and the subsequent iterations.

	\section{Further Applications in High-dimensional Statistics}\label{sec:applications}

	\subsection{SVD under Heteroskedastic Noise}\label{sec:heteroskedastic perturbation}
	
	Suppose one observes 
	\begin{equation}
		Y = X + E,
		\label{eq:denoising}
	\end{equation}
	where $X$ is the low-rank matrix of interest and the entries of noise $E$ are independent, zero-mean, but not necessarily homoskedastic. The goal is to recover the left singular subspace of $X$ based on noisy observation $Y$. The problem arises naturally in a range of applications, such as magnetic resonance imaging (MRI) and relaxometry \citep{candes2013unbiased}. This model can also be viewed as a prototype of various problems in high-dimensional statistics and machine learning, including Poisson PCA \citep{salmon2014poisson}, bipartite stochastic block model \citep{florescu2016spectral}, and  exponential family PCA \citep{liu2018pca}. Let the sample and population Gram matrices be $N = YY^\top$ and $M = XX^\top$, respectively. Then,
	\begin{equation*}
		\begin{split}
			& \left(\mathbb{E} N\right)_{ij} = \left\{\begin{array}{ll}
				M_{ij}, & i \neq j;\\
				M_{ij} + \sum_{k=1}^{p_2}\Var(E_{ik}), & i = j.
			\end{array}\right.\\
		\end{split}
	\end{equation*}
	Thus, only the off-diagonal entries of $N$ are unbiased estimators of the corresponding entries of $M$. When $\Var(E_{ij})$ are unequal, there can be significant differences between the spectrum of $\mathbb{E}N, \mathbb{E}\Delta(N)$, and $M$.  
	Since left singular vectors of $Y$ and $X$ are respectively identical to those of $N$ and $M$, the regular SVD or diagonal-deletion SVD on $Y$ can result in inconsistent estimates of the left singular subspace of $X$. 
	
	Compared to the regular or diagonal-deletion SVD, the next theorem shows the proposed HeteroPCA can be a better approach. 
	\begin{Theorem}
		\label{th:upper_bound_SVD}
		Consider the model \eqref{eq:denoising}. Suppose $X\in \mathbb{R}^{p_1\times p_2}$ is a fixed rank-$r$ matrix, the noise matrix $E$ has independent entries, $\mathbb{E}E_{ij} = 0$, $\Var(E_{ij}) =\sigma_{ij}^2$, and $E_{ji}$ is $\sigma_{ij}^2$-sub-Gaussian. Suppose the left singular subspace of $X$ is $U\in \mathbb{O}_{p_1, r}$. Assume that the condition number of $X$ is at most some absolute constant $C$, i.e.,	$\|X\|\leq C\lambda_r(X)$. Denote 
		\begin{equation}
			\sigma_R^2 = \max_i\sum_{j=1}^{p_2}\sigma_{ij}^2, \quad \sigma_C^2 = \max_j\sum_{i=1}^{p_1}\sigma_{ij}^2,\quad \sigmax^2 = \max_{ij} \sigma_{ij}^2
		\end{equation}
		as the rowwise, columnwise, and entrywise noise variances. Then there exists a constant $c_I>0$ such that if $U$ satisfies $I(U) = \max_{1\leq i \leq p_1} \frac{p_1}{r}\|e_i^\top U\|_2^2 \leq c_I p_1/r$,
		Algorithm \ref{al:heteroPCA} applied to $YY^\top$ with rank $r$ and number of iterations $T = \Omega\left(\log(\lambda_r(X)/\sigma_C)\vee1\right)$
		outputs $\widehat U$ that satisfies 
		\begin{equation}\label{ineq:upper_bound}
			\begin{split}
				& \mathbb{E}\left\|\sin\Theta(\widehat{U}, U)\right\| \\
				\lesssim &  \left(\frac{\sigma_C+\sqrt{r}\sigmax}{\lambda_r(X)}+\frac{\sigma_R\sigma_C + \sigma_R\sigmax \sqrt{\log (p_1\wedge p_2)} + \sigmax^2\log(p_1\wedge p_2)}{\lambda_r^2(X)}\right)\wedge 1.\\
			\end{split}
		\end{equation}
		If $\sigmax \lesssim \sigma_C/\max\{\sqrt{r}, \sqrt{\log (p_1\wedge p_2)}\}$ additionally holds, i.e., the variance array $\{\sigma_{ij}^2\}$ is not too ``spiky," we further have
		\begin{equation}\label{ineq:upper_bound_simpl1}
			\mathbb{E}\left\|\sin\Theta(\widehat{U}, U)\right\| \lesssim \left(\frac{\sigma_C}{\lambda_r(X)} + \frac{\sigma_R\sigma_C}{\lambda_r^2(X)}\right)\wedge 1.
		\end{equation}
	\end{Theorem}

	\begin{Remark}{\rm
			When $\sigma_{ij} = \sigmax$ for all $i, j$, \eqref{ineq:upper_bound} reduces to
			\begin{equation*}
				\mathbb{E}\left\|\sin\Theta(\widehat{U}, U)\right\| \lesssim \left(\frac{\sqrt{p_1}\sigmax}{\lambda_r(X)} + \frac{\sqrt{p_1p_2}\sigmax}{\lambda_r^2(X)}\right),
			\end{equation*}
			which matches the optimal rate for SVD under homoskedastic noise in the literature \citep[Theorems 3 and 4]{cai2016minimax}.
	}\end{Remark}
	\begin{Remark}{\rm
			In contrast to the scaling of $\lambda_r(\Lambda)$ in Theorems 1 and 2, the scale of $\lambda_r(X)$ in Theorem \ref{th:upper_bound_SVD} implicitly grows with both $p_1$ and $p_2$ as $X$ is a $p_1$-by-$p_2$ matrix.
		}
	\end{Remark}
	
	\subsection{Poisson PCA}\label{sec:Poisson PCA}

	Poisson PCA \citep{liu2018pca} is an important problem in statistics and engineering with a range of applications, including photon-limited imaging \citep{salmon2014poisson} and biological sequencing data analysis \citep{cao2020multisample}. Suppose we observe $Y\in \mathbb{R}^{p_1\times p_2}$, where $Y_{ij}\overset{ind}{\sim} {\rm Poisson}(X_{ij})$ and $X\in \mathbb{R}^{p_1\times p_2}$ is rank-$r$. Let $X = U\Lambda V^\top$ be the singular value decomposition, where $U\in \mathbb{O}_{p_1, r}, V\in \mathbb{O}_{p_2, r}$. The goal is to estimate the leading singular vectors of $X$, i.e., $U$ or $V$, based on $Y$. HeteroPCA is an appropriate method for Poisson PCA since it can well handle the heteroskedasticity of Poisson distribution. Although the aforementioned heteroskedastic low-rank matrix denoising can be seen as a prototype problem of Poisson PCA, Theorem \ref{th:upper_bound_SVD} is not directly applicable and more careful analysis is needed since the Poisson distribution has heavier tail than sub-Gaussian.
	\begin{Theorem}[Poisson PCA]\label{th:poisson}
		Suppose $X$ is a nonnegative $p_1$-by-$p_2$ matrix, $\rank(X)= r$, $\lambda_1(X)/\lambda_r(X) \leq C$, $X_{ij}\geq c$ for constant $c>0$, $U\in \mathbb{O}_{p_1, r}$ is the left singular subspace of $X$. Denote
		\begin{equation}
			\sigma^2_R = \max_{i} \sum_{j=1}^{p_2} X_{ij}, \quad  \sigma^2_C = \max_j \sum_{i=1}^{p_1} X_{ij}, \quad \sigma^2_\ast = \max_{i,j} X_{ij}.
		\end{equation}
		Suppose one observes $Y\in \mathbb{R}^{p_1\times p_2}, Y_{ij} \overset{ind}{\sim} {\rm Poisson}(X_{ij})$. Then there exists constant $c_I>0$ such that if $U$ satisfies $I(U) = \max_i \frac{p_1}{r}\|e_i^\top U\|_2^2 \leq c_I p_1/r$,
		the proposed HeteroPCA procedure (Algorithm \ref{al:heteroPCA}) on matrix $YY^\top$ with rank $r$ and number of iterations $T = \Omega\left(\log(\lambda_r(X)/\sigma_C)\vee1\right)$ yields
		\begin{equation}
			\begin{split}
				& \mathbb{E}\|\sin\Theta(\widehat{U}, U)\|\\ 
				\lesssim &  \left(\frac{\sigma_C+r\sigmax}{\lambda_r(X)} + \frac{\left\{\sigma_R+\sigma_C + \sigmax\sqrt{\log(p_2)\log(p_1)}\right\}^2-\sigma_R^2}{\lambda_r^2(X)}\right)\wedge 1.\\
			\end{split}
		\end{equation}
		In addition, if $\sigmax \leq \sigma_C/\max\{r, \sqrt{\log (p_1)\log(p_2)}\}$, then
		\begin{equation*}
			\mathbb{E}\left\|\sin\Theta(\widehat{U}, U)\right\| \lesssim \left(\frac{\sigma_C}{\lambda_r(X)} + \frac{\sigma_R\sigma_C}{\lambda_r^2(X)}\right)\wedge 1.
		\end{equation*}
	\end{Theorem}

	\subsection{SVD for Heteroskedastic and Incomplete Data}\label{sec:matrix-completion}
	
	Missing data problems arise frequently in high-dimensional statistics. Let $X\in \mathbb{R}^{p_1\times p_2}$ be a rank-$r$ unknown matrix. Suppose only a small fraction of entries of $X$, denoted by $\Omega\subseteq [p_1]\times [p_2]$, are observable with random noise,
	$$Y_{ij} = X_{ij} + Z_{ij}, \quad (i, j) \in \Omega.$$ 
	Here, each entry $Y_{ij}$ is observed or missing with probability $\theta$ or $1-\theta$ for some $0<\theta<1$ and $Z_{ij}$'s are independent, zero-mean, and possibly heteroskedastic. Let $R\in \mathbb{R}^{p_1\times p_2}$ be the indicator of observable entries: 
	$$R_{ij}=\left\{\begin{array}{ll}
		1, & (i, j)\in \Omega;\\
		0, & (i, j) \notin \Omega,
	\end{array}\right.$$ 
	and $R$ and $Y$ are independent. Assume $X = U\Lambda V^\top$ is the singular value decomposition, where $U\in \mathbb{O}_{p_1, r}$ and $V\in \mathbb{O}_{p_2, r}$. Denote $\widetilde{Y}$ as the entry-wise product of $Y$ and $R$, i.e., $\widetilde{Y}_{ij} = Y_{ij} R_{ij}, \forall (i,j)\in [p_1]\times[p_2]$. We aim to estimate $U$ based on $\{Y_{ij}, (i,j)\in \Omega\}$ or equivalently $\widetilde{Y}_{ij}$s.  This problem is heteroskedastic since $\mathbb{E}\widetilde{Y}_{ij} = \theta X_{ij}$ and $\Var(\widetilde{Y}_{ij})$ may vary for different $(i,j)$ pairs. We can apply HeteroPCA to $\widetilde{Y}\widetilde{Y}^\top$ to estimate $U$. The following theoretical guarantee holds.
	\begin{Theorem}\label{th:matrix-completion}
		Let $X$ be a $p_1$-by-$p_2$ rank-$r$ matrix, whose left singular subspace is denoted by 
		$U\in \mathbb{O}_{p_1, r}$. Assume that $\mathbb{E}Y = X$. Suppose $Y$ satisfies $\max_{ij}\|Y_{ij}\|_{\psi_2} \leq C$ and all entries $Y_{ij}$ are independent. 
		Suppose $0 < \theta \leq 1-c$ for constant $c>0$. There exists constant $c_I>0$ such that if $U\in \mathbb{O}_{p_1\times r}$ satisfies $I(U)\|X\|/\lambda_r(X) \leq c_Ip_1/r$, HeteroPCA applied to $\widetilde{Y}\widetilde{Y}^\top$ with $T = \Omega\left(\log(\theta\lambda_r^2(X)/p_1)\vee1\right)$ outputs an estimator $\widehat{U}$ satisfying
		\begin{equation}\label{ineq:noisy-mc-upper}
			\left\|\sin\Theta(\widehat{U}, U)\right\| \lesssim \frac{\max\left\{\sqrt{p_2(\theta+\theta^3p_1^2)\log(p_1)}, \theta p_1 \log^2(p_1)\right\}}{\theta^2\lambda_r^2(X)} \wedge 1
		\end{equation}
		with probability at least $1 - p_1^{-C}$.
	\end{Theorem}
	
	\begin{Remark}[Comparison with matrix completion]\label{rm:matrix-completion}{\rm
			Our result is related to a substantial body of literature on low-rank matrix completion. For example,  \cite{candes2009exact,candes2010power,recht2011simpler} analyzed the performance of nuclear norm minimization;
			\cite{mazumder2010spectral} introduced the spectral regularization algorithm for incomplete matrix learning and developed the software package \textit{SoftImpute}\footnote{\url{https://cran.r-project.org/web/packages/softImpute/index.html}}; \cite{keshavan2010matrix,keshavan2010matrixnoisy,keshavan2012efficient,jain2013low} analyzed the alternating gradient descent and spectral algorithm for matrix completion with/without noise; \cite{nadakuditi2014optshrink} developed \emph{OptShrink}, an algorithm for matrix estimation based on the optimal shrinkage of singular values and truncated SVD guided by random matrix theory; 			\cite{robin2019low} studied the low-rank model for count data with missing values; also see \cite{cai2018exploiting} for a recent survey of matrix completion. Different from the literature on matrix completion, our goal here is to estimate the singular subspace $U\in \mathbb{O}_{p_1, r}$ rather than the whole matrix $X\in \mathbb{R}^{p_1\times p_2}$. We apply HeteroPCA to impute the diagonal entries of $XX^\top$, not the missing entries in $X$ itself as in most of the aforementioned matrix completion literature.
			
			In addition, when the average amplitude of all entries in $X$ is a constant (i.e.~$\|X\|_F^2 \asymp p_1p_2$) and $X$ is well conditioned (i.e., $\lambda_1(X) \asymp \lambda_r(X)$), Theorem \ref{th:matrix-completion} implies that the HeteroPCA estimator is consistent as long as the expected sample size satisfies
			\begin{equation}\label{eq:omega-rate-consistent}
				\mathbb{E}|\Omega| \gg \max\left\{p_1^{1/3}p_2^{2/3}r^{2/3}\log^{1/3}(p_1), p_1r^2\log(p_1), p_1r\log(p_1)\log(p_1p_2)\right\}.
			\end{equation}
			In the classic literature on matrix completion \citep{keshavan2010matrix,recht2011simpler}, the sample size requirement is $|\Omega| \gtrsim (p_1+p_2)r \cdot {\rm polylog}(p)$. When $p_1 \gtrsim p_2$, these sample size requirements nearly match and coincide with existing lower bounds in the literature \cite[Theorem 1.7]{candes2010power}. When $p_1 \ll p_2$, \eqref{eq:omega-rate-consistent} requires much fewer samples than what is needed for matrix completion; in other words, HeteroPCA can consistently estimate the $p_1$-by-$r$ subspace $U_1$, even if most columns of $X$ are completely missing and estimating the whole $p_1$-by-$p_2$ matrix accurately is impossible. To our best knowledge, we are among the first to show such a result. }
	\end{Remark}
	
	\begin{Remark}[Time complexity]{\rm
			If the target matrix $X$ is $p_1$-by-$p_2$ and rank-$r$, the time complexity of HeteroPCA, regular SVD,  diagonal-deletion SVD, OptShrink \citep{nadakuditi2014optshrink}, and SoftImpute \citep{mazumder2010spectral} are $O(|\Omega|^2/p_2 + Tp_1^2r)$, $O(T(|\Omega|r + r^3))$, $O(|\Omega|^2/p_2 + Tp_1^2r)$, $O(T(|\Omega|r + r^3))$, and $O(T(|\Omega| + p_1p_2r))$, respectively. Here, $T$ denotes the number of iterations in each method.}
	\end{Remark}
	
	\begin{Remark}\label{rm:vaswani-compare}\rm
		Recently, \cite{vaswani2018pca,vaswani2020fast} studied PCA with sparse data-dependent noise and incomplete data. They proved that if the signal-to-noise ratio is strong enough, the uncorrelated noise is small enough, and the proportion of missing values is small enough, one can estimate the subspace accurately. Under the model setting of Theorem \ref{th:matrix-completion} and some regularity conditions, \cite[Corollary 3.7]{vaswani2018pca} can imply $\widetilde{U}_1 = \SVD_r(Y)$ satisfies
		\begin{equation}\label{ineq:MC-vaswani}
			\|\sin\Theta(\widetilde{U}_1, U_1)\| \lesssim \left(\sqrt{\frac{sbr}{p_1p_2}} + \frac{p_2}{\lambda_r^2(X)} + \sqrt{\frac{r^2s\log p_1}{p_2}} + \sqrt{\frac{r\log p_1}{\lambda_r^2(X)}}\right)\wedge 1.
		\end{equation}
		Here, $s, b$ are the maximum number of missing values in each row and in each column, respectively. To ensure \eqref{ineq:MC-vaswani} gives an nontrivial upper bound, one must have $(s/p_1)(b/p_2)\lesssim 1/r$. In contrast, Theorem \ref{th:matrix-completion} implies that HeteroPCA can consistently recover $U_1$ even if $s \approx p_1$ and $b\approx p_2$, i.e., only a smaller fracture of entries are observable, if the observable entries are uniform randomly selected from the target matrix.
	\end{Remark}
	
	\begin{Remark}\label{rm:matrix completion}\rm 
		PCA for heteroskedastic and incomplete data is another closely related problem. Suppose one observes incomplete i.i.d. samples $Y_1,\ldots, Y_n \in \mathbb{R}^p$ from the generalized spiked covariance model \eqref{eq:multivariate-normal} with missingness:
		\begin{equation*}
			\forall 1\leq i \leq p, 1\leq k \leq n,\quad  R_{ik} = \left\{\begin{array}{ll}
				1, & Y_{ik} \text{ is observable};\\
				0, & Y_{ik} \text{ is missing},
			\end{array}\right.
		\end{equation*}
		where $\{R_{ik}\}_{1\leq i \leq p, 1\leq k \leq n}$ are independent of $Y_1,\ldots, Y_n$. The goal is to estimate $U$. Many existing literature on PCA with incomplete data focused on regular SVD methods under the homoskedastic noisy setting (see, e.g., \cite{lounici2014high,cai2016minimax}), which are not directly suitable here. To estimate $U$ using HeteroPCA, we can evaluate the generalized sample covariance matrix,
		\begin{equation*}
			\begin{split}
				\widehat{\Sigma}^\ast = (\widehat{\sigma}_{ij}^\ast)_{1\leq i,j\leq p},\quad \text{with} \quad & \widehat{\sigma}_{ij}^\ast = \frac{\sum_{k=1}^n (Y_{ik} - \bar{Y}_i^\ast)(Y_{ik} - \bar{Y}_j^\ast)R_{ik}R_{jk}}{\sum_{k=1}^nR_{ik}R_{jk}}\\
				\text{and}\quad & \bar{Y}_i^\ast = \frac{\sum_{k=1}^n Y_{ik}R_{ik}}{\sum_{k=1}^n R_{ik}}.
			\end{split}
		\end{equation*}
		Then $U$ can be estimated by applying Algorithm \ref{al:heteroPCA} on $\widehat{\Sigma}^\ast$. A similar consistent upper bound result to Theorem \ref{th:matrix-completion} can be developed for this procedure. 
		
		In a more general scenario that noise $\varepsilon_k$ has non-diagonal covariance or depends linearly on the signal $X_k$, the readers are referred to  \cite{vaswani2018pca,vaswani2020fast} for a theory of the SVD estimator.
	\end{Remark}

	\section{Numerical Results}\label{sec:simu}
	
	In this section, we investigate the numerical performance of the proposed procedure. All simulation results are based on 1000 repeated independent experiments. The average and the standard deviation of estimation errors are respectively indicated by markers and error bars in each plot.

	\subsection{PCA under the generalized spiked covariance model}\label{sec:simu-PCA-generalized-spiked-covariance}
	
	We first consider PCA under the generalized spiked covariance model \eqref{eq:multivariate-normal}. Let $p =30, n \in [60, 600]$, and $r \in \{3, 5\}$. We generate a $p$-by-$r$ random matrix $U_0$ with i.i.d.~standard Gaussian entries, $w_1,\ldots, w_p\overset{iid}{\sim} \text{Unif}[0, 1]$, and $\sigma_1,\ldots, \sigma_p \overset{iid}{\sim}\text{Unif}[0,1]$. The purpose of generating uniform random vectors $w, \sigma$ is to introduce heteroskedasticity into observations. Then, we let $U = \text{QR}(\diag(w)\cdot U_0) \in \mathbb{O}_{p, r}$ and $\Sigma_0 = U\diag(1, \ldots, r)U^\top \in \mathbb{R}^{p\times p}$. We aim to recover $U$ based on i.i.d.~observations $\{Y_k = X_k+\varepsilon_k\}_{k=1}^n$, where $X_1,\ldots, X_n \overset{iid}{\sim}N(0, \Sigma_0), \varepsilon_1,\ldots, \varepsilon_n\overset{iid}{\sim} N(0, \diag(\sigma_1^2,\ldots, \sigma_n^2))$. We implement the proposed HeteroPCA, diagonal-deletion, and regular SVD approaches and plot the average estimation errors and standard deviation in $\sin\Theta$ distance. We also implement the classic factor analysis method \citep{thomson1939factorial,lawley1962factor}, \texttt{factanal} function in R \texttt{stats} package, and the Bayesian factor analysis method, \texttt{MCMCfactanal} function from R \texttt{MCMCpack} package  \citep{martin2011mcmcpack}. The simulation results are summarized in Figure \ref{fig:pca-1}.
	\begin{figure}[!h]\centering
		\includegraphics[height=4.5cm]{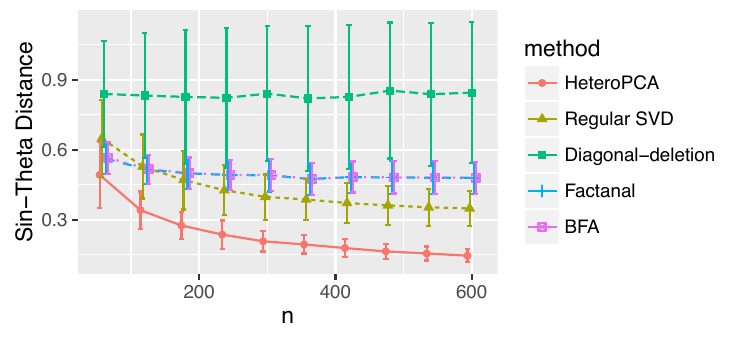}\\
		\includegraphics[height=4.5cm]{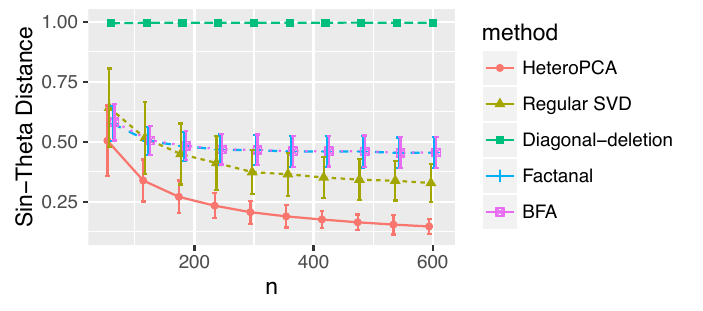}\\
		\caption{Average $\sin\Theta$ loss versus sample size $n$ under the generalized spiked covariance model (Section \ref{sec:simu-PCA-generalized-spiked-covariance}). Upper panel: $r= 3$; lower panel: $r=5$}
		\label{fig:pca-1}
	\end{figure}
	It can be seen that the proposed HeteroPCA estimator significantly outperforms other methods; the regular SVD yields larger estimation error; and the diagonal-deletion estimator performs unstably across different settings. This matches the theoretical findings in Section \ref{sec:pca}.
	
	Next we study how the degree of heteroskedasticity affects the performance. Let
	$$v_1,\ldots, v_p\overset{iid}{\sim}\text{Unif}[0, 1], \quad \sigma_k^2 = \frac{0.1\cdot p \cdot v_k^\alpha}{\sum_{i=1}^p v_i^\alpha}, \quad k=1,\ldots, p.$$ 
	In such case, $\sigsum^2 = \sigma_1^2+\cdots +\sigma_p^2$ always equals $0.1p$ and $\alpha$ characterizes the degree of heteroskedasticity: the larger $\alpha$ results in a more imbalanced distribution of $(\sigma_1,\ldots, \sigma_p)$; if $\alpha = 0$, $\sigma_1=\cdots = \sigma_p$ and the setting becomes homoskedastic. Now we generate $U, \Sigma_0$ and $\{Y_k, X_k, \varepsilon_k\}_{k=1}^n$ in the same way as the previous setting. We only compare HeteroPCA with regular SVD and diagonal-deletion estimator since it takes a too long time to run factor analysis methods in this setting. The average estimation errors for $U$ are plotted in Figure \ref{fig:pca-2}. The results again suggest that the performance of diagonal-deletion estimator is unstable across different settings. When $\alpha=0$, i.e., the noise is homoskedastic, the performance of HeteroPCA and regular SVD are comparable; but as $\alpha$ increases, the estimation error of HeteroPCA grows significantly slower than that of the regular SVD, which is consistent with the theoretical results in Theorem \ref{th:heterogeneous-PCA}.
	\begin{figure}[!h]
		\centering
		\subfigure[$p=50, n =30, r=5$]{\includegraphics[height=4.5cm]{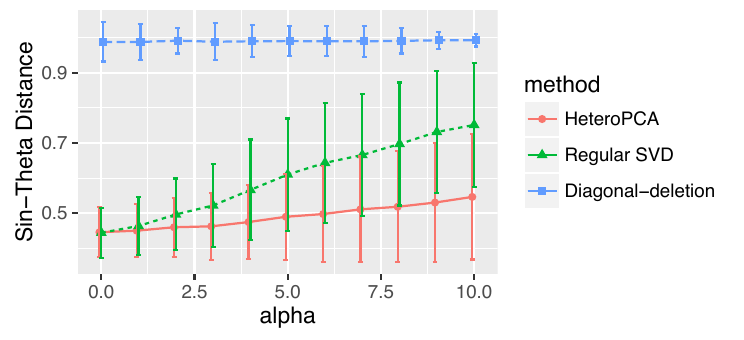}}
		\subfigure[$p=200, n=400, r=5$]{\includegraphics[height=4.5cm]{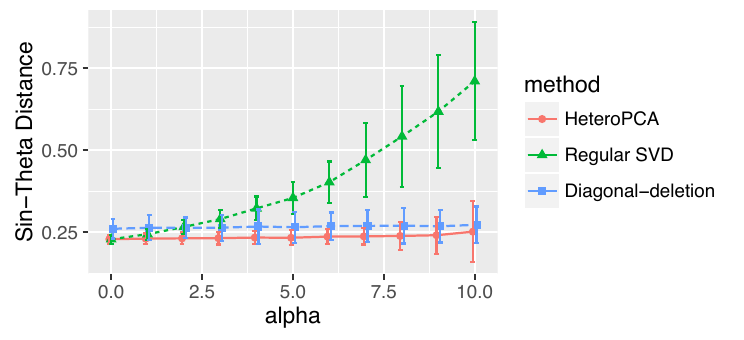}}
		\caption{Average $\sin\Theta$ loss versus heteroskedastic level $\alpha$ under the generalized spiked covariance model (Section \ref{sec:simu-PCA-generalized-spiked-covariance})}
		\label{fig:pca-2}
	\end{figure}
	
	\subsection{SVD under heteroskedastic noise}\label{sec:simu-SVD}
	
	Next, we consider the problem of SVD under heteroskedastic noise discussed in Section \ref{sec:heteroskedastic perturbation}. Let $U_0 \in \mathbb{R}^{p_1\times r}$ and $V_0\in \mathbb{R}^{p_2\times r}$ be i.i.d.~Gaussian ensembles for $(p_1, p_2) = (50, 200), (200, 1000)$ and $r=3$. To introduce heteroskedasticity, we also randomly draw $w, v_1\in \mathbb{R}^{p_1}$, and $v_2\in \mathbb{R}^{p_2}$ with i.i.d.~${\rm Unif}[0, 1]$ entries. Then we evaluate $U = \text{QR}\left(U_0 \cdot \diag(w)^4\right)$, $V = \text{QR}\left(V_0\right)$, and construct the signal matrix $X = (p_1p_2)^{1/4}\cdot U\diag(1,\ldots, r)V^\top$. The noise matrix is drawn as $E_{ij} \overset{ind}{\sim}N(0, \sigma_0^2\cdot \sigma_{ij}^2)$, where $\sigma_{ij} = (v_1)_i^4\cdot (v_2)_j^4$, $\sigma_0$ varies from 0 to 2, $1\leq i \leq p_1$, and $1\leq j \leq p_2$. Based on the $p_1$-by-$p_2$ observation $Y = X+E$, we implement HeteroPCA with input of $YY^\top$, regular-SVD, diagonal-deletion, and \emph{OptShrink}\footnote{Software package available at \url{https://web.eecs.umich.edu/~rajnrao/optshrink/}} (\cite{nadakuditi2014optshrink}, an algorithm for matrix estimation based on the optimal shrinkage of singular values and truncated SVD guided by random matrix theory) to evaluate $\widehat{U}, \widehat{V}$. For each of the estimators $\widehat{U}$ and $\widehat{V}$, we also estimate $X$ by $\widehat{X}=\widehat{U}\widehat{U}^\top Y\widehat{V}\widehat{V}^\top$. The average $\sin\Theta$ norm errors of $\widehat{U}, \widehat{V}$ and the average Frobenius norm error of $\widehat{X}$ are presented in Figure \ref{fig:svd}. We can see the proposed HeteroPCA outperforms other methods in all estimations for $U, V$, and $X$, and the advantage of HeteroPCA is more significant when the noise level increases.
	\begin{figure}
		\centering
		\subfigure[$p_1 = 50, p_2 = 200$ ]{\includegraphics[height=9cm]{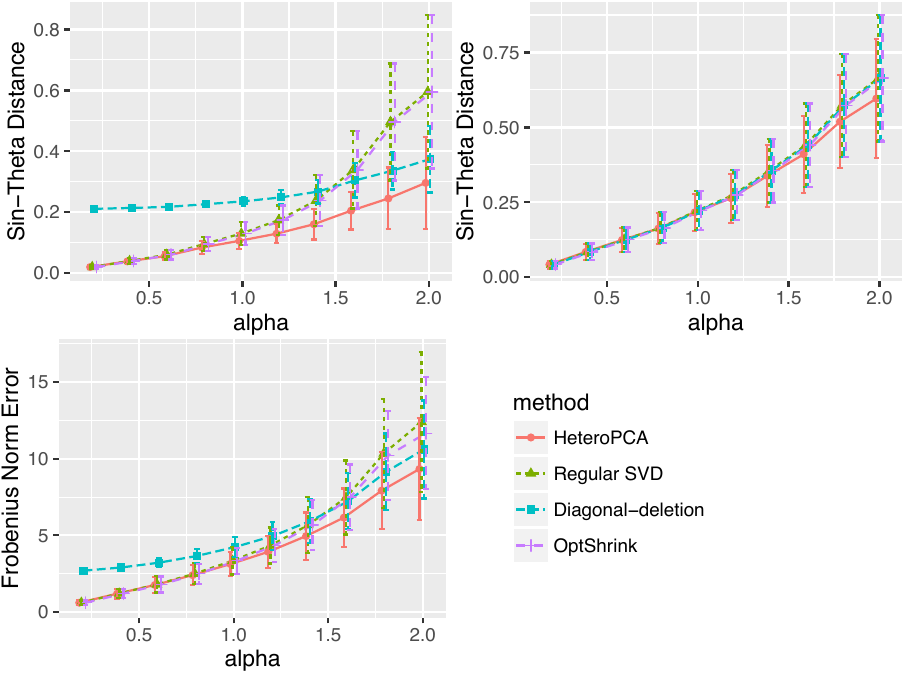}}
		\subfigure[$p_1 = 200, p_2 = 1000$]{\includegraphics[height=9cm]{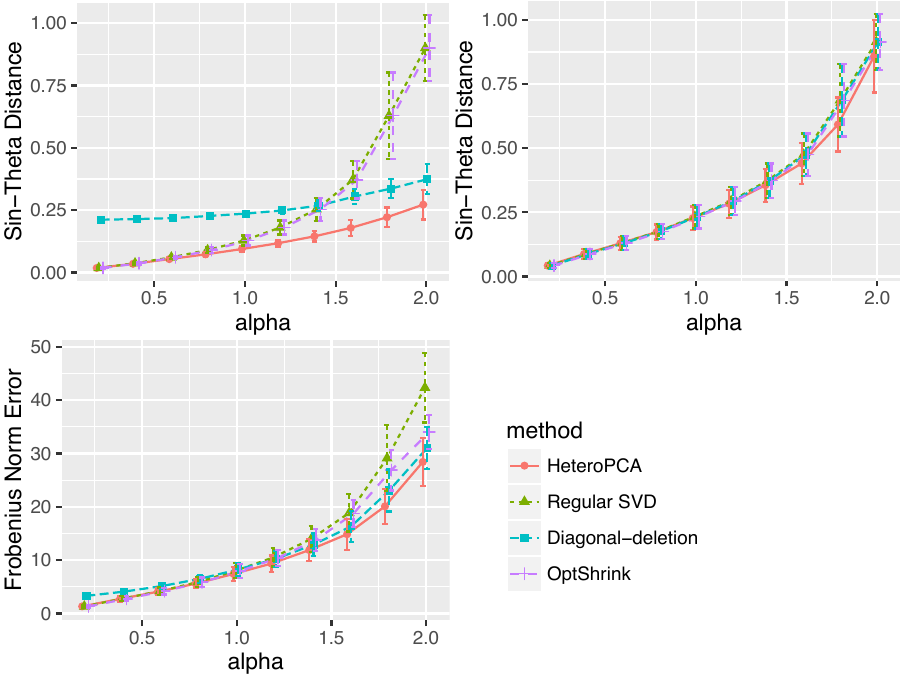}}
		\caption{Estimation errors of $\widehat{U}$ (top left), $\widehat{V}$ (top right), and $\widehat{X}$ (bottom left) in SVD under heteroskedastic noise (Section \ref{sec:simu-SVD})}
		\label{fig:svd}
	\end{figure}

	\subsection{Poisson PCA}\label{sec:simu-Poisson}
	
	We generate $U_0\in \mathbb{R}^{p_1\times r}$ and $V_0\in \mathbb{R}^{p_2\times r}$ with i.i.d. standard normal entries for $(p_1, p_2, r) = (50, 500, 3)$ or $(200, 1000, 3)$. Similarly to previous settings, we introduce heteroskedasticity by generating a vector $w\in \mathbb{R}^{p_1}$ with i.i.d. Unif[0, 1] entries. Let $U = |U_0\cdot\diag(w)^4| \in \mathbb{R}^{p_1\times r}, V = |V_0| \in \mathbb{R}^{p_2\times r}$, $X = \lambda U\diag(1,\ldots, r) V^\top \in \mathbb{R}^{p_1\times p_2}$, and $Y_{ij} \sim \text{Poisson}(X_{ij})$ independently. Here, $\lambda>0$ measures the signal strength. The performance of HeteroPCA, regular SVD, diagonal-deletion, and OptShrink on estimation of left singular subspaces are provided in Figure \ref{fig:poisson}. These plots again illustrate the merit of the proposed HeteroPCA method.
	\begin{figure}[!h]
		\centering
		\subfigure[$p_1 = 50, p_2 = 500$]{\includegraphics[height=4.5cm]{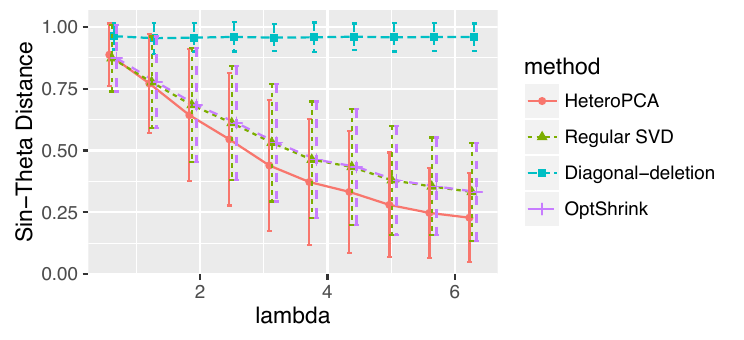}}
		\subfigure[$p_1=200, p_2 = 1000$]{\includegraphics[height=4.5cm]{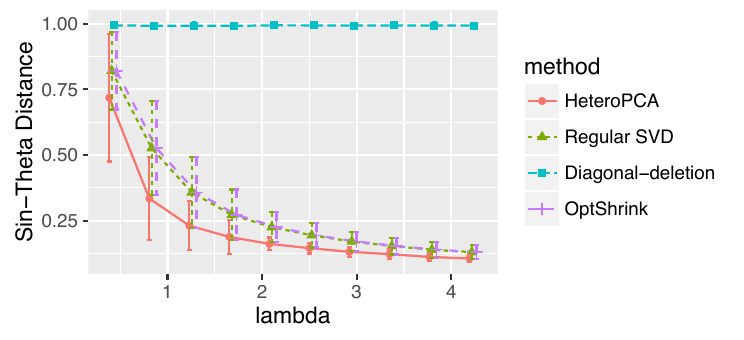}}
		\caption{Estimation errors for a ranging value of signal strength $\lambda$ under the Poisson PCA model (Section \ref{sec:simu-Poisson})}
		\label{fig:poisson}
	\end{figure} 
	
	\subsection{SVD based on heteroskedastic and incomplete data}\label{sec:simu-incomplete}
	
	Finally, in the following experiment we study SVD based on heteroskedastic and incomplete data in the setting of \prettyref{sec:matrix-completion}. Generate $Y, X, Z\in \mathbb{R}^{p_1\times p_2}$ in the same way as the previous heteroskedastic SVD setting with $p_1=50, 100$, $r = 3, 5$, $\sigma_0 = .2$, and $p_2$ ranging from 800 to 3200. 
	Each entry of $Y$ is observed independently with probability $\theta=0.1$.
	We aim to estimate $U$ based on $\{Y_{ij}: (i,j)\in\Omega\}$. 
	In addition to HeteroPCA, regular SVD, diagonal-deletion SVD, and OptShrink, we also apply the nuclear norm minimization via \emph{Soft-Impute} package (\cite{mazumder2010spectral}, also see Remark \ref{rm:matrix completion})
	$$\widehat{X}_\ast = \argmin_{\widehat{X}\in \mathbb{R}^{p_1\times p_2}}\sum_{(i,j) \in\Omega }(\widetilde{Y}_{ij} - \widehat{X}_{ij})^2 + \nu \|\widehat{X}\|_\ast, \quad \widehat{U} = \SVD_r(\widehat{X}).$$
	To avoid the cumbersome issue of parameter $\nu$ selection, we evaluate the above nuclear norm minimization estimator for a grid of values of $\nu$, then record the outcome with the minimum $\sin\Theta$ distance error $\|\sin\Theta(\widehat{U}, U)\|$. From the results plotted in Figure \ref{fig:mc}, we can see that HeteroPCA significantly outperforms all other methods when $p_1\ll p_2$, which matches the discussion in Remark \ref{rm:matrix-completion}. 
	
	\begin{figure}[!h]
		\centering
		\includegraphics[height=4.5cm]{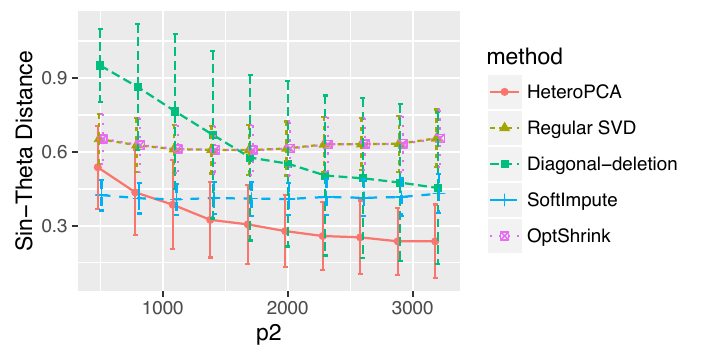}\\
		\includegraphics[height=4.5cm]{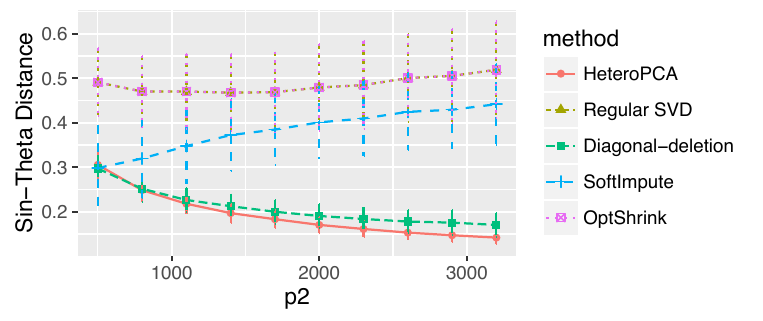}	
		\caption{Average $\sin\Theta$ distance error for SVD based on heteroskedastic and incomplete data (Section \ref{sec:simu-incomplete}). Here, $p_1 = 50, r =5, \theta = .2$ (Upper Panel) and $p_1 = 100, r=3, \theta = .2$ (Lower Panel); $p_2$ varies from 800 to 3200}
		\label{fig:mc}
	\end{figure}
	
	\section{Discussion}\label{sec:discussion}
	
	We consider PCA in the presence of heteroskedastic noise in this paper. To alleviate the significant bias incurred on diagonal entries of the Gram matrix due to heteroskedastic noise, we introduced a new procedure named HeteroPCA that adaptively imputes diagonal entries to remove the bias. The proposed procedure achieves optimal rate of convergence in a range of settings. In addition, we discuss the applications of the proposed algorithm to heteroskedastic low-rank matrix denoising, Poisson PCA, and SVD based on heteroskedastic and incomplete data. 
	
	The proposed HeteroPCA  procedure can also be applied to many other problems where the noise is heteroskedastic. First, {\it exponential family PCA} is a commonly used technique for dimension reduction on non-real-valued datasets \citep{collins2002generalization,mohamed2009bayesian}. As discussed in the introduction, the exponential family distributions, e.g., exponential, binomial, and negative binomial, may be highly heteroskedastic. As in the case of Poisson PCA considered in Section \ref{sec:Poisson PCA}, the proposed HeteroPCA algorithm can be applied to exponential family PCA.
	
	In addition, community detection in social networks has attracted significant attention in the recent literature \citep{fortunato2010community}. Although most of existing results focused on unipartite graphs, bipartite graphs, i.e., all edges are between two groups of nodes, often appear in practice \citep{melamed2014community,florescu2016spectral,zhou2020optimal}. The proposed HeteroPCA can also be applied to \emph{community detection for bipartite stochastic block model}. Similarly to the analysis for heteroskedastic low-rank matrix denoising in \prettyref{sec:heteroskedastic perturbation}, HeteroPCA can be shown to have advantages over other baseline methods. 
	
	The proposed framework is also applicable to solve the {\it heteroskedastic tensor SVD} problem, which aims to recover the low-rank structure from the tensorial observation corrupted by heteroskedastic noise. Suppose one observes $\Y = \X + \Z \in \mathbb{R}^{p_1\times p_2\times p_3}$, where $\X$ is a Tucker low-rank signal tensor and $\Z$ is the noise tensor with independent and zero-mean entries. If $\Z$ is homoskedastic, the higher-order orthogonal iteration (HOOI) \citep{de2000best} was shown to achieve the optimal performance for recovering $\X$ \citep{zhang2018tensor}. If $\Z$ is heteroskedastic, we can apply HeteroPCA instead of the regular SVD to obtain a better initialization for HOOI. Similarly to the argument in this article, we are able to show that this modified HOOI yields more stable and accurate estimates than the regular HOOI.
	
	\emph{Canonical correlation analysis} (CCA) is one of the most important tools in multivariate analysis for exploring the relationship between two sets of vector samples \citep{hotelling1936relations}. In the standard procedure of CCA, the core step is a regular SVD on the adjusted cross-covariance matrix between samples. When the observations contain heteroskedastic noise, one can replace the regular SVD procedure by HeteroPCA to achieve better performance.

	\section{Proofs}\label{sec:proof}
	
	In this section, we prove the main results, namely, Theorems \ref{th:heterogeneous-PCA} and  \ref{th:diagonal-less}. For reasons of space, the other proofs are given in the supplementary materials \citep{zhang2018supplement}.

	\subsection{Proofs for Heteroskedastic PCA}
	\begin{proof}[Proof of Theorem \ref{th:heterogeneous-PCA}]
		First, we introduce
		$$E = [\varepsilon_1,\ldots, \varepsilon_n]\in \mathbb{R}^{p\times n}, \quad \gamma_k = \Lambda^{1/2}U^\top(X_k-\mu)\in \mathbb{R}^r,\quad \Gamma = [\gamma_1,\ldots, \gamma_n] \in \mathbb{R}^{r\times n}.$$
		Then the observations can be written as
		$$Y_k = X_k + \varepsilon_k = \mu + U\Lambda^{1/2}\gamma_k  + \varepsilon_k, \quad \text{or}\quad \quad Y = \mu 1_n^\top +  U\Lambda^{1/2} \Gamma + E,$$
		where $\mu\in \mathbb{R}^p$ is a fixed vector, $\mathbb{E}\gamma_k = 0, \Cov(\gamma_k) = I$, $E$ has independent entries, and $\Gamma$ has independent columns. We also denote $\bar{X}\in \mathbb{R}^p, \bar{E}\in \mathbb{R}^p, \bar{\Gamma}\in \mathbb{R}^r$ as the averages of all columns of $X, E$, and $\Gamma$, respectively. Since $\widehat{\Sigma}$ is invariant after any translation on $Y$, we can assume $\mu = 0$ without loss of generality. The rest of the proof is divided into three steps for the sake of presentation.
		
		\begin{enumerate}[leftmargin=*]
			\item[Step 1] We define $\widehat{\Sigma}_X = (XX^\top - n\bar{X}\bar{X}^\top)/(n-1)$ as the signal sample covariance. The aim of this step is to develop a concentration inequality for $\widehat{\Sigma} - \Sigma_X$. To this end, we consider the following decomposition of $n(\widehat{\Sigma} - \widehat{\Sigma}_X)$,
			\begin{equation}\label{eq:hetero-PCA-upper-1}
				\begin{split}
					& (n-1)(\widehat{\Sigma} - \widehat{\Sigma}_X)  = (n-1)\widehat{\Sigma} - (XX^\top-n\bar{X}\bar{X}^\top)\\
					= & YY^\top - n\bar{Y}\bar{Y}^\top - (XX^\top-n\bar{X}\bar{X}) \\
					= & (X+E)(X+E)^\top - (XX^\top-n\bar{X}\bar{X}^\top) - n \left(\bar{X}\bar{X}^\top + \bar{X}\bar{E}^\top + \bar{E}\bar{X}^\top + \bar{E}\bar{E}^\top\right)\\
					= & XE^\top + EX^\top + EE^\top - n\left(\bar{X}\bar{E}^\top + \bar{E}\bar{X}^\top + \bar{E}\bar{E}^\top\right).
				\end{split}
			\end{equation}
			We analyze each term of \eqref{eq:hetero-PCA-upper-1} separately as follows. Since $E$ has independent entries and $\Var(E_{ij}) = \sigma_i^2$,
			the rowwise structured heteroskedastic concentration inequality \cite[Theorem 6]{cai2020non} implies
			\begin{equation}\label{ineq:hetero-PCA-upper-2}
				\begin{split}
					\mathbb{E}\left\|EE^\top - \mathbb{E} EE^\top \right\|  \lesssim & \sqrt{n}\sigsum\sigmax + \sigsum^2.
				\end{split}
			\end{equation} 
			Since $X$ is deterministic, $E$ is random, and $\mathbb{E}E=0$, we have $\mathbb{E}EX^\top =0$. By Lemma \ref{lm:orthogonal-projection} in the supplementary materials, 
			\begin{equation}\label{ineq:hetero-PCA-upper-3}
				\begin{split}
					& \mathbb{E}_E\left(\left\|EX^\top - \mathbb{E}EX^\top\right\|\Big| X\right) = \mathbb{E}_E\left(\left\|EX^\top \right\|\Big| X\right) \\
					\lesssim & \|X\|\left(\sigma_C + r^{1/4}\sigmax\sigma_C + \sqrt{r}\sigmax\right) \overset{\text{Cauchy-Schwarz}}{\lesssim} \|X\|\left(\sigsum + \sqrt{r}\sigmax\right).
				\end{split}
			\end{equation}
			Since $\mathbb{E}\|\bar{E}\|_2^2 = \sum_{i=1}^p \mathbb{E} (\bar{E})_i^2 = \sum_{i=1}^p \sigma_i^2/n = \sigsum^2/n$, we have
			\begin{equation}\label{ineq:hetero-PCA-upper-4}
				\begin{split}
					& \mathbb{E}_E \left(n\left\|\bar{X}\bar{E}^\top + \bar{E}\bar{X}^\top + \bar{E}\bar{E}^\top\right\| \right)\\
					\leq & \mathbb{E}_E n\|\bar{X}\bar{E}^\top\| + \mathbb{E}_E n\|\bar{E}\bar{X}^\top\| + \mathbb{E}_E n\|\bar{E}\bar{E}^\top\|\\
					\leq &  \mathbb{E}_E 2n\|\bar{X}\|_2 \|\bar{E}\|_2 + \mathbb{E} n\|\bar{E}\|_2^2 \\
					\leq & 2n\|\bar{X}\|_2 \cdot (\mathbb{E}\|\bar{E}\|_2^2)^{1/2} + \mathbb{E}n\|\bar{E}\|_2^2 \leq 2n^{1/2}\sigsum\|\bar{X}\|_2 + \sigsum^2.
				\end{split}
			\end{equation}
			Combining \eqref{ineq:hetero-PCA-upper-2}, \eqref{ineq:hetero-PCA-upper-3}, and \eqref{ineq:hetero-PCA-upper-4}, we have
			\begin{equation*}
				\begin{split}
					& \mathbb{E}_E\left\|(n-1)(\widehat{\Sigma}-\widehat{\Sigma}_X) - \mathbb{E}EE^\top \right\| \\
					\lesssim & \sqrt{n}\sigsum\sigmax + \sigsum^2 + \|X\|(\sigsum + \sqrt{r}\sigmax) + n^{1/2}\|\bar{X}\|_2\sigsum. 
				\end{split}
			\end{equation*}
			Noting that $\mathbb{E}EE^\top = n\diag(\sigma_1^2,\ldots, \sigma_p^2)$ is diagonal and $\Delta(\cdot)$ is the operator that sets all diagonal entries to zero, we further have
			\begin{equation*}
				\begin{split}
					& \mathbb{E}_E\left\|\Delta\left((n-1)(\widehat{\Sigma} - \widehat{\Sigma}_X)\right) \right\| =  \mathbb{E}_E\left\|\Delta\left((n-1)(\widehat{\Sigma} - \widehat{\Sigma}_X) - \mathbb{E}EE^\top\right)\right\|\\ \overset{\text{Lemma \ref{lm:diagonal-less-spectral-norm}}}{\leq}&  2\mathbb{E}_E\left\|(n-1)\left(\widehat{\Sigma} - \widehat{\Sigma}_X\right) - \mathbb{E}EE^\top\right\|\\
					\lesssim & \sqrt{n}\sigsum\sigmax + \sigsum^2 + \|X\|(\sigsum + \sqrt{r}\sigmax) + n^{1/2}\|\bar{X}\|_2\sigsum.
				\end{split}
			\end{equation*}
			\item[Step 2] Next, we study the expectation of the target function with respect to $X$. We specifically need to study $\lambda_r(n\widehat{\Sigma}_X)$, $\|X\|$, and $\|\bar{X}\|_2$. Since $\Gamma\in \mathbb{R}^{r\times n}$ has independent columns and each column is isotropic sub-Gaussian distributed, based on the random matrix theory \cite[Corollary 5.35]{vershynin2010introduction}, 
			$$\bbP\left(\sqrt{n} + C\sqrt{r} + t \geq \|\Gamma\| \geq \lambda_{r}(\Gamma) \geq \sqrt{n} - C\sqrt{r} - t\right) \leq \exp(-Ct^2/2).$$
			In addition, $\sqrt{n}\bar{\Gamma} \in \mathbb{R}^r$ is a sub-Gaussian vector with the identity covariance matrix. 
			By the Bernstein-type concentration inequality \citep[Proposition 5.16]{vershynin2010introduction}, 
			$$\bbP\left(\|\sqrt{n}\bar{\Gamma}\|_2^2 \geq r + C\sqrt{rx} + Cx\right) \leq C\exp(-cx).$$
			If $n \geq Cr$ for some large constant $C>0$, by setting $t = c\sqrt{n}$ and $x = cn$ in the previous two inequalities, we have
			\begin{equation}\label{ineq:hetero-PCA-upper-6}
				2\sqrt{n} \geq \|\Gamma\| \geq \lambda_r(\Gamma) \geq \sqrt{n}/2, \quad \text{and}\quad \|\sqrt{n}\bar{\Gamma}\|_2 \leq \sqrt{n}/3
			\end{equation}
			with probability at least $1 - C\exp(-cn)$. When \eqref{ineq:hetero-PCA-upper-6} holds,
			\begin{equation*}
				\begin{split}
					\lambda_r(n\widehat{\Sigma}_X) = & \lambda_r\left(n(XX^\top - n\bar{X}\bar{X}^\top)\right) = \lambda_r\left(nU\Lambda^{1/2}(\Gamma\Gamma^\top - n\bar{\Gamma}\bar{\Gamma}^\top)\Lambda^{1/2}U^\top\right) \\
					\geq & \lambda_r(\Lambda) \cdot \lambda_r\left(\Gamma\Gamma^\top - n\bar{\Gamma}\bar{\Gamma}^\top\right) \geq \lambda_r(\Lambda)\left(\lambda_r^2(\Gamma) - \|\sqrt{n}\bar{\Gamma}\|_2^2\right)\\
					\overset{\eqref{ineq:hetero-PCA-upper-4}}{\geq} & \lambda_r(\Lambda)\left(n/4 - n/9\right) \gtrsim n\lambda_r(\Lambda);
				\end{split}
			\end{equation*}
			\begin{equation}
				\|X\| \leq \|U\Lambda^{1/2}\Gamma\| \leq \|\Lambda^{1/2}\|\cdot \|\Gamma\| \overset{\eqref{ineq:hetero-PCA-upper-6}}{\leq} 2\sqrt{n}\|\Lambda^{1/2}\| \lesssim  \sqrt{n}\lambda_r^{1/2}(\Lambda),
			\end{equation}
			where the last inequality is due to the assumption that $\|\Lambda\|/\lambda_r(\Lambda) \leq C$ for some constant $C$.
			\begin{equation*}
				\|\bar{X}\|_2 = \|U\Lambda^{1/2} \bar{\Gamma}\|_2 \leq \|\Lambda^{1/2}\|\cdot\|\bar{\Gamma}\|_2 \lesssim \lambda_r^{1/2}(\Lambda).
			\end{equation*}
			Combining the previous three inequalities, we know if \eqref{ineq:hetero-PCA-upper-6} holds,
			\begin{equation}\label{ineq:hetero-PCA-upper-7}
				\begin{split}
					& \frac{\mathbb{E}_E\|\Delta((n-1)(\widehat{\Sigma}-\widehat{\Sigma}_X))\|}{\lambda_r((n-1)\widehat{\Sigma}_X)}\wedge 1 \\ \lesssim & \frac{\sqrt{n}\sigsum \sigmax + \sigsum^2 + (n\lambda_r(\Lambda))^{1/2}(\sigsum + \sqrt{r}\sigmax) + (n\lambda_r(\Lambda))^{1/2}(\Lambda)\sigsum}{n\lambda_r(\Lambda)}\wedge 1\\
					\lesssim & \left(\frac{\sigsum + \sqrt{r}\sigmax}{(n\lambda_r(\Lambda))^{1/2}} + \frac{\sqrt{n}\sigsum\sigmax + \sigsum^2}{n\lambda_r(\Lambda)} \right)\wedge 1\\
					\lesssim & \left(\frac{\sigsum + \sqrt{r}\sigmax}{(n\lambda_r(\Lambda))^{1/2}} + \frac{\sigsum\sigmax}{n^{1/2}\lambda_r(\Lambda)} \right)\wedge 1.
				\end{split}
			\end{equation}
			Here, the last ``$\lesssim$" is due to $\sigsum^2/(n\lambda_r(\Lambda))\wedge 1 \leq \sigsum/(n\lambda_r(\Lambda))^{1/2} \wedge 1$. 
			\item[Step 3] Finally, since $\rank(\widehat{\Sigma}_X) \leq r$, the eigenvectors of $\widehat{\Sigma}_X$ are $U$, and $U$ satisfies the incoherence condition: $I(U) \leq c_Ip/r$, the robust $\sin\Theta$ Theorem (Theorem \ref{th:diagonal-less}) for $T = \Omega\left(\log \left(\frac{n\lambda_r(\Lambda)}{\sigsum^2}\right) \vee 1\right)$ yields
			\begin{equation}\label{ineq:hetero-PCA-upper-5}
				\begin{split}
					& \bbE_E\left\|\sin\Theta(\widehat{U}, U)\right\| \lesssim \bbE_E\left(\frac{\|\Delta((n-1)(\widehat{\Sigma}-\widehat{\Sigma}_X))\|}{\lambda_r((n-1)\widehat{\Sigma}_X)} + 2^{-T}  \right)\wedge 1\\
					\lesssim & \left(\frac{\sqrt{n}\sigsum\sigmax + \sigsum^2 + \|X\|(\sigsum + \sqrt{r}\sigmax) + n^{1/2}\|\bar{X}\|_2\sigsum}{\lambda_r\left((n-1)\widehat{\Sigma}_X\right)} + \frac{\sigsum^2}{n\lambda_r(\Lambda)}\right)\wedge 1.
				\end{split}
			\end{equation}

			\begin{equation*}
				\begin{split}
					\mathbb{E}\|\sin\Theta(\widehat{U}, U)\| = &  \mathbb{E}\|\sin\Theta(\widehat{U}, U)\| 1_{\{\text{\eqref{ineq:hetero-PCA-upper-6} holds}\}} + \mathbb{E}\|\sin\Theta(\widehat{U}, U)\| 1_{\{\text{\eqref{ineq:hetero-PCA-upper-6} does not hold}\}}\\
					\overset{\eqref{ineq:hetero-PCA-upper-6}}{\lesssim} & \left(\frac{\sigsum + \sqrt{r}\sigmax}{(n\lambda_r(\Lambda))^{1/2}} + \frac{\sigsum\sigmax}{n^{1/2}\lambda_r(\Lambda)} \right)\wedge 1 + \bbP\left(\text{\eqref{ineq:hetero-PCA-upper-6} does not hold}\right)\\
					\lesssim & \left(\frac{\sigsum + \sqrt{r}\sigmax}{(n\lambda_r(\Lambda))^{1/2}} + \frac{\sigsum\sigmax}{n^{1/2}\lambda_r(\Lambda)} \right)\wedge 1 + C\exp(-cn)\\
					\lesssim & \left(\frac{\sigsum + \sqrt{r}\sigmax}{(n\lambda_r(\Lambda))^{1/2}} + \frac{\sigsum\sigmax}{n^{1/2}\lambda_r(\Lambda)} \right)\wedge 1.
				\end{split}
			\end{equation*}
		\end{enumerate}
		The last inequality is due to the assumption that $\lambda_r(\Lambda)\geq c\exp(-cn)$. Therefore, we have finished the proof of this theorem.
	\end{proof}
	
	\subsection{Proof of \prettyref{th:diagonal-less}}\label{sec:proof-robust-sin-theta}
	In this subsection we prove a more general version of \prettyref{th:diagonal-less}, where the corrupted entries lie in a known set $\mathcal{G} \subset [p]\times [p]$ which need not be the diagonal. Recall the model \prettyref{eq:model-1}, where we observe a symmetric $p\times p$ matrix $N = M +Z$, where $M$ is a rank-$r$ matrix of interest and $Z$ is the perturbation. Our goal is to estimate $U\in\mathbb{O}_{p, r}$, consisting of the eigenvectors of $M$. 
	Extending the ideas of \prettyref{al:heteroPCA} for HeteroPCA, \prettyref{al:heteroPCA-general} provides a robust estimate of $U$ which iteratively impute the values in the corrupted entries in $\calG$. 
	In the special case where $\mathcal{G}$ is the diagonal, i.e., $\mathcal{G} = \{(i,i): 1\leq i \leq p\}$, Algorithm \ref{al:heteroPCA-general} reduces to Algorithm \ref{al:heteroPCA}.
	\begin{algorithm}
		\caption{Generalized HeteroPCA}
		\begin{algorithmic}[1]
			\State Input: matrix $\widehat{\Sigma}$, rank $r$, number of iterations $T$, corruption subset $\mathcal{G}\subseteq [p]\times [p]$.
			\State Set $N^{(0)} = \Gamma(N)$.
			\For{$t=1,\ldots, T$}
			\State Calculate SVD: $N^{(t)} = \sum_i \lambda_i^{(t)}u_i^{(t)}(v_i^{(t)})^\top$, where $\lambda_1^{(t)}\geq \lambda_2^{(t)}\cdots \geq 0$. 
			\State Let $\widetilde{N}^{(t)} = \sum_{i=1}^r\lambda_{i}^{(t)} u_{i}^{(t)}(v_{i}^{(t)})^\top$.
			\State Update corrupted entries: $N^{(t+1)} = G(\widetilde{N}^{(t)}) + \Gamma(\widehat{\Sigma})$.
			\EndFor
			\State Output: $\widehat{U} = U^{(T)} = [u_1^{(T)} ~ \cdots ~ u_r^{(T)}]$.
		\end{algorithmic}
		\label{al:heteroPCA-general}
	\end{algorithm}
	
	Next we give a performance guarantee for \prettyref{al:heteroPCA-general}.
	For any $H\in \mathbb{R}^{p\times p}$, let $G(H)$ be the matrix $H$ with all entries but those in $\mathcal{G}$ set to zero and $\Gamma(H) = H - G(H)$. 
	Define
	\begin{equation}
		\eta = \max_{\substack{H\in \mathbb{R}^{m\times m}, \rank(H)\leq 2r}} \|G(H)\|/\|H\|,
		\label{eq:eta}
	\end{equation}
	which essentially measures the maximum perturbations due to the entries in $\mathcal{G}$ on the singular subspace. We also assume that the set of corrupted entries	$\mathcal{G}$ 
	is $b$-sparse in the sense that
	$$\max_i|\left\{j: (i, j)\in \mathcal{G} \right\}| \vee  \max_j |\left\{i: (i, j)\in \mathcal{G}\right\}| \leq b,$$ 
	i.e., the number of corrupted entries in each row and each column is at most $b$. To overcome the ``spiky" issue discussed in Remark \ref{rm:spiky}, we again assume the incoherence condition \eqref{ineq:incoherence-constant-general-d-k}. We have the following theoretical results for Algorithm \ref{al:heteroPCA-general}.
	\begin{Theorem}[General robust $\sin\Theta$ theorem]\label{th:robust-Davis-Kahan}
		Assume $\mathcal{G}\in [p]\times [p]$ is $b$-sparse. Suppose one observes the symmetric matrix $N = M+Z$, where $\rank(M) = r$, $Z$ is any symmetric perturbation, and the eigenvectors of $M$ are $U\in \mathbb{O}_{p, r}$. Let $\widehat{U}^{(t)}$ be the intermediate matrix in Algorithm \ref{al:heteroPCA} with $t$ iterations. There exists a constant $c>0$ such that if 
		the incoherence condition	
		\begin{equation}
			\label{ineq:incoherence-constant-general-d-k}
			\frac{I(U) \|M\|}{\lambda_r(M)} 
			\leq \frac{c p}{\eta br(b\wedge r)}
		\end{equation}
		is satisfied and $\eta\|\Gamma(Z)\|\leq c \lambda_r(M)$, then 
		\begin{equation}	\label{eq:robust-Davis-Kahan-intermediate}
			\left\|\sin\Theta(\widehat{U}^{(t)}, U)\right\| \leq 4\|\Gamma(Z)\| /\lambda_r(M) + 2^{-(t+3)}/\eta.
		\end{equation}
		Here, $\eta$ is defined in \eqref{eq:eta}. In particular, if $T = \Omega(\log\frac{\lambda_r(M)}{\eta\|\Gamma(Z)\|}\vee 1)$, the final outcome $\widehat{U}$ of Algorithm \ref{al:heteroPCA} with corrupted index set $\mathcal{G}$ satisfies
		\begin{equation}
			\left\|\sin\Theta(\widehat{U}, U)\right\| \lesssim \frac{\|\Gamma(Z)\|}{\lambda_r(M)}\wedge 1.
			\label{eq:robust-Davis-Kahan}
		\end{equation}	
	\end{Theorem}
	\begin{Remark}\rm 
		Though calculating the exact value of $\eta$ can be difficult in general, Lemma \ref{lm:diagonal-less-spectral-norm} in the supplement shows $\eta \leq \sqrt{b\wedge (2r)}$ for all $b$-sparse $\mathcal{G}$.
	\end{Remark}

	\begin{proof}[Proofs of Theorem \ref{th:robust-Davis-Kahan}]
		To characterize how the proposed procedure refines the estimation by initialization and iterations, we define $T_0 = \|\Gamma(N -M)\|=\|\Gamma (Z)\|$ and $K_t = \|N^{(t)} - M\|$ for $t = 0,1,\ldots$. Since $H = \Gamma(H) + G(H)$, we have $\|H\|\leq \|G(H)\|+\|\Gamma(H)\|$ for any matrix $H\in \mathbb{R}^{p\times p}$.
		\begin{itemize}[leftmargin=*]
			\item[Step 1.] We first analyze the initial error $K_0 = \|N^{(0)}-M\|$. By definition, $N^{(0)} = \Gamma(N)$. To better align $\Gamma(N)$ with $M$, we decompose $M = \Gamma(M)+G(M)$. Since the singular subspace of $M$ aligns with $U$, we have $M = P_U M P_U$. Thus,
			\begin{equation*}
				\begin{split}
					K_0 = & \|N^{(0)} - M\| = \|\Gamma(N-M) - G(M)\| \\
					\leq & \|\Gamma(N - M)\| + \|G(M)\| = \|\Gamma(Z)\| + \|G(P_UMP_U)\|\\
					\overset{(a)}{\leq} & \|\Gamma(Z)\| + \frac{I(U)rb}{p}\|M\| = T_0 + \frac{I(U)rb}{p}\|M\|.
				\end{split}
			\end{equation*}
			Here, (a) is due to the contraction property of the map $G(P_U\cdot)$ in \text{Lemma \ref{lm:diagonal-projection-condense}}. Provided that $\frac{I(U)rb}{p}\|M\| \leq \lambda_r(M)/(16\eta)$ in the assumption, we have 
			\begin{equation}\label{ineq:K_0}
				K_0 \leq T_0 + \lambda_r(M)/(16\eta).
			\end{equation}
			\item[Step 2.] Next, we analyze the evolution of iterations by establishing an upper bound for $\|N^{(t)}-M\|$ based on $\|N^{(t-1)}-M\|$. By definition, $$N^{(t)}-M=G(N^{(t)}-M)+\Gamma(N^{(t)}-M).$$ Since the entries indexed by $\mathcal{G}^c$ in $N^{(t)}$ do not change through iterations, $\|\Gamma(N^{(t)}-M)\|=\|\Gamma(N - M)\|$, which can be bounded by $T_0$. The analysis for $G(N^{(t)}-M)$ is more complicated. By definition, the entries indexed by $\mathcal{G}$ in $N^{(t)}$ is the same as the ones in $\tilde{N}^{(t-1)} = P_{U^{(t-1)}}N^{(t-1)}$. To align $M$ with $P_{U^{(t-1)}}N^{(t-1)}$, we decompose $M = P_{U^{(t-1)}}M + P_{U^{(t-1)}_\perp}M$. Thus, 
			\begin{equation*}
				\begin{split}
					G(N^{(t)}-M) = & G\left(P_{U^{(t-1)}}N^{(t-1)} - P_{U^{(t-1)}}M - P_{U^{(t-1)}_{\perp}}M\right)\\ 
					= & G\left(P_{U^{(t-1)}}(N^{(t-1)} - M)\right) - G\left(P_{U^{(t-1)}_{\perp}}M\right).
				\end{split}
			\end{equation*}
			It is still difficult to analyze $G(P_{U^{(t-1)}}(N^{(t-1)} - M))$ due to the complicated connection between $P_{U^{(t-1)}}$ and $N^{(t-1)} - M$. Thus, we decouple them by introducing $P_{U^{(t-1)}} = P_{U} + (P_{U^{(t-1)}}-P_{U})$. Then, we have decomposed $G(N^{(t)}-M)$ into the following three terms:
			\begin{equation}\label{eq:decompose-G(N^{(t)}-M)}
				\begin{split}
					G(N^{(t)}-M) = & G(P_U(N^{(t-1)}-M)) - G\left(P_{U^{(t-1)}_{\perp}}M\right)\\
					& + G\left((P_{U^{(t-1)}} - P_{U})\cdot (N^{(t-1)} - M)\right).
				\end{split}
			\end{equation}
			Next, we bound these three terms separately. In particular, the upper bound of $\|G(P_U(N^{(t-1)}-M))\|$ and $\left\|G(P_{U^{(t-1)}_\perp}M)\right\|$ can be achieved by the application of  the contraction property for the map $G(P_U \cdot)$ in Lemma \ref{lm:diagonal-projection-condense}; to prove an upper bound for $\|G\left((P_{U^{(t-1)}} - P_{U}) \cdot (N^{(t-1)} - M)\right)\|$, we apply the property of $\sin\Theta$ distance to relate $\|(P_{U^{(t-1)}} - P_{U})\|$ to $\frac{\|N^{(t-1)}-M\|}{\lambda_r(M)}$. 
			The detailed proofs are provided as follows.
			\begin{itemize}[leftmargin=*]
				\item By Lemma \ref{lm:diagonal-projection-condense}, 
				\begin{equation}\label{ineq:com-upper-intermediate-3}
					\begin{split}
						& \left\|G\left(P_U(N^{(t-1)}-M)\right)\right\| \leq \sqrt{\frac{I(U) r b(b\wedge r)}{p}} \left\|N^{(t-1)} - M\right\| \\
						= & \sqrt{\frac{I(U)rb(b\wedge r)}{p}}K_{t-1}.
					\end{split}
				\end{equation}
				\item By Lemmas \ref{lm:diagonal-projection-condense} and \ref{lm:P_hat_U M}, 
				\begin{equation}\label{ineq:com-upper-intermediate-5}
					\begin{split}
						& \left\|G(P_{U^{(t-1)}_\perp}M)\right\| = \left\|G(P_{U^{(t-1)}_{\perp}}M P_U)\right\| \leq \sqrt{\frac{I(U)rb(b\wedge r)}{p}}\left\|P_{U^{(t-1)}_\perp}M\right\|\\
						\leq & 2 \sqrt{\frac{I(U)rb(b\wedge r)}{p}}\left\|N^{(t-1)} - M\right\| = 2 \sqrt{\frac{I(U)rb(b\wedge r)}{p}}K_{t-1}.
					\end{split}
				\end{equation}
				\item Note that $U^{(t-1)}(U^{(t-1)})^\top$ and $UU^\top$ are both positive semi-definite and $\|U^{(t-1)}(U^{(t-1)})^\top\|\vee \|UU^\top\|\leq1$, we have $\|U^{(t-1)}(U^{(t-1)})^\top - UU^\top\|\leq 1$. By Lemma 1 in \cite{cai2018rate}, 
				\begin{equation*}
					\begin{split}
						\left\|U^{(t-1)}(U^{(t-1)})^\top - UU^\top\right\| \leq & 2\|\sin\Theta(U^{(t-1)}, U)\|\wedge 1 = 2\|(U^{(t-1)}_{\perp})^\top U\|\wedge 1\\ 
						\leq & \left(2\left\|(U^{(t-1)}_\perp)^\top U U^\top M\right\| \cdot \lambda_{\min}^{-1}(U^\top M)\right) \wedge 1 \\
						\leq & \left(2\left\|(U^{(t-1)}_\perp)^\top M\right\| \cdot \lambda_r^{-1}(M)\right)\wedge 1\\ 
						\leq &  \left(\frac{4\|N^{(t-1)} - M\|}{\lambda_r(M)}\right)\wedge 1 = \frac{4K_{t-1}}{\lambda_r(M)}\wedge 1,
					\end{split}
				\end{equation*}
				where the penultimate step follows from Lemma \ref{lm:P_hat_U M}.
				Note that 
				\begin{equation*}
					\begin{split}
						& \rank((P_{U^{(t-1)}} - P_U)(N^{(t-1)} - M)) \leq \rank(P_{U^{(t-1)}} - P_U) \\
						\leq & \rank(P_{U^{(t-1)}}) + \rank(P_{U}) \leq 2r,
					\end{split}
				\end{equation*}
				we have
				\begin{equation}\label{ineq:com-upper-intermediate-4}
					\begin{split}
						& \left\|G\left((P_{U^{(t-1)}} - P_{U}) \cdot (N^{(t-1)} - M)\right)\right\|\\
						\leq & \eta\cdot \|P_{U^{(t-1)}} - P_{U}\|\cdot \|N^{(t-1)} - M\|\leq \eta K_{t-1} \cdot \left(\frac{4K_{t-1}}{\lambda_r(M)} \wedge 1\right).
					\end{split}
				\end{equation}
			\end{itemize}
			Combining \eqref{eq:decompose-G(N^{(t)}-M)}--\eqref{ineq:com-upper-intermediate-4}, we have for all $t \geq 1$,
			\begin{align}
				K_t \leq & \|\Gamma(N^{(t)}-M)\| + \|G(N^{(t)}-M)\| \nonumber \\
				\leq & T_0 + 3 \sqrt{\frac{I(U) r b(b\wedge r)}{p}} K_{t-1} +  \frac{4\eta}{\lambda_r(M)} K_{t-1}^2.
				\label{eq:Ktt}
			\end{align}
			\item[Step 3.] Finally, we use induction to show that for all $t\geq 0$,
			\begin{equation}\label{ineq:induction-K_t}
				K_t \leq 2T_0 + 2^{-(t+4)} \lambda_r(M)/\eta.
			\end{equation}
			The base case of $t= 0$ is proved by \eqref{ineq:K_0}. Next, suppose the statement \eqref{ineq:induction-K_t} holds for $t-1$. Then
			\begin{equation*}
				\begin{split}
					K_t 
					\stepa{\leq} & T_0 + 3\sqrt{\frac{I(U) r b(b\wedge r)}{p}}K_{t-1} + \frac{4\eta}{\lambda_r(M)} K_{t-1}^2\\
					\stepb{\leq} & T_0 + \frac{K_{t-1}}{4} + K_{t-1} \left(\frac{8\eta T_0}{\lambda_r(M)}  + \frac{1}{4}\right) \\
					\stepc{\leq} & T_0 + \frac{K_{t-1}}{2} \stepd{\leq} T_0 + \frac{2T_0 + \lambda_r(M)\cdot (1/2)^{(t-1)+4}/\eta}{2}\\
					= & 2T_0 + \lambda_r(M)\cdot (1/2)^{t+4}/\eta,
				\end{split}
			\end{equation*}
			where (a) is \prettyref{eq:Ktt};
			(b) is due to the assumption $144I(U) r b (b\wedge r) \leq p$ and the induction hypothesis;
			(c) is from the assumption $T_0 \leq \lambda_r(M)/(64\eta)$;
			(d) is by the induction hypothesis. 
		\end{itemize}
		Therefore, for all $t \geq \Omega(\log \frac{\lambda_r(M)}{T_0\eta} \vee 1) = \Omega(\log\frac{\lambda_r(M)}{\eta\|\Gamma(Z)\|} \vee 1)$, we have $K_t \leq 3T_0$. Finally, the desired \prettyref{eq:robust-Davis-Kahan-intermediate} \prettyref{eq:robust-Davis-Kahan} follow from $\sin\Theta$ perturbation bound \citep[Theorem 5, $q=\infty$]{luo2020schatten}. For completeness, we still provide a proof here. Let $M = US V^\top$ be the singular value decomposition of $M$, where $S$ is diagonal and $U, V\in \mathbb{O}_{p, r}$. Then,
		\begin{equation*}
			\begin{split}
				\|\sin\Theta(\hat{U}^{(t)}, U)\| = & \|(\hat{U}_{\perp}^{(t)})^\top U\| \overset{(a)}{\leq} \frac{\|(\hat{U}_\perp)^\top U S V^\top\|}{\lambda_{r}(S V^\top)} \overset{(b)}{=} \frac{\|(\hat{U}_\perp)^\top M\|}{\lambda_r(USV^\top)}\\ \leq &  \frac{\|(\hat{U}_\perp)^\top(M+N^{(t)}-M)\|+\|(\hat{U}_\perp)^\top(N^{(t)} - M)\|}{\lambda_r(M)}\\
				\overset{(c)}{\leq} & \frac{\|\lambda_{r+1}(N^{(t)})\| + \|N^{(t)}-M\|}{\lambda_r(M)} \overset{(d)}{\leq} \frac{\min_{\rank(T)\leq r}\|N^{(t)} - T\| + K_t}{\lambda_r(M)} \\
				\leq & \frac{\|N^{(t)} - M\| + K_t}{\lambda_r(M)} = \frac{2K_t}{\lambda_r(M)} = \frac{4T_0}{\lambda_r(M)} + (1/2)^{t+3}/\eta.
			\end{split}
		\end{equation*}
		Here, (a) holds because for any matrices $A\in \mathbb{R}^{p\times r}, B\in\mathbb{R}^{r\times p}$, by defining $x_* = \argmax_{\|x\|_2=1} \|x^\top A\| = \|A\|$, then 
		$$\|AB\| = \sup_{\|x\|_2=1} \|x^\top AB\|_2 \geq \|(x_*^\top A) B\|_2 \geq \|x_*^\top A\|_2 \lambda_r(B) = \|A\|\lambda_r(B);$$ 
		(b) holds because $U$ has orthonormal columns; (c) is because $(\hat{U})_\perp$ correspond to the $(r+1)$st, ..., $p$th singular vector of $N^{(t)}$; (d) is due to Eckart-Young-Mirsky Theorem.\footnote{See a proof in \url{https://en.wikipedia.org/wiki/Low-rank_approximation}} Therefore, we have finished the proof of this theorem.
	\end{proof}
	\begin{Remark}\rm
		In fact, Theorem \ref{th:robust-Davis-Kahan} implies Theorem \ref{th:diagonal-less}. To see this, note that if the corruption set $\mathcal{G}$ is the diagonal, $\mathcal{G} = \{(i,i): 1\leq i \leq p\}$, we have
		$$b = \max_i\left\{j: (i, j)\in \mathcal{G}\right\}\vee \max_j\left\{i: (i, j)\in \mathcal{G}\right\} = 1,$$ 
		$$\eta = \max_M \|D(M)\|/\|M\| = \max_M\max_i \frac{|M_{ii}|}{\|M\|} = 1.$$
	\end{Remark}
	
	The next Lemma \ref{lm:diagonal-projection-condense} provides an important technical tool for the proof of robust $\sin\Theta$ theorem. It essentially shows that the operator norm of the composition of linear maps $G(P_U\cdot)$ is much smaller than the product of individual operator norms $\|G(\cdot)\|$ and $\|P_U\|$, provided that the basis $U$ is incoherent; 
	the same conclusion also applies to $G(\cdot P_V)$.
	
	\begin{Lemma}\label{lm:diagonal-projection-condense}
		Assume $\mathcal{G}\subseteq [m_1]\times [m_2]$ is $b$-sparse, i.e., $\max_j\{i:(i,j)\in \mathcal{G}\}\vee\max_i\{j:(i,j)\in \mathcal{G}\}\leq b$. Suppose $U\in \mathbb{O}_{m_1, r}$ and $V\in \mathbb{O}_{m_2, r}$. Recall that $G(A)$ is the matrix $A$ with all entries in $\mathcal{G}^c$ set to zero, $I(U) = \frac{m_1}{r}\max_i \|e_i^\top U\|_2^2$, $I(V) = \frac{m_2}{r}\max_i \|e_i^\top V\|_2^2$, $P_U = UU^\top$, and $P_V = VV^\top$. Then for any matrix $A\in \mathbb{R}^{p_1\times p_2}$, we have
		$$\|G(P_UA )\| \leq \sqrt{\frac{I(U) r b(b\wedge r)}{m_1}}\|A\|, \quad \|G( AP_V)\|\leq \sqrt{\frac{I(V) r b(b\wedge r)}{m_2}}\|A\|,$$
		$$\text{and}\quad \|G(P_U AP_V)\|\leq \frac{\sqrt{I(U)I(V)}\cdot r b}{\sqrt{m_1m_2}}\|A\|.$$
		In particular, recall that $D(A)$ is the matrix $A$ with all off-diagonal entries set to zero. Suppose $U\in \mathbb{O}_{m, r}$. Then for any matrix $A\in \mathbb{R}^{m\times m}$, 
		$$\|D(P_U(D(A)))\| \leq \frac{I(U) r}{m}\|D(A)\|, \quad \|D(P_U A)\| \leq \sqrt{\frac{I(U) r}{m}}\|A\|.$$
	\end{Lemma}
	\begin{proof}[Proof of Lemma \ref{lm:diagonal-projection-condense}]
		\begin{equation*}
			\begin{split}
				\|G(P_UA)\| = & \max_{\|v\|_2 = 1} \|v^\top G(UU^\top A)\|_2 = \max_{\|v\|_2 = 1} \left(\sum_{j=1}^{m_2} \left(v^\top [G(UU^\top A)]_{\cdot j}\right)^2\right)^{1/2}\\
				= & \max_{\|v\|_2 = 1} \left(\sum_{j=1}^{m_2} \left(\sum_{i: (i,j)\in \mathcal{G}}v_i (UU^\top A)_{i,j}\right)^2\right)^{1/2} \\
				\leq & \max_{\|v\|_2=1} \left(\sum_{j=1}^{m_2} \left(\sum_{i: (i,j)\in \mathcal{G}} v_i^2\right) \left(\sum_{i: (i, j)\in \mathcal{G}} (UU^\top A)_{i,j}^2\right)\right)^{1/2},
			\end{split}
		\end{equation*}
		where the inequality is due to Cauchy-Schwarz. Now, for any $1\leq j\leq m_2$, 
		\begin{equation*}
			\begin{split}
				\sum_{i: (i, j)\in \mathcal{G}} (UU^\top A)_{ij}^2 \leq & \sum_{i: (i, j)\in \mathcal{G}} \left(U_{i\cdot}\cdot (U^\top A)_{\cdot j}\right)^2 = \sum_{i: (i, j)\in \mathcal{G}} \left\|U_{i\cdot}\right\|_2^2\cdot \left\|(U^\top A)_{\cdot j}\right\|_2^2 \\
				\leq & \sum_{i: (i, j)\in \mathcal{G}} \frac{I(U)r}{m_1} \|A\|^2 \leq \frac{I(U)rb}{m_1}\|A\|^2.
			\end{split}
		\end{equation*}
		Thus,
		\begin{equation*}
			\begin{split}
				& \|G(P_UA)\| \leq \sqrt{\frac{I(U) rb}{m_1}}\|A\|\max_{\|v\|_2=1}\left(\sum_{j=1}^{m_2}\left(\sum_{i: (i,j)\in \mathcal{G}} v_i^2\right)\right)^{1/2} \\
				= & \sqrt{\frac{I(U) rb}{m_1}}\|A\| \cdot \max_{\|v\|_2=1}\left(\sum_{i=1}^{m_1} \sum_{j: (i,j)\in \mathcal{G}} v_i^2\right)^{1/2} \\
				\leq & \sqrt{\frac{I(U) rb}{m_1}}\|A\| \cdot \max_{\|v\|_2=1} \left(\sum_{i=1}^{m_1} v_i^2b\right)^{1/2} \leq \sqrt{\frac{I(U) r b^2}{m_1}} \|A\|.
			\end{split}
		\end{equation*}
		Additionally, since $\rank(A)\leq r$ and $\mathcal{G}$ is $b$-sparse,
		\begin{equation*}
			\begin{split}
				& \|G(P_UA)\|^2 \leq 	\|G(P_UA)\|_F^2= \sum_{i, j} \left(G(UU^\top A)\right)_{ij}^2 = \sum_{j=1}^{m_2} \sum_{i: (i, j)\in \mathcal{G}} \left(U_{i\cdot}(U^\top A)_{\cdot j}\right)^2\\
				\leq & \sum_{j=1}^{m_2} \sum_{i: (i, j)\in \mathcal{G}} \|U_{i\cdot}\|_2^2 \cdot \|U^\top A_{\cdot j}\|_2^2 \leq \sum_{j=1}^{m_2} \frac{I(U)rb}{m_1} \cdot \|U^\top A_{\cdot j}\|_2^2\\
				= & \frac{I(U)rb}{m_1} \cdot \|U^\top A\|_F^2 \leq \frac{I(U)rb}{m_1}\cdot r\|U^\top A\|^2 \leq \frac{I(U)r^2b}{m_1} \|A\|^2.
			\end{split}
		\end{equation*}
		Combining previous two inequalities, we have
		$$\|G(P_UA)\| \leq \sqrt{I(U)rb(r\wedge b)/m_1}\|A\|.$$
		The proof for $\|G(AP_V)\|\leq \sqrt{I(V) r b(b\wedge r)/m_2} \|A\|$ similarly follows. Next, for any $u\in \mathbb{R}^{m_1}, v\in \mathbb{R}^{m_2}$ such that $\|u\|_2=\|v\|_2=1$, we have
		\begin{equation*}
			\begin{split}
				& u^\top G(P_UAP_V) v = u^\top G(UU^\top AVV^\top) v = \sum_{(i, j)\in \mathcal{G}} u_iv_j \left[UU^\top A VV^\top \right]_{ij}\\
				\leq & \sum_{(i, j)\in \mathcal{G}} |u_iv_j| \|U_{i\cdot}\|_2\cdot \|U^\top AV\|\cdot \|V_{j\cdot}\|_2  \leq \sum_{(i,j)\in \mathcal{G}}|u_iv_j|\cdot \sqrt{\frac{rI(U)}{m_1}}\cdot \|A\|\cdot \sqrt{\frac{rI(V)}{m_2}}\\
				\leq & \frac{\sqrt{I(U)I(V)}r \|A\|}{\sqrt{m_1m_2}} \sum_{(i,j)\in \mathcal{G}}\frac{u_i^2 + v_j^2}{2}\leq \frac{\sqrt{I(U)I(V)}r\|A\|}{\sqrt{m_1m_2}} \left(\sum_i\sum_{j: (i,j)\in \mathcal{G}}\frac{u_i^2}{2} + \sum_j\sum_{i: (i,j)\in \mathcal{G}}\frac{v_j^2}{2}\right)\\
				\leq & \frac{br\sqrt{I(U)I(V)}}{\sqrt{m_1m_2}}\|A\|,
			\end{split}
		\end{equation*}
		which means $\|G(P_UAP_V)\|\leq br\sqrt{I(U)I(V)/(m_1m_2)}\|A\|.$
		
		For the diagonal operator $D(\cdot)$, since $D(A)$ is a diagonal matrix, we have $D(A)e_i = D(A)_{ii} e_i$ and
		\begin{equation*}
			\begin{split}
				& \|D(P_U(D(A)))\| = \max_{i}\left|\left\{P_U(D(A))\right\}_{ii}\right| = \max_i \left|e_i^\top P_U D(A) e_i \right| \\
				= & \max_i \left|e_i^\top P_U e_i\cdot A_{ii}\right|
				= \max_{i}\|U^\top e_i\|_2^2 \cdot |A_{ii}| \leq \frac{I(U) r}{m}\|D(A)\|,
			\end{split}
		\end{equation*}
		\begin{equation*}
			\begin{split}
				\|D(P_U(A))\| = & \max_i \left|(P_U A)_{ii}\right| = \max_{i}\left|e_i^\top UU^\top Ae_i\right| \\
				\leq & \|e_i^\top U\|_2\cdot \|A\| \leq \sqrt{\frac{I(U) r}{m}}\|A\|.
			\end{split}
		\end{equation*}
	\end{proof}
\end{sloppypar}

\bibliographystyle{imsart-number}
\bibliography{reference}

\newpage
\appendix

\setcounter{page}{1}
\setcounter{section}{0}

\begin{center}
	{\LARGE Supplement to ``Heteroskedastic PCA: Algorithm,}
	
	\medskip 
	
	{\LARGE Optimality, and Applications"
	}
	
	\bigskip
	
	{\large Anru R. Zhang, ~ T. Tony Cai, ~ and ~ Yihong Wu}
	
\end{center}

\begin{abstract}
	In this supplement, we provide additional proofs of the main theorems and the key technical tools for the main technical results.
\end{abstract}

\section{Additional Discussion and Lower Bound on Robust $\sin\Theta$ Theorem}\label{sec:addition-sin-theta}

\begin{Remark}\label{rm:spiky}\rm 
	In the robust $\sin\Theta$ theorem (Theorem \ref{th:diagonal-less}), we introduce the incoherence condition \eqref{ineq:incoherence-constant-dl-d-k} here to avoid those $M$ that are too ``spiky". For example, consider $M_1 = e_1e_1^\top$ and $M_2=e_2e_2^\top$. Then $\Delta(M_1) = \Delta(M_2)$ and there is no way to distinguish these two spiky matrices if one only has reliable off-diagonal observations. Similar conditions, such as the ``delocalized condition," appear in recent work on PCA from noisy and linearly reduced data \citep{dobriban2016pca}. The incoherence condition has been widely used in the matrix completion literature (e.g., \cite[Assumption A0]{recht2011simpler}), where $I(U) \leq \mu_0$ is often assumed for some constant $\mu_0$ independent of $p$. In comparison, in view of the trivial bound
	$$I(U) = \frac{p}{r}\max_{1\leq i\leq p}\|e_i^\top U\|_2^2 \leq \frac{p}{r}\cdot 1,$$
	our assumption $I(U)< c_Ip/r$ is much looser than those that are prevalent in the matrix completion literature.
	
\end{Remark}

The following lower bound shows that bounds for both the incoherence condition \eqref{ineq:incoherence-constant-dl-d-k} and the estimation error \eqref{ineq:dl-d-k-upper-bound} are rate-optimal.
\begin{Proposition}[Robust $\sin \Theta$ theorem: lower bound]\label{th:lower-deterministic}
	Define the following collection of pairs of signal and perturbation matrices:
	\begin{equation}
		\begin{split}
			\mathcal{D}_{p, r}(\nu,\delta, t) = \left\{(M, Z): \begin{array}{ll}
				M = U\Lambda U^\top, U\in \mathbb{O}_{p, r}, \\
				I(U) \|M\|/\lambda_r(M) \leq t, \\
				\|\Delta(Z)\|\leq \delta, \lambda_r(M) \geq \nu
			\end{array}\right\}.
		\end{split}
	\end{equation}
	Suppose $1\leq r\leq p/2, t\geq 4$, one observes $N = M +Z\in \mathbb{R}^{p\times p}$. Then 
	\begin{equation}
		\inf_{\widehat{U}} \sup_{(M, Z)\in \mathcal{D}_{p, r}(\nu,\delta, t)}\left\|\sin\Theta(\widehat{U}, U)\right\| \geq c\left(\frac{\delta}{\nu} \wedge 1\right).
	\end{equation}
	If the incoherence constraint, i.e., $I(U)\|M\|/\lambda_r(M)\leq t$, is weak in the sense that $t \geq p/r$, then
	\begin{equation}
		\inf_{\widehat{U}} \sup_{(M, Z)\in \mathcal{D}_{p, r}(\nu,\delta, t)}\left\|\sin\Theta(\widehat{U}, U)\right\| \geq 1/2.
	\end{equation}
\end{Proposition}

\begin{proof}[Proof of Proposition \ref{th:lower-deterministic}]
	We first develop the lower bound with the incoherence constraint. We first assume $\delta/\nu\leq 1/\sqrt{2}$. Let $d = 2\lfloor p/(2r)\rfloor$, $\alpha, \beta \in \mathbb{R}^{d}$ be unit vectors such that 
	$$\alpha = \frac{1}{\sqrt{d}}\left(1,\ldots, 1\right),\quad \beta = \frac{1}{\sqrt{d(1+\theta^2)}}\left(1+\theta ,\ldots, 1+\theta, 1-\theta, \ldots, 1-\theta\right).$$
	Clearly, $f(\theta)\triangleq \|\alpha\alpha^\top - \beta\beta^\top\| $ is a continuous function of $\theta$. One can verify that $f(0) = 0$; $f(1) = 1/\sqrt{2}$, then there exists $0\leq \theta \leq 1$ to ensure that 
	\begin{equation}\label{eq:alphalaphaT-betabetaT}
		\|\alpha \alpha^\top - \beta\beta^\top\| = \delta/\nu.
	\end{equation}
	Based on \eqref{eq:alphalaphaT-betabetaT}, we additionally construct
	\begin{equation}\label{eq:U^1-U^2}
		U^{(1)} = \begin{bmatrix}
			\alpha_1 I_r\\
			\vdots \\
			\alpha_d I_r\\
			0_{(p-rd), r}
		\end{bmatrix},\quad U^{(2)} = \begin{bmatrix}
			\beta_1 I_r\\
			\vdots \\
			\beta_d I_r\\
			0_{(p-rd), r}
		\end{bmatrix}.
	\end{equation}
	Here, $\frac{1}{\sqrt{d}}I_r$ is repeated for $d$ times in $U^{(1)}$; both $\frac{1+\theta}{\sqrt{d(1+\theta^2)}}I_r$ and $\frac{1-\theta}{\sqrt{d(1+\theta^2)}}I_r$ are repeated for $d/2$ times in $U^{(2)}$. Let $M^{(1)} = \nu U^{(1)}(U^{(1)})^\top, M^{(2)} = \nu U^{(2)}(U^{(2)})^\top$, $Z^{(1)} = \frac{1}{2}(M^{(2)} - M^{(1)}), Z^{(2)} = \frac{1}{2}(M^{(1)} - M^{(2)})$. By such the construction, $\lambda_r(M^{(1)}) = \lambda_r(M^{(2)}) = \nu$, $\|M^{(1)}\|/\lambda_r(M^{(1)})=\|M^{(2)}\|/\lambda_r(M^{(2)})=1$,
	$$I(U^{(1)}) = \frac{p}{r}\max_i\|e_i^\top U^{(1)}\|_2^2 \leq \frac{p}{rd} = \frac{p}{r\cdot 2\lfloor p/(2r)\rfloor} < \frac{p}{r\cdot2(p/(2r) -1)} \leq 2,$$
	\begin{equation*}
		\begin{split}
			I(U^{(2)}) = & \frac{p}{r}\max_i \|e_i^\top U^{(2)}\|_2^2 \leq \frac{p(1+\theta)^2}{r\cdot \left(d\left(1+\theta^2\right)\right)}\leq \frac{p}{r}\cdot \frac{2}{d} = \frac{p}{r}\cdot \frac{1}{\lfloor p/(2r) \rfloor}\\
			\leq & \left\{\begin{array}{ll}
				4\cdot 1 \leq 4, & \text{if } 2r\leq p \leq 4r;\\
				\frac{p}{r}\cdot \frac{1}{p/(2r)-1} = \frac{2p}{p-2r}\leq 4, & \text{if } 4r+1\leq p.
			\end{array}\right..
		\end{split}
	\end{equation*}

	\begin{equation*}
		\begin{split}
			\|\Delta(Z^{(1)})\| = & \|\Delta(Z^{(2)})\| \overset{\text{Lemma \ref{lm:diagonal-less-spectral-norm}}}{\leq} 2\|Z^{(2)}\| \leq 2\left\|\frac{1}{2}\left(M^{(2)} - M^{(1)}\right)\right\| \\
			= & \nu\left\|\alpha\alpha^\top - \beta\beta^\top \right\| = \delta,
		\end{split}
	\end{equation*}
	which means $(M^{(1)}, Z^{(1)}), (M^{(2)}, Z^{(2)}) \in \mathcal{D}_{p, r}(\nu,\delta, t)$ for $t \geq 4$. On the other hand, by \cite[Lemma 1]{cai2018rate}, 
	\begin{equation*}
		\begin{split}
			& \left\|\sin\Theta(U^{(1)}, U^{(2)})\right\| \geq \frac{1}{2}\|U^{(1)}(U^{(1)})^\top - U^{(2)}(U^{(2)})^\top \|\\
			\overset{\eqref{eq:U^1-U^2}}{=} & \frac{1}{2}\left\|\begin{bmatrix}
				(\alpha_1^2 - \beta_1^2)I_r & \cdots & (\alpha_1\alpha_d - \beta_1\beta_d)I_r\\
				\vdots & & \vdots\\
				(\alpha_d\alpha_1 - \beta_d\beta_1)I_r & \cdots & (\alpha_d^2 - \beta_d^2)I_r \\
			\end{bmatrix}\right\| \\
			= & \frac{1}{2}\|\alpha\alpha^\top - \beta\beta^\top\| = \delta/(2\nu).
		\end{split}
	\end{equation*}
	Given $M^{(1)} + Z^{(1)} = M^{(2)} + Z^{(2)}$, we have
	\begin{equation*}
		\begin{split}
			& \inf_{\widehat{U}} \sup_{(M, Z)\in \mathcal{D}_{p,r}(\nu,\delta,t)} \left\|\sin\Theta(\widehat{U}, U)\right\| \\
			\geq & \inf_{\widehat{U}} \sup_{(M, Z)\in \left\{(M^{(1)}, Z^{(1)}), (M^{(2)}, Z^{(2)})\right\}} \left\|\sin\Theta(\widehat{U}, U)\right\| \\
			\geq & \inf_{\widehat{U}} \frac{1}{2}\left( \left\|\sin\Theta(\widehat{U}, U^{(1)})\right\| + \left\|\sin\Theta(\widehat{U}, U^{(2)})\right\|\right) \geq \frac{1}{2}\left\|\sin\Theta(U^{(1)}, U^{(2)})\right\| = \frac{\delta}{4\nu}.
		\end{split}
	\end{equation*}
	Next, if $\delta/\nu \geq \sqrt{2}/2$, let $\delta_0 = \nu\cdot \sqrt{2}/2$. By the previous argument, one can show
	\begin{equation*}
		\inf_{\widehat{U}} \sup_{(M, Z)\in \mathcal{D}_{p,r}(\nu,\delta,t)} \left\|\sin\Theta(\widehat{U}, U)\right\| \geq \frac{\delta_0}{4\nu} = \frac{\sqrt{2}}{8} \geq \frac{\sqrt{2}}{8}\left(\frac{\delta}{\nu}\wedge 1\right).
	\end{equation*}
	In summary, we must have
	\begin{equation*}
		\inf_{\widehat{U}} \sup_{(M, Z)\in \mathcal{D}_{p,r}(\nu,\delta,t)} \left\|\sin\Theta(\widehat{U}, U)\right\| \geq  \frac{\sqrt{2}}{8}\left(\frac{\delta}{\nu}\wedge 1\right)
	\end{equation*}
	in the first scenario that $t\geq 4$.
	
	Then we consider the second part that $t\geq p/r$. Let
	\begin{equation*}
		U^{(1)} = \begin{bmatrix}
			I_r\\
			0_{(p-r)\times r}
		\end{bmatrix}, \quad U^{(2)} = \begin{bmatrix}
			0_{r\times r}\\
			I_r\\
			0_{(p-2r)\times r}
		\end{bmatrix}
	\end{equation*}
	be two orthogonal matrices, $M^{(1)} = \nu U^{(1)}(U^{(1)})^\top, M^{(2)} = \nu U^{(2)}(U^{(2)})^\top$, $Z^{(1)} = - M^{(1)}, Z^{(2)} = -M^{(2)}$. Then clearly, $M^{(1)} +Z^{(1)} = M^{(2)} + Z^{(2)}$, $\lambda_r(M^{(1)}) = \lambda_r(M^{(2)}) \geq \nu$, $\|\Delta(Z^{(1)})\| = \|\Delta(Z^{(2)})\|=0$,
	$$\|\sin\Theta(U^{(1)}, U^{(2)})\| = \left(1 - \lambda_r((U^{(1)})^\top U^{(2)})\right)^{1/2} = (1-0)^{1/2} = 1.$$ 
	Moreover, for any $t\geq p/r$,
	\begin{equation*}
		I(U^{(1)}) = \frac{p}{r}\|e_i^\top U^{(1)}\|_2^2 = \frac{p}{r} \leq t, \quad I(U^{(2)}) = \frac{p}{r}\|e_i^\top U^{(2)}\|_2^2 = \frac{p}{r} \leq t.
	\end{equation*}
	We thus have
	\begin{equation*}
		\left(M^{(1)}, Z^{(1)}\right), \left(M^{(2)}, Z^{(2)}\right)  \in \mathcal{D}_{p,r}(\nu, \delta,t)
	\end{equation*}
	if $t\geq p/r$. Given $M^{(1)} + Z^{(1)} = M^{(2)} + Z^{(2)}$, we have
	\begin{equation*}
		\begin{split}
			& \inf_{\widehat{U}} \sup_{(M, Z)\in \mathcal{D}_{p,r}(\nu,\delta,t)} \left\|\sin\Theta(\widehat{U}, U)\right\| \\
			\geq & \inf_{\widehat{U}} \sup_{(M, Z)\in \left\{(M^{(1)}, Z^{(1)}), (M^{(2)}, Z^{(2)})\right\}} \left\|\sin\Theta(\widehat{U}, U)\right\| \\
			\geq & \inf_{\widehat{U}} \frac{1}{2}\left( \left\|\sin\Theta(\widehat{U}, U^{(1)})\right\| + \left\|\sin\Theta(\widehat{U}, U^{(2)})\right\|\right) \geq \frac{1}{2}\left\|\sin\Theta(U^{(1)}, U^{(2)})\right\| = \frac{1}{2},
		\end{split}
	\end{equation*}
	which has finished the proof of this theorem. 
\end{proof}

\section{Additional Proofs}\label{sec:additional-proofs}

\subsection{Additional Proofs for Heteroskedastic PCA}

\begin{proof}[Proof of Theorem \ref{th:lower-bound-hetero-PCA}] 
	We only need to show the following two inequalities to prove this theorem,
	\begin{equation}\label{ineq:PCA-lower-1}
		\inf_{\widehat{U}}\sup_{\Sigma \in \mathcal{F}_{p, n, r}(\sigsumch, \sigmaxch, \nu, \kappa)}\mathbb{E}\left\|\sin\Theta(\widehat{U}, U)\right\| \gtrsim \left(\frac{\sigsumch}{\left(n\nu\right)^{1/2}} + \frac{\sigsumch\sigmaxch}{n^{1/2}\nu}\right) \wedge 1,
	\end{equation}
	\begin{equation}\label{ineq:PCA-lower-2}
		\inf_{\widehat{U}}\sup_{\Sigma \in \mathcal{F}_{p, n, r}(\sigsumch, \sigmaxch, \nu, \kappa)}\mathbb{E}\left\|\sin\Theta(\widehat{U}, U)\right\| \gtrsim \frac{\sqrt{r}\sigmaxch}{(n\nu)^{1/2}} \wedge 1. 
	\end{equation}
	We first consider \eqref{ineq:PCA-lower-1}. Since all parameters can be rescaled, we assume $\nu=1$ without loss of generality. The proof is divided into three steps.
	\begin{itemize}[leftmargin=*]
		\item[Step 1] In this step, we construct a series of ``candidate covariance matrices" and prove that they belong to the subset of covariance matrices in the theorem statement. Let 
		\begin{equation}\label{eq:def_d-L}
			d = \lfloor \sigsumch^2/(8\sigmaxch^2)\rfloor \vee 6, \quad L = 2\lceil 1/(dc_I)\rceil.
		\end{equation}
		Now, we impose the assumption that 
		\begin{equation}\label{ineq:p-lower-condition}
			p\geq 50 \vee \{2(r-1)(1+c_I)/c_I\} \vee \{8/c_I\}.
		\end{equation} 
		Since $\sigsumch \leq \sqrt{p} \sigmaxch$, we must have
		\begin{equation}\label{ineq:Ld<p/2}
			\begin{split}
				Ld \overset{\eqref{eq:def_d-L}}{=} & 2d\left\lceil \frac{1}{dc_I} \right\rceil < 2d \left(\frac{1}{dc_I} + 1\right) = \frac{2}{c_I} + 2\left(\left\lfloor\frac{\sigsumch^2}{8\sigmaxch^2} \right\rfloor\vee 6\right)\\ \overset{\eqref{ineq:p-lower-condition}}{\leq} & \frac{p}{4} + \frac{\sigsumch^2}{4\sigmaxch^2}\vee 12
				\overset{\eqref{ineq:p-lower-condition}}{\leq} \frac{p}{4} + \max\left\{\frac{p}{4}, \frac{p}{4}\right\} = \frac{p}{2}.
			\end{split}
		\end{equation}

		By Lemma \ref{lm:incoherence-orthogonoal-construction}, we can construct $Q\in \mathbb{O}_{(p-Ld), (r-1)}$ with small incoherence constant:
		\begin{equation}\label{ineq:e_i^topQ}
			\begin{split}
				& \max_i\|e_i^\top Q\|_2^2 \leq  \frac{1}{\lfloor\frac{p-Ld}{r-1}\rfloor} \leq \frac{1}{\frac{p-Ld}{r-1}-1} \\
				\overset{\eqref{ineq:Ld<p/2}}{\leq} & \frac{1}{\frac{p/2}{r-1}-1} \overset{\eqref{ineq:p-lower-condition}}{\leq} \frac{r-1}{(r-1)(1+c_I)/c_I-(r-1)} \leq c_I.
			\end{split}
		\end{equation}
		By the Varshamov-Gilbert bound \citep[Lemma 4.7]{massart2007concentration}, we can find series of vectors $v^{(1)},\ldots, v^{(N)} \subseteq \{-1, 1\}^d$ with $N \geq \exp(d/8)$, such that 
		\begin{equation}\label{ineq:PCA-lower-3}
			\|v^{(l)} - v^{(k)}\|_2^2 \geq d, \quad \text{for all } 1\leq k \neq l \leq N
		\end{equation}
		Next, we construct a series of candidate covariance matrices for $k=1,\ldots, N$,
		\begin{equation*}
			\begin{split}
				& U^{(k)} = \begin{bmatrix}
					u^{(k)} & 0_{(Ld)\times (r-1)}\\
					0_{(p-Ld)\times 1} & Q\\
				\end{bmatrix} \in \mathbb{R}^{p\times r}, \\
				& u^{(k)} = \begin{bmatrix}
					\frac{1}{\sqrt{Ld(1+\theta^2)}}\left(1_d+\theta v^{(k)}\right) \\
					\vdots\\
					\frac{1}{\sqrt{Ld(1+\theta^2)}}\left(1_d+\theta v^{(k)}\right) \\
					\frac{1}{\sqrt{Ld(1+\theta^2)}}\left(1_d-\theta v^{(k)}\right) \\
					\vdots\\
					\frac{1}{\sqrt{Ld(1+\theta^2)}}\left(1_d-\theta v^{(k)}\right)
				\end{bmatrix}\in \mathbb{R}^{Ld};
			\end{split}
		\end{equation*}
		\begin{equation*}
			\begin{split}
				& D_{ij} = \left\{
				\begin{array}{ll}
					\sigma_0^2, & 1 \leq i = j \leq Ld;\\
					0, & \text{otherwise},
				\end{array}\right. \quad \sigma_0^2 = \sigmaxch^2\wedge \{\sigsumch^2/(Ld)\}, \\
				& \Sigma^{(k)} = U^{(k)}(U^{(k)})^\top + D.
			\end{split}
		\end{equation*}
		Here, $0\leq \theta\leq 1$ is a constant to be specified later; both $\frac{1}{\sqrt{Ld(1+\theta^2)}}(1+\theta v^{(k)})$ and $\frac{1}{\sqrt{Ld(1+\theta^2)}}(1-\theta v^{(k)})$ are repeated for $(L/2)$ times in the first column of $U^{(k)}$. Then, all columns of $U^{(k)}$ are orthonormal and
		\begin{equation*}
			\begin{split}
				\max_{1\leq i\leq p}\|e_i^\top U^{(k)}\|_2^2 \leq & \max\left\{\frac{(1+\theta)^2}{Ld(1+\theta^2)}, \max_i\|e_i^\top Q\|_2^2\right\} \\
				\leq & \max\left\{\frac{2}{Ld}, \max_i\|e_i^\top Q\|_2^2\right\}  \overset{\eqref{eq:def_d-L}\eqref{ineq:e_i^topQ}}{\leq}  c_I.
			\end{split}
		\end{equation*}
		Then $U^{(k)}(U^{(k)})^\top$ satisfies the incoherence constraint of the class $\mathcal{F}_{p,n,r}(\sigsumch,\sigmaxch, \nu, \kappa)$, 
		$$I\left(U^{(k)}\right) = \frac{p}{r}\max_i\|e_i^\top U^{(k)}\|_2^2 \leq c_I p/r.$$
		In addition,
		\begin{equation*}
			\begin{split}
				& \max_{1\leq i\leq p}D_{ii} = \sigmaxch^2\wedge \{\sigsumch^2/(Ld)\} \leq \sigmaxch^2,\\ 
				& \sum_{i=1}^p D_{ii} = Ld \left(\sigmaxch^2\wedge \{\sigsumch^2/(Ld)\}\right) \leq \sigsumch^2,\\
				& \lambda_r\left(U^{(k)}(U^{(k)})^\top\right) = 1 = \nu.
			\end{split}
		\end{equation*}
		Therefore, $\Sigma^{(1)}, \ldots, \Sigma^{(N)}$ truly belongs to the class in the theorem statement:
		\begin{equation}\label{ineq:PCA-lower-belongs}
			\Sigma^{(1)}, \ldots, \Sigma^{(N)} \subseteq \mathcal{F}_{p, n, r}(\sigsumch, \sigmaxch, \nu, \kappa).
		\end{equation}
		\item[Step 2] Next for any $k\neq l$, we prove that $U^{(k)}, U^{(l)}$ are well-separated and the KL-divergence of $X^{(k)}$ and $U^{(l)}$ are bounded if $X^{(k)}\sim N(0, \Sigma^{(k)}), X^{(l)} \sim N(0, \Sigma^{(l)})$. Since $\sigsumch \geq \sigmaxch$, we have
		\begin{equation}\label{ineq:PCA-lower-tilde-sigma-0}
			\begin{split}
				\sigma_0^2 = & \sigmaxch^2\wedge \frac{\sigsumch^2}{Ld} \overset{\eqref{eq:def_d-L}}{\geq} \sigmaxch^2\wedge \frac{\sigsumch^2}{2d\lceil1/(dc_I)\rceil}  \geq \sigmaxch^2\wedge \frac{\sigsumch^2}{2d\left(\frac{1}{dc_I} + 1\right)}\\
				\geq & \sigmaxch^2\wedge \frac{\sigsumch^2}{\frac{2}{c_I} + 2\left(\lfloor\sigsumch^2/(8\sigmaxch^2)\rfloor\wedge 6\right)} \geq \sigmaxch^2 \wedge \frac{\sigsumch^2}{\frac{2}{c_I} + 12} \geq c\sigmaxch^2;\\
				d\sigma_0^2 \geq & cd\sigmaxch^2 = c\left(\lfloor\sigsumch^2/(8\sigmaxch^2)\rfloor\vee 6\right)\sigmaxch^2\geq c\left(\sigsumch^2/(16\sigmaxch^2)\right)\sigmaxch^2 \geq c'\sigsumch^2
			\end{split}
		\end{equation}
		for some constants $c, c'>0$ that only rely on $c_I$.
		
		By the definition of \eqref{ineq:PCA-lower-3}, we have for any $1\leq k \neq l \leq N$, 
		\begin{equation*}
			\begin{split}
				& \left\|\sin\Theta\left(U^{(k)}, U^{(l)}\right)\right\| = \left(1 - \lambda_r^2\left((U^{(k)})^\top U^{(l)}\right)\right)^{1/2} = \left(1 - \left(u^{(k)\top} u^{(l)}\right)^2\right)^{1/2}\\
				= & \left(1 - \frac{(L/2)^2}{L^2d^2(1+\theta^2)^2}\left((1_d+\theta v^{(k)})^\top(1_d+\theta v^{(l)}) + (1_d-\theta v^{(k)})^\top(1_d-\theta v^{(l)}) \right)^2\right)^{1/2}\\
				= & \left(1 - \frac{1}{4d^2(1+\theta^2)^2}\left(2d + 2\theta^2v^{(k)\top}v^{(l)}\right)^2\right)^{1/2} \\
				= & \left(1 - \left(\frac{1+\theta^2 (v^{(k)})^\top v^{(l)}/d}{1+\theta^2}\right)^2\right)^{1/2}.
			\end{split}
		\end{equation*}
		By \eqref{ineq:PCA-lower-3}, for any $k\neq l$, we have $d\leq \|v^{(k)}-v^{(l)}\|_2^2\leq 4d$ and
		\begin{equation*}
			\begin{split}
				(v^{(k)})^\top v^{(l)} = & \frac{1}{2}\left(\|v^{(k)}\|_2^2 + \|v^{(l)}\|_2^2 - \|v^{(k)} - v^{(l)}\|_2^2 \right) \\
				= & \frac{1}{2}\left(2d - \|v^{(k)}-v^{(l)}\|_2^2\right)\in \left[-d, d/2\right].
			\end{split}
		\end{equation*}
		Consequently,
		\begin{equation}\label{ineq:PCA-lower-4}
			\left(1 - \left(\frac{1+\theta^2/2}{1+\theta^2}\right)^2\right)^{1/2}\leq \left\|\sin\Theta(U^{(k)}, U^{(l)})\right\|\leq \left(1 - \left(\frac{1-\theta^2}{1+\theta^2}\right)^2\right)^{1/2}.
		\end{equation}
		Provided that $0<\theta\leq 1$, 
		\begin{equation}\label{ineq:PCA-lower-5}
			\left(1 - \left(\frac{1-\theta^2}{1+\theta^2}\right)^2\right)^{1/2} = \left(\frac{(1+\theta^2)^2 - (1-\theta^2)^2}{(1+\theta^2)^2}\right)^{1/2} = \frac{2\theta}{1+\theta^2} \leq 2\theta,
		\end{equation}
		\begin{equation}\label{ineq:PCA-lower-6}
			\left(1 - \left(\frac{1+\theta^2/2}{1+\theta^2}\right)^2\right)^{1/2} = \frac{\left(\theta^2 + (3/4)\theta^4\right)^{1/2}}{1+\theta^2} \geq \frac{\theta}{2}.
		\end{equation}
		Combining \eqref{ineq:PCA-lower-4}, \eqref{ineq:PCA-lower-5}, and \eqref{ineq:PCA-lower-6}, we have
		\begin{equation}\label{ineq:PCA-lower-7}
			\frac{\theta}{2}\leq \left\|\sin\Theta(U^{(k)}, U^{(l)})\right\| \leq 2\theta,\quad \forall 1\leq k\neq l \leq N.
		\end{equation}
		
		Suppose
		$$X^{(k)} = \left[X^{(k)}_1 ~\ldots~X^{(k)}_n\right] \overset{iid}{\sim} N(0, \Sigma^{(k)}),\quad k=1,\ldots, N.$$
		Next, we consider the Kullback-Leibler divergence between $X^{(k)}$ and $X^{(l)}$ for any $1\leq k \neq l\leq N$. Note the following fact on the Kullback-Leibler divergence between multivariate Gaussians: suppose $X = [X_1,\ldots, X_n] \overset{iid}{\sim} N(0, \Sigma)$ and $X' = [X_1',\ldots, X_n'] \overset{iid}{\sim} N(0, \Sigma')$ are $p$-dimensional vectors. If $\Sigma$ and $\Sigma'$ are non-degenerating, then
		$$D_{KL}\left(X||X'\right) = \frac{n}{2}\left(\tr\left((\Sigma')^{-1}\Sigma\right) - p + \log \left(\frac{\det \Sigma'}{\det \Sigma}\right)\right).$$
		Since $\Sigma^{(k)}$ and $\Sigma^{(l)}$ may be degenerating, one cannot directly apply the previous formula to calculate their KL divergence. Instead, denote the top $(Ld)$-by-$(Ld)$ sub-matrix of $\Sigma^{(k)}$ as
		\begin{equation*}
			\begin{split}
				& \widetilde{\Sigma}^{(k)} = u^{(k)}(u^{(k)})^\top + \widetilde{D} \in \mathbb{R}^{(Ld) \times (Ld)},\\
				\text{where} \quad & u^{(k)} = \begin{bmatrix}
					\frac{1}{\sqrt{Ld(1+\theta^2)}}\left(1_d+\theta v^{(k)}\right) \\
					\vdots \\
					\frac{1}{\sqrt{Ld(1+\theta^2)}}\left(1_d+\theta v^{(k)}\right) \\
					\frac{1}{\sqrt{Ld(1+\theta^2)}}\left(1_d-\theta v^{(k)}\right) \\
					\vdots \\
					\frac{1}{\sqrt{Ld(1+\theta^2)}}\left(1_d-\theta v^{(k)}\right) \\
				\end{bmatrix}\in \mathbb{R}^{Ld}, \quad \widetilde{D} = 
				\sigma_0^2 I.
			\end{split}
		\end{equation*}
		By the structure of $\Sigma^{(k)}$, we know $\det(\widetilde\Sigma^{(k)}) = \det(\widetilde\Sigma^{(l)})$ for all $1\leq k, l\leq N$, and $\Sigma^{(k)}_{[1:Ld, 1:Ld]} = \widetilde{\Sigma}^{(k)}$, $\Sigma^{(k)}_{[(Ld+1):p, 1:Ld]} = 0$, $\Sigma^{(k)}_{[1:Ld,(Ld+1):p]} = 0$, $\Sigma^{(k)}_{[(Ld+1):p,(Ld+1):p]} = QQ^\top$. Here, $\Sigma^{(k)}_{[1:Ld, 1:Ld]}$ represents the submatrix formed by the first to $Ld$-th rows and first to $Ld$-th columns of $\Sigma^{(k)}$; $\Sigma^{(k)}_{[1:Ld,(Ld+1):p]}$ and $\Sigma^{(k)}_{[(Ld+1):p,(Ld+1):p]}$ are defined in a similar fashion. Then, 1) for any $1\leq k\leq N$ and $1\leq i\leq n$, $(X_i^{(k)})_{[1:Ld]}$ and $(X_i^{(k)})_{[(Ld+1):p]}$, i.e., the first $Ld$ entries and the other entries of $X_i$, are two independent vectors; 2) $(X_1^{(k)})_{[(Ld+1):p]},\ldots, (X_n^{(k)})_{[(Ld+1):p]}$ are independent and identically distributed. Thus,
		\begin{equation*}
			D_{KL}\left(X^{(l)}||X^{(k)}\right) = D_{KL}\left(X^{(l)}_{[1:Ld, :]}||X^{(k)}_{[1:Ld, :]}\right) =  \frac{n}{2}\left(\tr\left((\widetilde{\Sigma}^{(k)})^{-1}\widetilde{\Sigma}^{(l)}\right) - Ld\right).
		\end{equation*}
		Here, $X_{[1:LD, :]}^{(k)}$ and $X_{[1:LD, :]}^{(l)}$ represent the first $LD$ rows of $X^{(k)}$ and $X^{(l)}$, respectively. Since $u^{(k)}$ is a unit vector, one can verify that
		\begin{equation*}
			(\widetilde{\Sigma}^{(k)})^{-1} = \sigma_0^{-2}I_{Ld} + \left(\frac{1}{\sigma_0^2+1} - \sigma_0^{-2}\right)u^{(k)}(u^{(k)})^\top,
		\end{equation*} 
		\begin{equation*}
			\begin{split}
				(\widetilde{\Sigma}^{(k)})^{-1}\widetilde{\Sigma}^{(l)} = & I_{Ld} + \left(\frac{\sigma_0^2}{\sigma_0^2+1} - 1\right)u^{(k)}(u^{(k)})^\top + \sigma_0^{-2}u^{(l)}(u^{(l)})^\top \\
				& + \left(\frac{1}{\sigma_0^2+1} - \sigma_0^{-2}\right)u^{(k)}(u^{(k)})^\top u^{(l)}(u^{(l)})^\top,
			\end{split}
		\end{equation*}
		and
		\begin{equation}\label{ineq:D_KLX^lX_k}
			\begin{split}
				& D_{KL}\left(X^{(l)}||X^{(k)}\right) \\
				= & \frac{n}{2}\left(Ld + \left(\frac{\sigma_0^2}{\sigma_0^2+1}-1 + \sigma_0^{-2}\right) + \left(\frac{1}{\sigma_0^2+1}-\sigma_0^{-2}\right)\left((u^{(k)})^\top u^{(l)}\right)^2 - Ld\right)\\
				= & \frac{n}{2\sigma_0^2(\sigma_0^2+1)}\cdot \left(1 - \left((u^{(k)})^\top u^{(l)}\right)^2\right) = \frac{n}{2\sigma_0^2(\sigma_0^2+1)} \left\|\sin\Theta\left(U^{(k)}, U^{(l)}\right)\right\|^2\\
				\overset{\eqref{ineq:PCA-lower-7}}{\leq} & \frac{2n\theta^2}{\sigma_0^2(1+\sigma_0^2)}.
			\end{split}
		\end{equation}
		\item[Step 3] We finalize the proof by the generalized Fano's lemma. Specifically by \cite[Lemma 3]{yu1997assouad}, we have
		\begin{equation*}
			\begin{split}
				& \inf_{\widehat{U}}\sup_{\Sigma \in \mathcal{F}_{p,n, r}(\sigsumch, \sigmaxch, \nu, \kappa)}\mathbb{E}\|\sin\Theta(\widehat{U}, U)\| \overset{\eqref{ineq:PCA-lower-belongs}}{\geq} \inf_{\widehat{U}}\sup_{\Sigma\in \{\Sigma^{(l)}\}_{l=1}^N}\mathbb{E}\|\sin\Theta(\widehat{U}, U)\|\\ 
				\overset{\eqref{ineq:PCA-lower-7}\eqref{ineq:D_KLX^lX_k}}{\geq} & \frac{\theta}{4}\left(1 - \frac{\frac{2n\theta^2}{\sigma_0^2(1+\sigma_0^2)}+\log(2)}{\log(N)}\right) \overset{N\geq 3}{\geq} \frac{\theta}{4}\left(1 - \frac{\frac{2n\theta^2}{\sigma_0^2(1+\sigma_0^2)} + \log(2)}{(d/8)\vee \log(3)}\right)\\
				\geq & \frac{\theta}{4}\left(1 - \frac{\frac{2n\theta^2}{\sigma_0^2(1+\sigma_0^2)}}{(d/8)\vee \log(3)} - \frac{\log(2)}{(d/8)\vee \log(3)}\right) \\
				\geq & \frac{\theta}{4}\left(1 - \frac{\frac{2n\theta^2}{\sigma_0^2(1+\sigma_0^2)}}{d/8} - \frac{\log(2)}{\log(3)}\right).
			\end{split}
		\end{equation*}
		Now we set $\theta = \left(\frac{\sigma_0^2(1+\sigma_0^2)}{2n}\cdot\left(\frac{d}{32}\right)\right)^{1/2} \wedge 1$. Then, for uniform constant $c>0$, we have
		\begin{equation*}
			\begin{split}
				\theta & \geq c\left(\sqrt{\frac{d}{n}}(\sigma_0 + \sigma_0^2) \wedge 1\right) \geq \frac{c\left(\sqrt{d\sigma_0^2}+\sqrt{d\sigma_0^2\cdot \sigma_0^2}\right)}{\sqrt{n}}\wedge 1\\ 
				& \overset{\eqref{ineq:PCA-lower-tilde-sigma-0}}{\geq} \frac{c\left(\sigsumch + \sigmaxch\sigsumch\right)}{\sqrt{n}}\wedge 1.
			\end{split}
		\end{equation*}
		Therefore,
		\begin{equation*}
			\begin{split}
				& \inf_{\widehat{U}}\sup_{\Sigma \in \mathcal{F}_{p,n, r}(\sigsumch, \sigmaxch, \nu)}\mathbb{E}\|\sin\Theta(\widehat{U}, U)\| \\
				\geq & c\left(\frac{\sigsumch+\sigmaxch\sigsumch}{\sqrt{n}} \wedge 1\right) \gtrsim \left(\frac{\sigsumch}{(n\nu)^{1/2}} + \frac{\sigsumch\sigmaxch}{n^{1/2}\nu}\right)\wedge 1,
			\end{split}
		\end{equation*}
		which has finished the proof for \eqref{ineq:PCA-lower-1}. 
	\end{itemize}

	The proof of \eqref{ineq:PCA-lower-2} is similar to \eqref{ineq:PCA-lower-1}: we still (a) first construct a series of candidate covariance matrices, (b) prove separateness of these covariance matrices and boundedness of KL divergence of random samples, and (c) apply generalized Fano's lemma to finalize the proof. 
	
	We still assume $\nu=1$ without loss of generality. Since $\sigmaxch \leq \sigsumch$, \eqref{ineq:PCA-lower-2} is directly implied by \eqref{ineq:PCA-lower-1} (which has been just proved) when $r$ is a constant. Thus, we can assume $r \geq 50$ in this part of proof without loss of generality. By the Varshamov-Gilbert bound \citep[Lemma 4.7]{massart2007concentration}, we can find $w^{(1)}, \ldots, w^{(N)}\subseteq \{\pm 1\}^r$, such that
	\begin{equation}\label{eq:u}
		\left\|w^{(l)} - w^{(k)}\right\|_2^2 \geq r \quad \text{for all } 1\leq k\neq l\leq N,
	\end{equation}
	and $N\geq \exp(r/8)$. Consider the following set of covariance matrices for $l=1,\ldots, N$,
	\begin{equation*}
		\begin{split}
			A^{(l)} = \begin{bmatrix}
				(\theta w^{(l)})^\top\\
				\frac{1}{\sqrt{L}}I_r\\
				\vdots\\
				\frac{1}{\sqrt{L}}I_r\\
				0_{(p-dr-1) \times r}
			\end{bmatrix}, \quad A^{(l)} = U^{(l)}R^{(l)}\text{ is the QR orthogonalization};
		\end{split}
	\end{equation*}
	$$\Sigma^{(l)} = A^{(l)}(A^{(l)})^\top + D \in \mathbb{R}^{p\times p}, \quad D_{ij}=\left\{\begin{array}{ll}
		\sigmaxch^2, & i=j=1;\\
		0, & \text{otherwise}.
	\end{array} \right.  $$
	Here, $L = \lceil 1/c_I \rceil$; $w^{(l)} \in \mathbb{R}^{r}$ has i.i.d.~Rademacher entries; $0<\theta\leq \sqrt{(c_I\wedge 1)/r}$ is some parameter to be determined later; $\frac{1}{\sqrt{L}}I_r$ is repeated for $L$ times; by design, the noise only appears in the first entry of the vector, so that the conditions 
	$$\max_i D_{ii} = D_{11} \leq \sigmaxch^2\quad \text{and}\quad \sum_{i=1}^p D_{ii} = D_{11} \leq \sigsumch^2$$
	naturally hold, provided that $\sigsumch\geq \sigmaxch$. 
	
	By the relationship between singular values of the matrix and its submatrices (see \cite[Lemma 2]{cai2018rate}), we have
	\begin{equation*}
		\begin{split}
			& \lambda_r\left(A^{(l)}\right) \geq \lambda_r\left(\begin{bmatrix}
				\frac{1}{\sqrt{L}} I_r\\
				\vdots\\
				\frac{1}{\sqrt{L}} I_r
			\end{bmatrix}\right) = 1, \\
			& \|A^{(l)}\|\leq \left(\|\theta w^{(l)}\|_2^2 + \left\|\begin{bmatrix}
				\frac{1}{\sqrt{L}} I_r\\
				\vdots\\
				\frac{1}{\sqrt{L}} I_r
			\end{bmatrix}\right\|^2\right)^{1/2},
		\end{split}
	\end{equation*}
	which means
	\begin{equation*}
		\begin{split}
			I(U^{(l)}) = & \frac{p}{r}\max_i \|e_i^\top U^{(l)}\|_2^2 \leq \frac{p}{r}\max_i\|e_i^\top A^{(l)} (R^{(l)})^{-1}\|_2^2 \\
			\leq & \frac{p}{r}\max_i \|e_i^\top A^{(l)}\|_2^2 \cdot \lambda_r^{-2}(R^{(l)}) \\
			\leq & \frac{p}{r}\max\left\{\theta^2 r, \frac{1}{L}\right\} \cdot \lambda_r^{-2}(A^{(l)}) \leq c_I p/r.
		\end{split}
	\end{equation*}
	Therefore,
	\begin{equation}\label{eq:Sigma-be-to-class}
		\Sigma^{(1)},\ldots, \Sigma^{(N)} \in \mathcal{F}_{p, n, r}(\sigsumch, \sigmaxch, \nu, \kappa).
	\end{equation}
	
	Again, suppose $X^{(l)} = [X^{(l)}_1,\ldots, X^{(l)}_n]\overset{iid}{\sim}N(0, \Sigma^{(l)})$ for $l=1,\ldots, N$. Next, we evaluate the $\sin\Theta$ distances between each pair of $(U^{(l)}, U^{(k)})$ and the KL divergence among $X^{(l)}$'s. Similarly to the proof for the first part of this theorem, we introduce a ``condensed version" of $A^{(l)}, \Sigma^{(l)},$ and $X^{(l)}$.
	\begin{equation*}
		\widetilde{A}^{(l)} = \begin{bmatrix}
			(\theta w^{(l)})^\top\\
			I_r
		\end{bmatrix}\in \mathbb{R}^{(r+1)\times r}, \quad \widetilde{A}^{(l)} = \widetilde{U}^{(l)}R^{(l)} \text{ is the QR decomposition},
	\end{equation*}
	\begin{equation*}
		\widetilde{\Sigma}^{(l)} = \widetilde{A}^{(l)}(\widetilde{A}^{(l)})^\top + \widetilde{D}\in \mathbb{R}^{(r+1)\times(r+1)}, \quad \widetilde{D}_{ij}=\left\{\begin{array}{ll}
			\sigmaxch^2, & i=j=1;\\
			0, & \text{otherwise},
		\end{array} \right.
	\end{equation*}
	\begin{equation*}
		\begin{split}
			& \widetilde{X}^{(l)} = [\widetilde{X}_1^{(l)},\ldots, \widetilde{X}_n^{(l)}], \quad X_i^{(l)} = T\widetilde{X}_i^{(l)}\in \mathbb{R}^{r+1},\\ 
			& \text{where } T = \begin{bmatrix}
				1 & 0_{1\times r}\\
				0_{r\times 1} & \frac{1}{\sqrt{L}}I_r\\
				\vdots & \vdots\\
				0_{r\times 1} & \frac{1}{\sqrt{L}}I_r\\
				0_{(p-Lr-1)\times 1} & 0_{(p-Lr-1)\times r}
			\end{bmatrix}.
		\end{split}
	\end{equation*}
	Then, $\widetilde{A}^{(l)}$, $A^{(l)}$, $\widetilde{U}^{(l)}$, and $U^{(l)}$ can be similarly related via $T$,
	\begin{equation}\label{eq:T}
		T\widetilde{A}^{(l)} = A^{(l)},\quad T\widetilde{U}^{(k)} = U^{(k)}.
	\end{equation}
	One can also verify that $\widetilde{X}^{(l)} \overset{iid}{\sim}N(0, \widetilde\Sigma^{(l)})$. Noting that
	\begin{equation*}
		v^{(l)} = \frac{1}{\sqrt{1 + r\theta^2}}\begin{pmatrix}
			1\\ 
			-\theta w^{(l)}
		\end{pmatrix} \in \mathbb{R}^{r+1}
	\end{equation*}
	is the orthogonal complement to $\widetilde{A}^{(l)}$, we have
	\begin{equation*}
		\begin{split}
			& \left\|\sin\Theta(\widetilde{U}^{(k)}, \widetilde{U}^{(l)})\right\| = \left\|(v^{(l)})^\top\widetilde{U}^{(k)}\right\| = \left\|(v^{(l)})^\top\widetilde{A}^{(k)}(R^{(l)})^{-1}\right\| \\
			\geq & \left\|(v^{(l)})^\top\widetilde{A}^{(k)}\right\| \cdot \lambda_r^{-1}(A^{(l)}) \geq \left\|\frac{\theta w^{(l)} - \theta w^{(k)}}{\sqrt{1+r\theta^2}}\right\|_2\frac{1}{\sqrt{1+r\theta^2}} \\
			= & \frac{\theta}{1+r\theta^2}\left\|w^{(l)} - w^{(k)}\right\|_2.
		\end{split}
	\end{equation*}
	Since $0\leq \theta \leq \sqrt{(c_I\wedge 1)/r}$, we additionally have
	\begin{equation}\label{ineq:sin-Uk-Ul}
		\begin{split}
			& \left\|\sin(U^{(k)}, U^{(l)})\right\| \overset{\eqref{eq:T}}{=}  \left\|\sin(\widetilde{U}^{(k)}, \widetilde{U}^{(l)})\right\| \geq \frac{\theta}{1+r\theta^2}\left\|w^{(l)} - w^{(k)}\right\|_2\\
			\geq & \frac{\theta}{2}\left\|w^{(l)} - w^{(k)}\right\|_2 \overset{\eqref{eq:u}}{\geq} \frac{\sqrt{r}\theta}{2}, \quad \text{for all } 1\leq k\neq l \leq N.
		\end{split}
	\end{equation}
	
	Next, we consider the KL divergence among these samples. Given the linear relationship $X^{(l)} = T\widetilde{X}^{(l)}, X^{(k)} = T\widetilde{X}^{(k)}$ with non-singular map $T$, we have
	\begin{equation*}
		\begin{split}
			& D_{KL}\left(X^{(l)}||X^{(k)}\right) = D_{KL}\left(\widetilde{X}^{(l)}||\widetilde{X}^{(k)}\right)\\
			= &  \frac{n}{2}\left(\tr\left((\widetilde\Sigma^{(k)})^{-1}\widetilde\Sigma^{(l)}\right) - (r+1) + \log\left(\frac{\det(\widetilde\Sigma^{(k)})}{\det(\widetilde\Sigma^{(l)})}\right)\right).
		\end{split}
	\end{equation*}
	Noting that
	\begin{equation*}
		\widetilde{\Sigma}^{(k)} = 
		\begin{bmatrix}
			\theta^2r + \sigmaxch^2 & \theta (w^{(k)})^\top\\
			\theta w^{(k)} & I_r
		\end{bmatrix}, 
	\end{equation*}
	$\det(\widetilde\Sigma^{(k)}) = \det(\widetilde\Sigma^{(l)})$ by symmetry. By the matrix inversion formula and calculation, one has
	\begin{equation*}
		\begin{split}
			(\widetilde{\Sigma}^{(k)})^{-1} = &\begin{bmatrix}
				\sigmaxch^{-2} & -\sigmaxch^{-2}\nu(w^{(k)})^\top \\\
				-\sigmaxch^{-2}\nu w^{(k)}  & I_r + \sigmaxch^{-2}\nu^2w^{(k)} (w^{(k)})^\top
			\end{bmatrix},
		\end{split}
	\end{equation*}
	and
	\begin{equation}\label{ineq:KL-X^l-X^k}
		\begin{split}
			& D_{KL}\left(X^{(l)}||X^{(k)}\right) = \frac{n}{2}\left(\tr\left((\widetilde{\Sigma}^{(k)})^{-1}\widetilde{\Sigma}^{(l)}\right) - (r+1)\right) \\
			= & \frac{n}{2}\left((r+1) + 2\sigmaxch^{-2}\theta^2\left(r-(w^{(k)})^\top w^{(l)}\right) - (r+1)\right) \\
			= & \frac{n}{2}\left(2\sigmaxch^{-2}\theta^2(r-(w^{(k)})^\top w^{(l)})\right)\\
			\leq & n\sigmaxch^{-2}\theta^2 r.
		\end{split}
	\end{equation}	
	
	Finally, by generalized Fano's lemma \cite[Lemma 3]{yu1997assouad},
	\begin{equation*}
		\begin{split}
			& \inf_{\widehat{U}}\sup_{\Sigma\in \mathcal{F}_{p,n,r}(\sigsumch, \sigmaxch, \nu, \kappa)}\left\|\sin\Theta(\widehat{U}, U)\right\|\\ \overset{\eqref{eq:Sigma-be-to-class}}{\geq} &  \inf_{\widehat{U}}\sup_{\Sigma\in \{\Sigma_k\}_{k=1}^{N}}\left\|\sin\Theta(\widehat{U}, U)\right\| \overset{\eqref{ineq:sin-Uk-Ul}\eqref{ineq:KL-X^l-X^k}}{\geq}  \frac{\sqrt{r}\theta}{4}\left(1 - \frac{n\sigmaxch^{-2}\theta^2 r+\log(2)}{r/8}\right).
		\end{split}
	\end{equation*}
	Set $\theta = \sigmaxch/(32\sqrt{n}) \wedge \sqrt{(c_I\wedge 1)/r}$. Given $r\geq 50$, we have
	$$1 - \frac{n\sigmaxch^{-2}\theta^2 r +\log(2)}{r/8} \geq 1 - \frac{r/32 + \log(2)}{r/8}\geq 1/3,$$
	which means
	\begin{equation*}
		\inf_{\widehat{U}}\sup_{\Sigma\in \mathcal{F}_{p,n,r}(\sigsumch, \sigmaxch, \nu, \kappa)}\left\|\sin\Theta(\widehat{U}, U)\right\| \gtrsim c\left(\sqrt{\frac{r}{n}}\sigmaxch\wedge 1\right) = c\left(\frac{r^{1/2}\sigmaxch}{(n\nu)^{1/2}}\wedge 1\right).
	\end{equation*}
	for some constant $c>0$ that only relies on $c_I$.
	Thus, we have finished the proof for \eqref{ineq:PCA-lower-2}. 
\end{proof}

\begin{proof}[Proof of Proposition \ref{pr:heteroPCA-approximate}]
	Since $\widehat{\Sigma}$ is invariant after translation on $Y$, we can assume that the mean vector $\mu=0$ without loss of generality. Let
	$$\Sigma_0 = \widetilde{U}\Lambda \widetilde{U}^\top = \begin{bmatrix}
		U ~ U_\perp
	\end{bmatrix} \begin{bmatrix}
		\Lambda_1 & \\
		& \Lambda_2
	\end{bmatrix} \begin{bmatrix}
		U^\top\\
		U_\perp^\top
	\end{bmatrix}$$
	be the full eigenvalue decomposition of $\Sigma_0$. Here, $\widetilde{U} = [U ~ U_\perp]$ is the $p$-by-$p$ orthogonal matrix comprised of all eigenvectors of $\Sigma_0$, $\widetilde{U}= [U ~ U_{\perp}]$, $\Lambda_1$ and $\Lambda_2$ are $r$-by-$r$ and $(p-r)$-by-$(p-r)$ non-negative diagonal matrices containing the first $r$ and the other $(p-r)$ eigenvalues of $\Sigma_0$, respectively. We can also decompose $Y_k$ based on its principal components as
	$$Y_k = X_k + \varepsilon_k = U\Lambda_1^{1/2} \gamma_{1k} + U_\perp \Lambda_2^{1/2}\gamma_{2k} + \varepsilon_k,$$
	where the random scores satisfy $\mathbb{E}(\gamma_{1k}^\top, \gamma_{2k}^\top) = 0, \Cov((\gamma_{1k}^\top, \gamma_{2k}^\top)) = I$. We can further write this decomposition in a matrix form,
	$$\Gamma_1 = [\gamma_{11} \cdots \gamma_{1n}], \quad \Gamma_2 = [\gamma_{21} \cdots \gamma_{2n}], \quad \Gamma = \begin{bmatrix}
		\Gamma_1\\
		\Gamma_2
	\end{bmatrix} = \begin{bmatrix}
		\gamma_{11} \cdots \gamma_{1n}\\
		\gamma_{21} \cdots \gamma_{2n}
	\end{bmatrix},$$
	$$Y = X^{(1)} + X^{(2)} + E, \quad  X^{(1)} = U\Lambda^{1/2}_1\Gamma_1, \quad X^{(2)} =   U_\perp\Lambda^{1/2}_2\Gamma_2.$$
	We divide the rest of the proof in three steps.
	\begin{itemize}[leftmargin=*]
		\item[Step 1] Define $\widehat{\Sigma}_X = (XX^\top - n\bar{X}\bar{X}^\top)/(n-1)$ and $\widehat{\Sigma}_{X^{(1)}} = (X^{(1)}X^{(1)\top} - n\bar{X}^{(1)}\bar{X}^{(1)\top})/(n-1)$. 
		By the same argument as the proof of Theorem \ref{th:heterogeneous-PCA}, we can prove the following average perturbation inequality for $\widehat{\Sigma} - \widehat{\Sigma}_X$,
		\begin{equation*}
			\mathbb{E}_E\left\|\Delta\left((n-1)(\widehat{\Sigma} - \widehat{\Sigma}_X)\right)\right\| \lesssim \sqrt{n}\sigsum\sigmax + \sigsum^2 + \|X\|(\sigsum + \sqrt{r}\sigmax) + n^{1/2}\|\bar{X}\|_2\sigsum. 
		\end{equation*}
		Here, $\mathbb{E}_E$ means the expectation with respect to the noise part $E$. In addition, we can decompose $(n-1)(\widehat{\Sigma} - \widehat{\Sigma}_{X^{(1)}})$ in the similar way as \eqref{eq:hetero-PCA-upper-1}:
		\begin{equation*}
			\begin{split}
				(n-1)(\widehat{\Sigma}_X - \widehat{\Sigma}_{X^{(1)}}) = & X^{(1)}X^{(2)\top} + X^{(2)}X^{(1)\top} + X^{(2)}X^{(2)\top} \\
				& - n\left(\bar{X}^{(1)}\bar{X}^{(2)\top} + \bar{X}^{(2)}\bar{X}^{(1)\top} + \bar{X}^{(2)}\bar{X}^{(2)\top}\right).
			\end{split}
		\end{equation*}
		Therefore,
		\begin{equation}\label{ineq:hat_Sigma-hat_Sigma_X}
			\begin{split}
				& \mathbb{E}_E\left\|\Delta(n\widehat{\Sigma} - n\widehat{\Sigma}_X)\right\| \lesssim  \sqrt{n}\sigsum\sigmax + \sigsum^2 + \|X\|(\sigsum + \sqrt{r}\sigmax) \\
				& + n^{1/2}\|\bar{X}\|_2\sigsum + 2\|X^{(1)}X^{(2)\top}\| + \|X^{(2)}\|^2 + 2n\|\bar{X}^{(1)}\|_2 \|\bar{X}^{(2)}\|_2 + n\|\bar{X}^{(2)}\|_2^2.
			\end{split}
		\end{equation}
		Noting that $\widehat{\Sigma}^{(1)}$ is rank-$r$ and has singular subspace $U$, by the robust $\sin\Theta$ theorem (Theorem \ref{th:diagonal-less}),
		\begin{equation}\label{proposition-E}
			\begin{split}
				& \mathbb{E}_E\left\|\sin\Theta(\widehat{U}, U)\right\| \lesssim \Bigg(\frac{\sqrt{n}\sigsum\sigmax + \sigsum^2 + \|X\|(\sigsum + \sqrt{r}\sigmax) + n^{1/2}\|\bar{X}\|_2\sigsum}{(n-1)\lambda_r(\widehat{\Sigma}^{(1)})} \\
				& \quad +\frac{2\|X^{(1)}X^{(2)\top}\| + \|X^{(2)}\|^2 + 2n\|\bar{X}^{(1)}\|_2 \|\bar{X}^{(2)}\|_2 + n\|\bar{X}^{(2)}\|_2^2}{(n-1)\lambda_r(\widehat{\Sigma}^{(1)})}\Bigg) \wedge 1.
			\end{split}
		\end{equation}
		We analyze each term above as follows. Specifically, we introduce the following desirable probability event $\mathcal{A}$, which happens if the inequalities \eqref{ineq:probability-1} -- \eqref{ineq:probability-3} all hold:
		\begin{equation}\label{ineq:probability-1}
			\begin{split}
				& \sqrt{n} + C\sqrt{p} \geq \lambda_1 \left(\begin{bmatrix}
					\Gamma_1\\
					\Gamma_2
				\end{bmatrix}\right), \lambda_1 \left(\begin{bmatrix}
					\Gamma_1\\
					- \Gamma_2
				\end{bmatrix}\right), \lambda_p \left(\begin{bmatrix}
					\Gamma_1\\
					\Gamma_2
				\end{bmatrix}\right), \lambda_p \left(\begin{bmatrix}
					\Gamma_1\\
					-\Gamma_2
				\end{bmatrix}\right) \\
				& \geq (\sqrt{n} - C\sqrt{p})\vee 0,
			\end{split}
		\end{equation}
		\begin{equation}\label{ineq:probability-2}
			\lambda_r\left(\widehat{\Sigma}_{X^{(1)}}\right) \geq \frac{5}{36}\lambda_r(\Lambda) ,\quad \|X^{(1)}\| \leq  2\sqrt{n}\|\Lambda^{1/2}\|, \quad  \|\sqrt{n}\bar{\Gamma}^{(1)}\|_2\leq \sqrt{n}/3, 
		\end{equation}
		\begin{equation}\label{ineq:probability-3}
			\|\Gamma_2\| \leq C(\sqrt{n}+\sqrt{p}),\quad \|\bar{\Gamma}^{(2)}\|_2 \leq C\sqrt{p/n}.
		\end{equation}
		In the next two steps, we analyze each term in \eqref{proposition-E} given $\mathcal{A}$ holds, then evaluate the probability that $\mathcal{A}$ holds.
		\item[Step 2] Now we assume $\mathcal{A}$ happens and \eqref{ineq:probability-1}--\eqref{ineq:probability-3} all hold. We plug in $X^{(1)} = U\Lambda_1^{1/2}\Gamma_1$ and $X^{(2)} = U_{\perp}\Lambda_2^{1/2}\Gamma_2$ and obtain
		\begin{equation*}
			\begin{split}
				& \|X^{(1)}X^{(2)\top}\| = \|X^{(2)}X^{(1)\top}\|\leq \|\Lambda_1^{1/2}\|\|\Lambda_2^{1/2}\|\|\Gamma_1\Gamma_2^{\top}\| \\
				= &  \frac{1}{2}\lambda_1^{1/2}(\Lambda)\lambda_{r+1}^{1/2}(\Lambda) \left\|\begin{bmatrix}
					\Gamma_1\\
					\Gamma_2\\
				\end{bmatrix} \begin{bmatrix}
					\Gamma_1^\top ~\Gamma_2^\top 
				\end{bmatrix} - \begin{bmatrix}
					\Gamma_1\\
					-\Gamma_2\\
				\end{bmatrix} \begin{bmatrix}
					\Gamma_1^\top ~ - \Gamma_2^\top 
				\end{bmatrix}\right\|,
			\end{split}
		\end{equation*}
		\begin{equation*}
			\begin{split}
				& \left\|\begin{bmatrix}
					\Gamma_1\\
					\Gamma_2
				\end{bmatrix}\begin{bmatrix}
					\Gamma_1^\top ~\Gamma_2^\top 
				\end{bmatrix} - \begin{bmatrix}
					\Gamma_1\\
					-\Gamma_2
				\end{bmatrix} \begin{bmatrix}
					\Gamma_1^\top ~ - \Gamma_2^\top 
				\end{bmatrix}\right\| \\
				= & \max_{w \in \mathbb{R}^p: \|w\|_2\leq 1}\left|\left([\Gamma_1^\top ~\Gamma_2^\top]w\right)^2 - \left([\Gamma_1^\top ~ -\Gamma_2^\top]w\right)^2\right|\\
				\leq & \max\left\{\lambda_1^2\left(\begin{bmatrix}
					\Gamma_1\\
					\Gamma_2
				\end{bmatrix}\right), \lambda_1^2\left(\begin{bmatrix}
					\Gamma_1\\
					-\Gamma_2
				\end{bmatrix}\right)\right\} - \min\left\{\lambda_p^2\left(\begin{bmatrix}
					\Gamma_1\\
					\Gamma_2
				\end{bmatrix}\right), \lambda_p^2\left(\begin{bmatrix}
					\Gamma_1\\
					-\Gamma_2
				\end{bmatrix}\right)\right\} \\
				\overset{\eqref{ineq:probability-1}}{\leq} & \left(\sqrt{n} + C\sqrt{p}\right)^2 - \left\{\left(\sqrt{n} - C\sqrt{p}\right) \vee 0\right\}^2 \leq C(\sqrt{np} + p).
			\end{split}
		\end{equation*}
		Thus,
		\begin{equation*}
			\|X^{(1)}X^{(2)\top}\| \leq C(\sqrt{np}+p)\lambda_1^{1/2}(\Lambda)\lambda_{r+1}^{1/2}(\Lambda) \lesssim (\sqrt{np}+p)\lambda_r^{1/2}(\Lambda)\lambda_{r+1}^{1/2}(\Lambda).
		\end{equation*}
		Similarly,
		\begin{equation*}
			\|X^{(2)}X^{(2)\top}\| = \|\Lambda_2^{1/2} \Gamma_2\Gamma_2^{\top}\Lambda_2^{1/2}\| = \|\Lambda_2\| \|\Gamma_2\|^2 \overset{\eqref{ineq:probability-3}}{\leq} C(n+p)\lambda_{r+1}(\Lambda); 
		\end{equation*}
		
		\begin{equation*}
			\begin{split}
				& 2n\|\bar{X}^{(1)}\|_2\|\bar{X}^{(2)}\|_2 + n \|\bar{X}^{(2)}\|_2^2\\
				= & 2n\|U\Lambda_1 \bar{\Gamma}_1\|_2 \|U_\perp\Lambda_2 \bar{\Gamma}_2\|_2 + n\|U\Lambda_1\bar{\Gamma}_2\|_2^2\\
				\leq & 2n\lambda_1^{1/2}(\Lambda)\lambda_{r+1}^{1/2}(\Lambda)\|\bar{\Gamma}_1\|_2 \|\bar{\Gamma}_2\|_2 + n\lambda_{r+1}(\Lambda) \|\bar{\Gamma}_2\|_2^2\\
				\overset{\eqref{ineq:probability-2}\eqref{ineq:probability-3}}{\lesssim} & n\lambda_r^{1/2}(\Lambda)\lambda_{r+1}^{1/2}(\Lambda)\sqrt{p/n} +n \lambda_{r+1}(\Lambda) p/n,
			\end{split}
		\end{equation*}
		\begin{equation*}
			\|\bar{X}^{(1)}\|_2 = \|U\Lambda_1\bar{\Gamma}_1\|_2 \overset{\eqref{ineq:probability-2}}{\leq} \lambda_1^{1/2}(\Lambda)\|\bar{\Gamma}^{(1)}\|_2 \lesssim \lambda_r^{1/2}(\Lambda),
		\end{equation*}
		\begin{equation*}
			\|\widehat{X}^{(2)}\|_2 \leq \|\Lambda_2^{1/2}\| \|\bar{\Gamma}_2\|_2 \lesssim \lambda_{r+1}(\Lambda) \sqrt{p/n},
		\end{equation*}
		\begin{equation*}
			\|\bar{X}\|_2 \leq \|\bar{X}^{(1)}\|_2 + \|\bar{X}^{(2)}\|_2 \leq \lambda_r^{1/2}(\Lambda) + \lambda_{r+1}^{1/2}(\Lambda)\sqrt{p/n}.
		\end{equation*}
		Summarizing all previous bounds, when $\mathcal{A}$ holds, we have
		\begin{equation*}
			\begin{split}
				& \mathbb{E}_E\left\|\sin\Theta(\widehat{U}, U)\right\| \lesssim \Bigg(\frac{\sqrt{n}\sigsum\sigmax + \sigsum^2 + \|X\|(\sigsum + \sqrt{r}\sigmax) + n^{1/2}\|\bar{X}\|_2\sigsum}{n\lambda_r(\widehat{\Sigma}_{X^{(1)}})} \\
				& \quad +\frac{2\|X^{(1)}X^{(2)\top}\| + \|X^{(2)}\|^2 + 2n\|\bar{X}^{(1)}\|_2 \|\bar{X}^{(2)}\|_2 + n\|\bar{X}^{(2)}\|_2^2}{n\lambda_r(\widehat{\Sigma}_{X^{(1)}})}\Bigg) \wedge 1\\
				\lesssim & \Bigg(\frac{\sqrt{n}\sigsum\sigmax+\sigsum^2+(\sqrt{n}\lambda_r^{1/2} + \sqrt{p}\lambda_{r+1}^{1/2}(\Lambda))(\sigsum+\sqrt{r}\sigmax)}{n\lambda_r(\Lambda)}\\
				& + \frac{ n^{1/2}\sigsum(\lambda_r^{1/2}(\Lambda)+\lambda_{r+1}^{1/2}(\Lambda)\sqrt{p/n})}{n\lambda_r(\Lambda)}\\
				& + \frac{ (\sqrt{np}+p)\lambda_r^{1/2}(\Lambda)\lambda_{r+1}^{1/2}(\Lambda) + (n+p)\lambda_{r+1}(\Lambda) + n\lambda^{1/2}_r(\Lambda)\lambda_{r+1}^{1/2}(\Lambda)\sqrt{p/n} + n\lambda_{r+1}(\Lambda)p/n }{n\lambda_r(\Lambda)}\Bigg)\wedge 1\\
				\lesssim & \left(\frac{\sigsum+\sqrt{r}\sigmax}{n^{1/2}\lambda_r^{1/2}(\Lambda)} + \frac{\sigsum\sigmax}{n^{1/2}\lambda_r(\Lambda)} + \frac{(\sqrt{np}+p)\lambda_{r+1}^{1/2}(\Lambda)}{n\lambda_r^{1/2}(\Lambda)} + \frac{\lambda_{r+1}(\Lambda)}{\lambda_r(\Lambda)}\right)\wedge 1.
			\end{split}
		\end{equation*}
		Here, the penultimate inequality is due to the following facts:
		\begin{itemize}[leftmargin=*]
			\item  $\frac{\sigsum^2}{n\lambda_r(\Lambda)}\wedge 1 \leq \frac{\sigsum}{n^{1/2}\lambda_r^{1/2}(\Lambda)}\wedge 1$;
			\item Since $ab\wedge 1 \leq (a+b)\wedge 1$ for any $a, b\geq 0$, $$\frac{\sqrt{p}\lambda_{r+1}^{1/2}(\Lambda)(\sigsum+\sqrt{r}\sigmax)}{n\lambda_r(\Lambda)}\wedge 1 \leq \left(\frac{p^{1/2}\lambda_{r+1}^{1/2}(\Lambda)}{n^{1/2}\lambda_r^{1/2}(\Lambda)} + \frac{\sigsum+\sqrt{r}\sigmax}{n^{1/2}\lambda^{1/2}_r(\Lambda)}\right)\wedge 1; $$
			\item By the same reason above,
			$$\frac{n^{1/2}\sigsum\lambda_{r+1}^{1/2}(\Lambda)\sqrt{p/n}}{n\lambda_r(\Lambda)} \wedge 1 \leq \left(\frac{\sigsum}{n^{1/2}\lambda_r^{1/2}(\Lambda)}+\left(\frac{p\lambda_{r+1}(\Lambda)}{n\lambda_r(\Lambda)}\right)^{1/2}\right)\wedge1,$$
			\item $	\left(\frac{p\lambda_{r+1}(\Lambda)}{n\lambda_r(\Lambda)}\right)\wedge 1 \leq \left(\frac{p\lambda_{r+1}(\Lambda)}{n\lambda_r(\Lambda)}\right)^{1/2}\wedge 1,$
		\end{itemize}
		\item[Step 3] In this step, we evaluate the probability that the event $\mathcal{A}$ holds by giving probability upper bounds for \eqref{ineq:probability-1} -- \eqref{ineq:probability-3} as follows.
		\begin{itemize}[leftmargin=*]
			\item Noting that $\begin{bmatrix}
				\Gamma_1\\ \Gamma_2
			\end{bmatrix}\in \mathbb{R}^{p\times n}$ and $\begin{bmatrix}
				\Gamma_1\\ -\Gamma_2
			\end{bmatrix}\in \mathbb{R}^{p\times n}$ are random matrices with i.i.d. columns, by \cite[Corollary 5.35]{vershynin2010introduction}, 
			\begin{equation*}
				\begin{split}
					& \bbP\left(\sqrt{n} + C\sqrt{p} + t \geq 
					\text{all singular values of} \begin{bmatrix}
						\Gamma_1 \\
						\Gamma_2
					\end{bmatrix}\begin{bmatrix}
						\Gamma_1 \\
						- \Gamma_2
					\end{bmatrix} \geq \sqrt{n} - C\sqrt{p} - t\right) \\
					\leq & \exp(-Ct^2/2).
				\end{split}
			\end{equation*}
			By setting $t = C\sqrt{p}$ for large constant $C>0$, we know \eqref{ineq:probability-1} holds with probability at least $1 - C\exp(-Cp)$. \item Since $\Gamma_1\in \mathbb{R}^{r\times n}$ has isotropic sub-Gaussian columns, 
			based on the argument of \eqref{ineq:hetero-PCA-upper-6} in the proof of Theorem \ref{th:heterogeneous-PCA},
			\eqref{ineq:probability-2} holds with probability at least $1 - C\exp(-Cn)$.
			\item Noting that $\Gamma_2$ is a $(p-r)$-by-$n$ random matrix with i.i.d. isotropic sub-Gaussian rows, by \cite[Corollary 5.35]{vershynin2011spectral}, 
			\begin{equation*}
				\|\Gamma_2\| \leq C(\sqrt{n}+\sqrt{p})
			\end{equation*}
			with probability at least $1 - \exp(-C(n+p))$; by Bernstein-type concentration inequality \cite[Proposition 5.16]{vershynin2011spectral},
			\begin{equation*}
				\bbP\left(\|\sqrt{n}\bar{\Gamma}^{(2)}\|_2^2 \geq p + C\sqrt{px}+Cx\right) \leq C\exp(-cx).
			\end{equation*}
			By setting $x = Cp$, we conclude that \eqref{ineq:probability-3} holds with probability at least $1 - C\exp(-Cp)$.
		\end{itemize}
		To sum up, the event $\mathcal{A}$ happens, i.e., \eqref{ineq:probability-1} -  \eqref{ineq:probability-3} all hold, with probability at least $1 - C\exp(-Cn) - C\exp(-Cp)$. 
		\item[Step 4] We finalize the proof in this step.
		\begin{equation*}
			\begin{split}
				& \mathbb{E}\left\|\sin\Theta(\widehat{U}, U)\right\| =  \mathbb{E}\left\|\sin\Theta(\widehat{U}, U)\right\|1_{\mathcal{A}} + \mathbb{E}\left\|\sin\Theta(\widehat{U}, U)\right\|1_{\mathcal{A}^c}\\
				\lesssim & \frac{\sigsum+\sqrt{r}\sigmax}{n^{1/2}\lambda_r^{1/2}(\Lambda)} + \frac{\sigsum\sigmax}{n^{1/2}\lambda_r(\Lambda)} + \frac{((np)^{1/2}+p)\lambda_{r+1}^{1/2}(\Lambda)}{n\lambda_r^{1/2}(\Lambda)} + \frac{\lambda_{r+1}(\Lambda)}{\lambda_r(\Lambda)} \\
				& + C\exp(-Cn) + C\exp(-Cp)\\
				\lesssim & \frac{\sigsum+\sqrt{r}\sigmax}{n^{1/2}\lambda_r^{1/2}(\Lambda)} + \frac{\sigsum\sigmax}{n^{1/2}\lambda_r(\Lambda)} + \frac{((np)^{1/2}+p)\lambda_{r+1}^{1/2}(\Lambda)}{n\lambda_r^{1/2}(\Lambda)} + \frac{\lambda_{r+1}(\Lambda)}{\lambda_r(\Lambda)},
			\end{split}
		\end{equation*}
		where the last inequality is due to $\sigsum^2/\lambda_r(\Lambda)\geq \exp(-Cp)+\exp(-Cn)$ in the assumption. Finally, the trivial upper bound 1 always holds for $\mathbb{E}\left\|\sin\Theta(\widehat{U}, U)\right\|$. We thus have finished this proof. 
	\end{itemize}
\end{proof}

\subsection{Proofs in Heteroskedastic Low-rank Matrix Denoising}

\begin{proof}[Proof of Theorem \ref{th:upper_bound_SVD}]
	First, we assume $\delta\in \mathbb{R}^{p_1}, \delta_i = \sum_{j=1}^{p_2} \sigma_{ij}^2$ as the row-wise summation of variances. Note that 
	\begin{equation*}
		YY^\top = XX^\top + X E^\top + EX^\top + EE^\top.
	\end{equation*}
	Then, $\mathbb{E}YY^\top = XX^\top + \diag(\delta)$. By the Wishart-type heteroskedastic concentration inequality \cite[Corollary 1]{cai2020non},
	\begin{equation}\label{ineq:thm2-middle-1}
		\begin{split}
			& \mathbb{E}\left\|EE^\top - \diag(\delta)\right\| \lesssim \sigma_C^2+\sigma_C\sigma_R + \sigma_R\sigmax\sqrt{\log(p_1\wedge p_2)} + \sigmax^2\log(p_1\wedge p_2).\\
		\end{split}
	\end{equation}
	By Lemma \ref{lm:orthogonal-projection} and $\|X\| \leq C\lambda_r(X)$, 
	\begin{equation}\label{ineq:thm2-middle-1.5}
		\mathbb{E}\left\|XE^\top\right\| \lesssim \|X\|\left(\sigma_C + \sqrt{r}\sigmax\right) \lesssim \lambda_r(X)\left(\sigma_C + \sqrt{r}\sigmax\right).
	\end{equation}
	By Lemma \ref{lm:diagonal-less-spectral-norm}, 
	\begin{equation}\label{ineq:thm2-middle-2}
		\begin{split}
			& \left\|\Delta(YY^\top - XX^\top)\right\| = \left\|\Delta(YY^\top - XX^\top - \diag(\delta))\right\|\\
			\leq & 2\left\|YY^\top - XX^\top - \diag(\delta)\right\| \leq \left\|XE^\top + EX^\top + EE^\top - \diag(\delta)\right\| \\ 
			\leq & 2\left\|EE^\top - \diag(\delta)\right\| + 4\|EX^\top\|.
		\end{split}
	\end{equation}
	Combining \eqref{ineq:thm2-middle-1}, \eqref{ineq:thm2-middle-1.5}, and \eqref{ineq:thm2-middle-2}, we have
	\begin{equation}\label{ineq:thm2-middle-3}
		\begin{split}
			& \mathbb{E}\left\|\Delta(YY^\top - XX^\top)\right\| \\
			\lesssim & \sigma_C^2+\sigma_C\sigma_R + \sigma_R\sigmax\sqrt{\log(p_1\wedge p_2)} + \sigmax^2\log(p_1\wedge p_2) + \lambda_r(X)\left(\sigma_C + \sqrt{r}\sigmax\right).
		\end{split}
	\end{equation}
	Note that the eigen-subspace of $XX^\top$ is the same as $U$, i.e., the left singular subspace of $X$. Since $I(U) \leq c_Ip/r$, the robust $\sin\Theta$ theorem (Theorem \ref{th:diagonal-less}) implies
	\begin{equation}\label{ineq:thm2-middle-4}
		\begin{split}
			& \mathbb{E}\left\|\sin\Theta\left(\widehat{U}, U\right)\right\| \leq \frac{C\mathbb{E}\|\Delta(YY^\top - XX^\top)\|}{\lambda_r^2(X)} \wedge 1\\
			\lesssim & \left(\frac{\sigma_C^2+\sigma_C\sigma_R + \sigma_R\sigmax\sqrt{\log(p_1\wedge p_2)} + \sigmax^2\log(p_1\wedge p_2) + \lambda_r(X)\left(\sigma_C + \sqrt{r}\sigmax\right)}{\lambda_r^2(X)}\right)\wedge 1\\
			\lesssim & \left( \frac{\sigma_C + \sqrt{r}\sigmax}{\lambda_r(X)} + \frac{\sigma_C\sigma_R + \sigma_R\sigmax\sqrt{\log(p_1\wedge p_2)} + \sigmax^2\log(p_1\wedge p_2)}{\lambda_r^2(X)}\right)\wedge 1.
		\end{split}
	\end{equation}
	The last inequality is due to the fact that $\sigma_C/\lambda_r(X)\wedge 1 \geq  \sigma_C^2/\lambda_r^2(X)\wedge 1$. In particular when $\sigmax\lesssim \sigma_C/\max\{\sqrt{r}, \sqrt{\log(p_1\wedge p_1)}\}$, we have 
	$$\frac{\sqrt{r}\sigmax}{\lambda_r(X)} \lesssim \frac{\sigma_C}{\lambda_r(X)}, \quad \frac{\sigma_R\sigmax\sqrt{\log(p_1\wedge p_2)}}{\lambda_r^2(X)} \lesssim \frac{\sigma_C\sigma_R}{\lambda_r^2(X)},$$  
	$$\frac{\sigmax^2\log(p_1\wedge p_2)}{\lambda_r^2(X)}\wedge 1 \lesssim \frac{\sigma_C^2}{\lambda_r^2(X)}\wedge 1 \leq \frac{\sigma_C}{\lambda_r(X)}\wedge 1.$$ 
	We thus have
	\begin{equation*}
		\mathbb{E}\left\|\sin\Theta(\widehat{U}, U)\right\| \lesssim \left(\frac{\sigma_C}{\lambda_r(X)} + \frac{\sigma_C\sigma_R}{\lambda_r^2(X)}\right)\wedge 1.
	\end{equation*}
\end{proof}

\subsection{Proofs in Poisson PCA}
\begin{proof}[Proof of Theorem \ref{th:poisson}]
	Denote $E = Y - X\in \mathbb{R}^{p_1\times p_2}$. Recall the following tail probability bound of Poisson distribution (see, e.g., \cite[Pages 22-23]{boucheron2013concentration}),
	\begin{equation*}
		\begin{split}
			& \bbP\left(|Y_{ij} - X_{ij}| \geq t\right) \leq 2\exp\left(-(t+X_{ij}) \log\left(1 + t/X_{ij}\right) + t\right), \quad \forall t\geq 0.
		\end{split}
	\end{equation*}
	Next, we aim to show
	\begin{equation}\label{ineq:Poisson-heavy-tail}
		\begin{split}
			\bbP\left(|Y_{ij} - X_{ij}|\geq t\right) \leq 2\exp\left(1 - ct/\sqrt{X_{ij}}\right), \quad \forall t>0.
		\end{split}
	\end{equation}
	for some uniform constant $c>0$.
	\begin{itemize}[leftmargin=*]
		\item If $t < \sqrt{X_{ij}}$, 
		\begin{equation*}
			\bbP\left(|Y_{ij} - X_{ij}|\geq t\right)\leq 1 \leq 2\exp(1 - t/\sqrt{X_{ij}}), \quad \forall t>0.
		\end{equation*}
		\item If $\sqrt{X_{ij}}\leq t\leq X_{ij}/2$, by Taylor expansion for $\log(1 + x/X_{ij})$,
		\begin{equation*}
			\begin{split}
				& (X_{ij} + t) \log\left(1 + \frac{t}{X_{ij}}\right) - t \geq (X_{ij} + t) \left(\frac{t}{X_{ij}} - \frac{t^2}{2X_{ij}^2}\right) - t\\
				= & \frac{t^2}{X_{ij}} - \frac{t^2}{X_{ij}} \cdot \frac{X_{ij}+t}{2X_{ij}} \geq \frac{t^2}{X_{ij}} -  \frac{t^2}{X_{ij}} \cdot \frac{3}{4} = \frac{t^2}{4X_{ij}} \geq \frac{t}{4\sqrt{X_{ij}}} - \frac{1}{16}.
			\end{split}
		\end{equation*}
		Thus,
		\begin{equation*}
			\begin{split}
				\bbP\left(|Y_{ij} - X_{ij}|\geq t\right) \leq 2\exp\left(\frac{1}{16} - \frac{t}{4\sqrt{X_{ij}}}\right).
			\end{split}
		\end{equation*}
		\item If $X_{ij}/2 \leq t \leq 2X_{ij}$, we shall note that
		\begin{equation*}
			\frac{\partial}{\partial X_{ij}}\left((t+X_{ij})\log\left(1+\frac{t}{X_{ij}}\right)\right) = \log\left(1 + \frac{t}{X_{ij}}\right) - \frac{t}{X_{ij}} \leq 0,
		\end{equation*}
		then $(t+X_{ij})\log\left(1+\frac{t}{X_{ij}}\right)$ is a decreasing function of $X_{ij}$. Thus,
		\begin{equation*}
			\begin{split}
				\bbP\left(|Y_{ij} - X_{ij}|\geq t\right) \leq & 2\exp\left(-(t+X_{ij})\log\left(1+\frac{t}{X_{ij}}\right)+t\right)\\
				\leq & 2\exp\left(-(t + 2t)\log\left(1 + \frac{t}{2t}\right) + t\right)\\
				\leq & 2\exp\left(-(3\log(1.5)-1)t\right)\\
				\leq & 2\exp\left(-\sqrt{c}(3\log(1.5)-1)t/\sqrt{X_{ij}}\right)
			\end{split}
		\end{equation*}
		\begin{equation*}
			\begin{split}
				(X_{ij}+t)\log\left(1 + \frac{t}{X_{ij}}\right) - t \geq (t + t)\log(1+1/2) - t.
			\end{split}
		\end{equation*}
		\item If $t \geq 2X_{ij}$, 
		\begin{equation*}
			\begin{split}
				& \bbP\left(|Y_{ij}-X_{ij}| \geq t\right) \leq 2\exp\left( -t\log(1 + 2) + t\right) \\
				\leq & 2\exp\left(-t/\sqrt{X_{ij}} \cdot \left(\sqrt{c} (\log(3)-1)\right)\right).
			\end{split}
		\end{equation*}
	\end{itemize}
	In summary, \eqref{ineq:Poisson-heavy-tail} always hold, which means $E_{ij}/\sqrt{X_{ij}}$ is a sub-exponential distributed random variable. By the sub-exponential Wishart-type concentration inequality \cite[Theorem 3]{cai2020non},
	\begin{equation*}
		\mathbb{E}\left\|EE^\top - \mathbb{E}EE^\top \right\| \lesssim \sigma_C\sigma_R + \sigma_C^2 + \sigma_R\sigmax \sqrt{\log(p_1)\log(p_2)} + \sigmax \log(p_1)\log(p_2).
	\end{equation*}
	Suppose the right singular subspace of $X$ is $V\in \mathbb{O}_{p_2, r}$. By Lemma \ref{lm:poisson-projection} and $\|X\|\leq C\lambda_r(X)$,
	\begin{equation*}
		\mathbb{E}\left\|XE^\top\right\| \lesssim \|X\| \mathbb{E}\|EV\| \lesssim \|X\|(\sigma_C + r\sigmax)\lesssim \lambda_r(X)(\sigma_C + \sqrt{r}\sigmax).
	\end{equation*}
	Now, the rest of the proof follows from the Inequality \ref{ineq:thm2-middle-2} and the arguments below in proof of Theorem \ref{th:upper_bound_SVD}. We can finally prove that
	\begin{equation*}
		\begin{split}
			& \mathbb{E}\left\|\sin\Theta(\widehat{U}, U)\right\| \\
			\lesssim  & \left(\frac{\sigma_C+r\sigmax}{\lambda_r(X)} + \frac{\sigma_C\sigma_R + \sigma_R\sigmax\sqrt{\log(p_1)\log(p_2)} + \sigmax^2\log(p_1)\log(p_2)}{\lambda_r^2(X)}\right)\wedge 1.
		\end{split}
	\end{equation*}
	
\end{proof}

\subsection{Proofs in SVD Based on Heteroskedastic and Incomplete Data} 
\begin{proof}[Proof of Theorem \ref{th:matrix-completion}] ~ 
	
	\begin{itemize}[leftmargin=*]
		\item[Step 1] We first derive bounds for some key quantities, including $\sigma_B^2$ and $\|\|\B_k\|\|_{\psi_1}$ to be defined later, for the application of matrix concentration in the next step. Since $\|Y_{ij}\|_{\psi_2}\leq C$, $Y_{ij}$ is sub-Gaussian and has bounded moments
		$$\mathbb{E}|Y_{ij}|^\alpha \leq C,\quad \alpha=1,2,3,4.$$
		Since
		\begin{equation}\label{eq:EYY}
			\begin{split}
				\left(\mathbb{E}\widetilde{Y}\widetilde{Y}^\top\right)_{ij} = & \sum_{k=1}^{p_2} \mathbb{E}\widetilde{Y}_{ik}\widetilde{Y}_{jk} = \left\{\begin{array}{ll}
					\sum_{k=1}^n \theta\mathbb{E}Y_{ik}^2, & i = j;\\
					\sum_{k=1}^n \theta^2\mathbb{E}Y_{ik}Y_{jk}, & i\neq j
				\end{array}\right. \\
				= &  \left\{\begin{array}{ll}
					\theta(XX^\top)_{ii} + \theta\sum_{k=1}^{p_2} \Var(Z_{ik}), & i=j;\\
					\theta^2(XX^\top)_{ij}, & i\neq j,
				\end{array}\right.
			\end{split}
		\end{equation}
		we know $\Delta(\mathbb{E}\widetilde{Y}\widetilde{Y}^\top) = \Delta(\theta^2XX^\top)$, i.e., $\mathbb{E}\widetilde{Y}\widetilde{Y}^\top$ and $\theta^2XX^\top$ share the off-diagonal part. Recall $D(\cdot)$ and $\Delta(\cdot)$ represent the diagonal and off-diagonal part of the matrix, respectively. 
		
		Next, we establish a concentration inequality for $\left\|\widetilde{Y}\widetilde{Y}^\top - \mathbb{E}\widetilde{Y}\widetilde{Y}^\top\right\|$. Note the following decomposition,
		\begin{equation}\label{eq:tildeY-B}
			\widetilde{Y}\widetilde{Y}^\top - \mathbb{E}\widetilde{Y}\widetilde{Y}^\top = \sum_{k=1}^{p_2}\left(\widetilde{Y}_{\cdot k}\widetilde{Y}_{\cdot k}^\top - \mathbb{E}\widetilde{Y}_{\cdot k}\widetilde{Y}_{\cdot k}^\top\right) \triangleq \sum_{k=1}^{p_2} B_k, \quad B_k = \widetilde{Y}_{\cdot k}\widetilde{Y}_{\cdot k}^\top - \mathbb{E}\widetilde{Y}_{\cdot k}\widetilde{Y}_{\cdot k}^\top.
		\end{equation}
		Based on the assumption,
		$$|\mathbb{E}\widetilde{Y}_{ij}|^\alpha = \theta\mathbb{E}|Y_{ij}|^\alpha \leq C\theta,\quad \alpha=1,2,3,4.$$
		Then,
		\begin{equation}\label{ineq:EBB^top-bound}
			\begin{split}
				0 \preceq & \mathbb{E}B_kB_k^\top = \mathbb{E}\left(\widetilde{Y}_{\cdot k}\widetilde{Y}_{\cdot k}^\top -\mathbb{E}\widetilde{Y}_{\cdot k}\widetilde{Y}_{\cdot k}^\top\right)^2 \\
				= & \mathbb{E}\widetilde{Y}_{\cdot k}\widetilde{Y}_{\cdot k}^\top \widetilde{Y}_{\cdot k}\widetilde{Y}_{\cdot k}^\top - (\mathbb{E}\widetilde{Y}_{\cdot k}\widetilde{Y}_{\cdot k}^\top)^2 \preceq \mathbb{E}\widetilde{Y}_{\cdot k}\widetilde{Y}_{\cdot k}^\top \widetilde{Y}_{\cdot k}\widetilde{Y}_{\cdot k}^\top,
			\end{split}
		\end{equation}
		\begin{equation}\label{ineq:EYYYY1}
			\left|\left(\mathbb{E}\widetilde{Y}_{\cdot k}\widetilde{Y}_{\cdot k}^\top \widetilde{Y}_{\cdot k}\widetilde{Y}_{\cdot k}^\top\right)_{ij}\right| = \left|\mathbb{E} \widetilde{Y}_{ik}\left(\sum_{s=1}^{p_1}\widetilde{Y}_{sk}^2\right)\widetilde{Y}_{jk}\right|. \end{equation}
		If $i\neq j$, 
		\begin{equation}\label{ineq:EYYYY2}
			\begin{split}
				& \left|\mathbb{E} \widetilde{Y}_{ik}\left(\sum_{s=1}^{p_1}\widetilde{Y}_{sk}^2\right)\widetilde{Y}_{jk}\right| = \left|\mathbb{E}\widetilde{Y}_{ik}^3\widetilde{Y}_{jk} + \mathbb{E}\widetilde{Y}_{ik}\widetilde{Y}_{jk}^3 + \sum_{s\neq i,j} \mathbb{E}\widetilde{Y}_{ik}\widetilde{Y}_{sk}^2\widetilde{Y}_{jk}\right| \\
				\leq & \mathbb{E}|\widetilde{Y}_{ik}|^3\cdot\mathbb{E}|\widetilde{Y}_{jk}| + \mathbb{E}|\widetilde{Y}_{ik}|\cdot \mathbb{E}|\widetilde{Y}_{jk}|^3 + \sum_{s\neq i,j} \mathbb{E}|\widetilde{Y}_{ik}|\cdot\mathbb{E}|\widetilde{Y}_{sk}|^2\cdot \mathbb{E}|\widetilde{Y}_{jk}| \\
				\leq & C(\theta^3(p_1-2) + 2\theta^2);
			\end{split}
		\end{equation}
		if $i = j$,
		\begin{equation*}
			\begin{split}
				\left|\mathbb{E} \widetilde{Y}_{ik}\left(\sum_{s=1}^{p_1}\widetilde{Y}_{sk}^2\right)\widetilde{Y}_{jk}\right| = \left|\mathbb{E}\widetilde{Y}_{ik}^4 + \sum_{s\neq i} \mathbb{E}\widetilde{Y}_{ik}^2\widetilde{Y}_{sk}^2\right| \leq C(\theta^2(p_1-1) + \theta).
			\end{split}
		\end{equation*}
		Then,
		\begin{equation}\label{ineq:sigma_B_upper}
			\begin{split}
				\sigma_B^2 \triangleq & \left\|\sum_{k=1}^{p_2}\mathbb{E} B_k^2 \right\| \leq \sum_{k=1}^{p_2}\left\|\mathbb{E} B_k^2 \right\| \overset{\eqref{ineq:EBB^top-bound}}{\leq} \sum_{k=1}^{p_2}\left\|\mathbb{E}\widetilde{Y}_{\cdot k}\widetilde{Y}_{\cdot k}^\top \widetilde{Y}_{\cdot k}\widetilde{Y}_{\cdot k}^\top \right\| \\
				\leq & \sum_{k=1}^{p_2} \left(\left\|D(\mathbb{E}\widetilde{Y}_{\cdot k}\widetilde{Y}_{\cdot k}^\top \widetilde{Y}_{\cdot k}\widetilde{Y}_{\cdot k}^\top)\right\| + \left\|\Delta(\mathbb{E}\widetilde{Y}_{\cdot k}\widetilde{Y}_{\cdot k}^\top \widetilde{Y}_{\cdot k}\widetilde{Y}_{\cdot k}^\top)\right\|\right) \\
				\leq & \sum_{k=1}^{p_2} \left(\max_{i}\left(\mathbb{E}\widetilde{Y}_{\cdot k}\widetilde{Y}_{\cdot k}^\top \widetilde{Y}_{\cdot k}\widetilde{Y}_{\cdot k}^\top\right)_{ii} + \left\|\Delta(\mathbb{E}\widetilde{Y}_{\cdot k}\widetilde{Y}_{\cdot k}^\top \widetilde{Y}_{\cdot k}\widetilde{Y}_{\cdot k}^\top)\right\|_F\right) \\
				\overset{\eqref{ineq:EYYYY1}}{\leq} & Cp_2\left(\theta^2p_1 + \theta + \left\{\sum_{1\leq i\neq j \leq p_1}\left(\mathbb{E}\widetilde{Y}_{\cdot k}\widetilde{Y}_{\cdot k}^\top \widetilde{Y}_{\cdot k}\widetilde{Y}_{\cdot k}^\top\right)^2_{ij}\right\}^{1/2}\right)\\
				\overset{\eqref{ineq:EYYYY2}}{\leq} & Cp_2\left(\theta^2p_1 + \theta + p_1(\theta^3p_1+\theta^2)\right) \\
				= & Cp_2(\theta+\theta^2p_1 + \theta^3p_1^2) \leq Cp_2\left(\theta + \theta^3p_1^2\right).
			\end{split}
		\end{equation}
		On the other hand, 
		\begin{equation*}
			\sigma_B^2 \geq \max_{1\leq i\leq p_1} \left(\sum_{k=1}^{p_2}\mathbb{E}B_k^2\right)_{ii},
		\end{equation*}
		\begin{equation*}
			\begin{split}
				\text{where}\quad \left(\mathbb{E}B_k^2\right)_{ii} = & \left(\mathbb{E}\widetilde{Y}_{\cdot k}\widetilde{Y}_{\cdot k}^\top \widetilde{Y}_{\cdot k}\widetilde{Y}_{\cdot k}^\top\right)_{ii} - \left(\left(\mathbb{E}\widetilde{Y}_{\cdot k}\widetilde{Y}_{\cdot k}^\top\right)^2\right)_{ii}\\
				= & \mathbb{E}\widetilde{Y}_{ik}^2\left(\sum_{s=1}^{p_1}\widetilde{Y}_{sk}^2\right) - \sum_{s=1}^{p_1}\left(\mathbb{E}\widetilde{Y}_{ik}\widetilde{Y}_{sk}\right)^2\\
				= & \mathbb{E}\widetilde{Y}_{ik}^4 + \sum_{s\neq i} \mathbb{E}\widetilde{Y}_{ik}^2 \cdot \mathbb{E}\widetilde{Y}_{sk}^2 - \left(\mathbb{E}\widetilde{Y}_{ik}^2\right)^2 - \sum_{s\neq i} (\mathbb{E}\widetilde{Y}_{ik})^2 (\mathbb{E}\widetilde{Y}_{sk})^2\\
				\geq & \mathbb{E}\widetilde{Y}_{ik}^4 - \left(\mathbb{E}\widetilde{Y}_{ik}^2\right)^2 = \theta \mathbb{E}Y_{ik}^4 - \theta^2\mathbb{E}Y_{ik}^2 \geq (\theta - \theta^2)\mathbb{E}Y_{ik}^4.
			\end{split}
		\end{equation*}
		Provided that $\theta \leq 1-c$ for constant $c>0$, we have
		\begin{equation}\label{ineq:sigma_B_lower}
			\begin{split}
				\sigma_B^2 \geq & \max_i \left(\sum_{k=1}^{p_2}\mathbb{E}B_k^2\right)_{ii} \geq (\theta-\theta^2)\max_i \sum_{k=1}^{p_2} \mathbb{E} Y_{ik}^4 \geq \frac{c\theta}{p_2}\max_i \left(\mathbb{E} \sum_{k=1}^{p_2}Y_{ik}^2\right)^2 \\
				\geq & \frac{c\theta}{p_1^2p_2} \left(\mathbb{E} \sum_{i=1}^{p_1}\sum_{k=1}^{p_2}Y_{ik}^2\right)^2 \geq \frac{c\theta}{p_1^2p_2}(\mathbb{E}\|X\|_F^2)^2 \geq \frac{c\theta r^2}{p_1^2p_2}\lambda_r^4(X). 
			\end{split}
		\end{equation}
		
		Next, we give an upper bound for $\left\|\|B_k\|\right\|_{\psi_1}$. Note that
		\begin{equation*}
			\begin{split}
				\|B_k\| = \left\|\widetilde{Y}_{\cdot k} \widetilde{Y}_{\cdot k}^\top - \mathbb{E}\widetilde{Y}_{\cdot k} \widetilde{Y}_{\cdot k}^\top\right\| \leq \|\widetilde{Y}_{\cdot k}\widetilde{Y}_{\cdot k}^\top\| + \|\mathbb{E}\widetilde{Y}_{\cdot k}\widetilde{Y}_{\cdot k}^\top\|\leq \|\widetilde{Y}_{\cdot k}\|_2^2 + \|\mathbb{E}\widetilde{Y}_{\cdot k} \widetilde{Y}_{\cdot k}^\top\|.
			\end{split}
		\end{equation*}
		In particular, we set $t = C_1\theta p_1$ for sufficiently large constant $C_1>0$. Then,
		\begin{equation*}
			\begin{split}
				& \mathbb{E}\exp\left(\|B_k\|/t\right) \leq \mathbb{E}\exp\left\{\left(\|\widetilde{Y}_{\cdot k}\|_2^2 + \|\mathbb{E}\widetilde{Y}_{\cdot k}\widetilde{Y}_{\cdot k}^\top\|\right)/t\right\} \\
				= & \mathbb{E}\exp\left(\|\widetilde{Y}_{\cdot k}\|_2^2/t\right) \cdot \exp\left(\|\mathbb{E}\widetilde{Y}_{\cdot k}\widetilde{Y}_{\cdot k}^\top\|/t\right)\\
				\leq & \mathbb{E}\exp\left(\|\widetilde{Y}_{\cdot k}\|_2^2/t\right) \cdot \mathbb{E} \exp\left(\|\widetilde{Y}_{\cdot k}\widetilde{Y}_{\cdot k}^\top\|/t\right) \quad \text{(by Jensen's inequality)}\\
				= & \left(\mathbb{E}\exp\left(\|\widetilde{Y}_{\cdot k}\|_2^2/t\right)\right)^2 =  \left(\mathbb{E}\prod_{i=1}^{p_1}\exp\left(\widetilde{Y}_{ik}^2/t\right)\right)^2\\
				= & \prod_{i=1}^{p_1}\left(\mathbb{E}\exp\left(\widetilde{Y}_{ik}^2/t\right)\right)^2 \leq \prod_{i=1}^{p_1}\left(\mathbb{E}\exp(0/t)1_{\{R_{ik}=0\}} + \mathbb{E}\exp(Y_{ik}^2/t)1_{\{R_{ik}=1\}}\right)^2\\
				\overset{\text{Lemma \ref{lm:sub-exponential}}}{\leq} &  \prod_{i=1}^{p_1}\left((1-\theta)+\theta(1+C/t)\right)^2 = \left(1 + C\theta/t\right)^{2p_1} \leq 1 + C\theta p_1/t \leq 1 + C/C_1 \leq 2,
			\end{split}
		\end{equation*}
		which means 
		\begin{equation}\label{ineq:U_B}
			U_B^{(1)} \triangleq \left\|\|B_k\|\right\|_{\psi_1} = \inf\{b>0: \mathbb{E}\exp(\|B_k\|/b)\leq 2\} \leq C_1\theta p_1.
		\end{equation} 
		\item[Step 2] Next, we derive an upper bound for $\|\Delta(\widetilde{Y}\widetilde{Y}^\top - \theta^2XX^\top)\|$ based on the results of the previous step. By the Bernstein-type matrix concentration inequality (c.f., Proposition 2 in \cite{koltchinskii2011nuclear}),  \eqref{ineq:sigma_B_upper}, \eqref{ineq:sigma_B_lower}, and \eqref{ineq:U_B}, we have
		\begin{equation*}
			\begin{split}
				\left\|\sum_{k=1}^{p_2} B_k\right\| \leq & C\max\left\{\sigma_B\sqrt{\log(p_1)}, U_B^{(1)}\log(p_1)\log\left(\frac{U_B^{(1)}}{\sigma_B/\sqrt{p_2}} \right)\right\}\\
				\leq & C\max\left\{\sqrt{p_2(\theta+\theta^3p_1^2)\log(p_1)}, \theta p_1\log(p_1)\log\left(\frac{C\theta^{1/2} p_1^2p_2}{r\lambda_r^2(X)}\right)\right\}
			\end{split}
		\end{equation*}
		with probability at least $1 - p_1^{-C}$. By  \eqref{eq:EYY} and \eqref{eq:tildeY-B}, we further have $P(\mathcal{A}) \geq 1-p_1^{-C}$, where $\mathcal{A}$ is the event such that
		\begin{equation*}
			\begin{split}
				\mathcal{A} = \Bigg\{& \left\|\Delta\left(\widetilde Y\widetilde Y^\top - \theta^2XX^\top\right)\right\| \\
				& \qquad \leq C\max\left\{\sqrt{p_2(\theta+\theta^3p_1^2)\log(p_1)}, \theta p_1 \log(p_1)\log\left(\frac{C\theta^{1/2} p_1^2p_2}{r\lambda_r^2(X)}\right)\right\} \Bigg\}
			\end{split}
		\end{equation*}
		\item[Step 3] Finally, we finalize the proof by using the robust $\sin\Theta$ theorem. When the event $\mathcal{A}$ holds, by Theorem \ref{th:diagonal-less}, we have the following theoretical guarantee for the HeteroPCA estimator applying to $\widetilde{Y}\widetilde{Y}^\top$,
		\begin{equation}\label{ineq:U-hat-U}
			\begin{split}
				& \left\|\sin\Theta(\widehat{U}, U)\right\| \leq \frac{C\|\Delta(\widetilde{Y}\widetilde{Y}^\top - \theta^2XX^\top)\|}{\lambda_r(\theta^2 XX^\top)}\wedge 1\\
				\leq & \frac{C\max\left\{\sqrt{p_2(\theta+\theta^3p_1^2)\log(p_1)}, \theta p_1 \log(p_1)\log\left(\frac{C\theta^{1/2}p_1^2p_2}{r\lambda_r^2(X)}\right)\right\}}{\theta^2\lambda_r^2(X)} \wedge 1.
			\end{split}
		\end{equation}
		We discuss the bound above in two cases: first, if $\lambda_r^2(X) \geq \sqrt{p_2p_1^2/\theta}$,
		\begin{equation*}
			\log\left(\frac{C\theta^{1/2}p_1^2p_2}{r\lambda_r^2(X)}\right) \leq C\log\left(p_1p_2\right);
		\end{equation*}
		second, if $\lambda_r^2(X) \leq \sqrt{p_2p_1^2/\theta}$, we have
		\begin{equation*}
			\begin{split}
				& \left\|\sin\Theta(\widehat{U}, U)\right\| \leq 1 \leq \frac{C\sqrt{p_2(\theta+\theta^3p_1^2)\log(p_1)}}{\theta^2\lambda_r^2(X)} \wedge 1.
			\end{split}
		\end{equation*}
		Thus, if $\mathcal{A}$ holds, we always have
		\begin{equation*}
			\left\|\sin\Theta(\widehat{U}, U)\right\| \leq \frac{C\max\left\{\sqrt{p_2(\theta+\theta^3p_1^2)\log(p_1)}, \theta p_1 \log(p_1)\log\left(p_1p_2\right)\right\}}{\theta^2\lambda_r^2(X)}\wedge 1.
		\end{equation*}
	\end{itemize}
\end{proof}

\begin{proof}[Proof of the Consistency Result in Remark \ref{rm:matrix-completion}]
	
	If $\|X\| \leq C\lambda_r(X)$ and $\|X\|_F^2 \geq c p_1p_2$, we have 
	$$\lambda_r^2(X) \geq \frac{1}{C}\|X\|^2 \geq \frac{1}{Cr}\sum_{i=1}^r \lambda_i^2(X) \geq \frac{1}{Cr}\|X\|_F^2 \geq \frac{p_1p_2}{Cr}.$$
	If
	$$\theta \gg \max\left\{\frac{r^{2/3}\log^{1/3}(p_1)}{p_1^{2/3}p_2^{1/3}}, \frac{r^2\log(p_1)}{p_2}, \frac{r\log(p_1)\log(p_1p_2)}{p_2}\right\},$$
	or equivalently
	$$\mathbb{E}|\Omega| \gg \max\left\{p_1^{1/3}p_2^{2/3}r^{2/3}\log^{1/3}(p_1), p_1r^2\log(p_1), p_1r\log(p_1)\log(p_1p_2)\right\},$$
	we have 
	\begin{equation*}
		\begin{split}
			& \mathbb{E}\left\|\sin\Theta(\widehat{U}, U)\right\| = \mathbb{E}\left\|\sin\Theta(\widehat{U}, U)\right\| 1_{\mathcal{A}} + \mathbb{E}\left\|\sin\Theta(\widehat{U}, U)\right\| 1_{\mathcal{A}^c} \\
			\leq & \frac{C\max\left\{\sqrt{p_2(\theta+\theta^3p_1^2)\log(p_1)}, \theta p_1 \log(p_1)\log(p_1p_2)\right\}}{\theta^2\lambda_r^2(X)} \wedge 1 + \bbP(\mathcal{A}^c) \\
			\leq & \frac{C\max\left\{\sqrt{p_2(\theta+\theta^3p_1^2)\log(p_1)}, \theta p_1 \log(p_1)\log(p_1p_2)\right\}}{C\theta^2p_1p_2/r} \wedge 1 + p_1^{-C}= o(1)
		\end{split}
	\end{equation*} 
	as $p_1, p_2\to \infty$. 
\end{proof}
\section{Technical Lemmas}\label{sec:technical-lemma}

\begin{Lemma}\label{lm:orthogonal-projection}
	Assume that $E\in \mathbb{R}^{p_1\times p_2}$ has independent sub-Gaussian entries, $\Var(E_{ij}) = \sigma_{ij}^2$, $\sigma_C^2 = \max_j \sum_i\sigma_{ij}^2$, $\sigma_R^2 = \max_i\sum_j \sigma_{ij}^2, \sigmax^2 = \max_{i,j}\sigma_{ij}^2$. Assume that 
	\[
	\|E_{ij}/\sigma_{ij}\|_{\psi_2} = \max_{q\geq 1}q^{-1/2}\{\mathbb{E}(|E_{ij}|/\sigma_{ij})^q\}^{1/q} \leq \kappa.
	\]
	Let $V\in \mathbb{O}_{p_2, r}$ be a fixed orthogonal matrix. Then
	\begin{equation}\label{ineq:lm-ZV-1}
		\bbP\left(\|EV\| \geq 2\left(\sigma_C + x\right)\right) \leq 2\exp\left(5r - \min\left\{\frac{x^4}{\kappa^4 \sigmax^2\sigma_C^2}, \frac{x^2}{\kappa^2\sigmax^2}\right\} \right),
	\end{equation}
	\begin{equation}\label{ineq:lm-ZV-2}
		\mathbb{E}\|EV\| \lesssim \sigma_C + \kappa r^{1/4}(\sigmax \sigma_C)^{1/2} + \kappa r^{1/2}\sigmax .
	\end{equation}
\end{Lemma}
\begin{proof}[Proof of Lemma \ref{lm:orthogonal-projection}]
	We first construct $\mathcal{W}\subseteq \mathcal{B}_r = \{w\in \mathbb{R}^r: \|w\|_2\leq 1\}$ as the $\ell_2$ distance $\varepsilon$-net in $r$-dimensional space, such that $|\mathcal{W}|\leq (1+2/\varepsilon)^r$ \citep[Lemma 2.5]{vershynin2011spectral}. Since $E\in \mathbb{R}^{p_1\times p_2}$ has independent entries, for each fixed $w\in \mathcal{W}$, $EVw \in \mathbb{R}^{p_1}$ has independent entries and 
	$$\Var\left((EVw)_i\right) = \sum_{j=1}^{p_2} \Var(E_{ij})\cdot (Vw)_j^2 \leq \sum_{j=1}^{p_2} \sigmax^2 (Vw)_j^2 \leq \sigmax^2\|Vw\|_2^2 \leq \sigmax^2,$$
	$$\sum_{i=1}^{p_1} \Var\left((EVw)_i\right) = \sum_{i=1}^{p_1}\sum_{j=1}^{p_2} \Var(E_{ij})\cdot (Vw)_j^2 \leq \sum_{j=1}^{p_2}\sigma_C^2 (Vw)_j^2 \leq \sigma_C^2.$$
	Thus, we can rewrite the centralized $\|EVw\|_2^2$ as
	\begin{equation*}
		\begin{split}
			& \|EVw\|_2^2 - \sum_{i=1}^{p_1}\Var((EVw)_i) = \sum_{i=1}^{p_1} \left((EVw)_i^2/\Var((EVw)_i)-1\right) \cdot \Var((EVw)_i)
		\end{split}
	\end{equation*}
	Here, 
	\begin{equation*}
		\begin{split}
			& \sum_{i=1}^{p_1} \Var\left((EVw)_i\right) \leq \sigma_C^2, \quad \max_i \Var((EVw)_i) \leq \sigmax^2, \\
			& \sum_{i=1}^{p_1} \Var^2((EVw)_i) \leq \sigmax^2 \sum_{i=1}^{p_1}\Var((EVw)_i) \leq \sigmax^2 \sigma_C^2.
		\end{split}
	\end{equation*}
	By Bernstein-type concentration inequality \citep[Proposition 5.16]{vershynin2010introduction}, 
	\begin{equation*}
		\bbP\left(\|EVw\|_2^2 \geq \sigma_C^2 + t\right) \leq 2\exp\left(-\min\left\{\frac{t^2}{\kappa^4\sigmax^2\sigma_C^2}, \frac{t}{\kappa^2\sigmax^2}\right\}\right).
	\end{equation*}
	Applying the union bound for all $w\in \mathcal{W}$, we obtain
	\begin{equation*}
		\bbP\left(\max_{w\in \mathcal{W}}\|EVw\|_2^2 \geq \sigma_C^2 + t\right) \leq 2\left(1+2/\varepsilon\right)^r\exp\left(-\min\left\{\frac{t^2}{\kappa^4\sigma_C^2\sigmax^2}, \frac{t}{\kappa^2\sigmax^2}\right\}\right).
	\end{equation*}
	Next, suppose $u^\ast = \argmax_{\substack{u\in \mathbb{R}^r\\\|u\|_2\leq 1}} \|EVu\|_2$. By definition of $\varepsilon$-net, there exists $w\in\mathcal{W}$, such that $\|u^\ast-w\|_2\leq \varepsilon$ and
	\begin{equation*}
		\begin{split}
			\|EV\| = & \|EVu^\ast\|_2 \leq \|EVw\|_2 + \|EV(u^\ast-w)\|_2 \\
			\leq & \varepsilon\|EV\| + \max_{w\in \mathcal{W}} \|EVw\|_2.
		\end{split}
	\end{equation*}
	Namely, $\|EV\| \leq \max_{w\in \mathcal{W}} \|EVw\|_2/(1-\varepsilon)$. Setting $\varepsilon=1/2$, we have
	\begin{equation}
		\bbP\left(\|EV\| \geq 2\left(\sigma_C + x\right)\right) \leq 2\exp\left(5r - \min\left\{\frac{x^4}{\kappa^4 \sigmax^2\sigma_C^2}, \frac{x^2}{\kappa^2\sigmax^2}\right\} \right),
	\end{equation}
	which has proved \eqref{ineq:lm-ZV-1}. 
	
	Next, we consider the expectation upper bound. For any $x\geq 0$, $\bbP\left(\|EV\|\geq x\right)\leq 1$; for any $x\geq 2\sigma_C + 10\kappa\sqrt{r}\sigmax + 10\kappa r^{1/4}(\sigmax\sigma_C)^{1/2}$,
	\begin{equation*}
		\begin{split}
			& \bbP\left(\|EV\|\geq x\right) \leq 2\exp\left(5\log (r) - \min\left\{\frac{(x/2-\sigma_C)^4}{\kappa^4 \sigmax^2\sigma_C^2}, \frac{(x/2-\sigma_C)^2}{\kappa^2\sigmax^2}\right\} \right) \\
			\leq & 2\exp\left(5\log(r) - \frac{(x/2-\sigma_C)^4}{\kappa^4 \sigmax^2\sigma_C^2}\right) + 2\exp\left(5\log(r) - \frac{(x/2-\sigma_C)^2}{\kappa^2\sigmax^2}\right)\\
			\leq & 2\exp\left(- \frac{(x/2-\sigma_C)^4}{2\kappa^4 \sigmax^2\sigma_C^2}\right) + 2\exp\left( - \frac{(x/2-\sigma_C)^2}{2\kappa^2\sigmax^2}\right).\\
		\end{split}
	\end{equation*}
	Thus,
	\begin{equation*}
		\begin{split}
			& \mathbb{E}\|EV\| = \int_0^\infty \bbP\left(\|EV\|\geq x\right)dx \\
			= &  \int_0^{2\sigma_C + 10\kappa \sqrt{r} \sigmax + 10\kappa r^{1/4}(\sigmax\sigma_C)^{1/2}} \bbP\left(\|EV\|\geq x\right)dx\\
			& + \int_{2\sigma_C + 10\kappa \sqrt{r} \sigmax + 10\kappa r^{1/4}(\sigmax\sigma_C)^{1/2}}^\infty \bbP\left(\|EV\|\geq x\right)dx\\
			\leq & 2\sigma_C + 10\kappa\sqrt{r}\sigmax + 10\kappa r^{1/4}(\sigmax\sigma_C)^{1/2} \\
			& + \int_0^\infty \left\{2\exp\left(-\frac{(x/2)^4}{\kappa^4\sigmax^2\sigma_C^2}\right) + 2\exp\left(-\frac{(x/2)^2}{\kappa^2\sigmax^2}\right)\right\}dx\\
			\leq & 2\sigma_C + 10\kappa\sqrt{r}\sigmax + 10\kappa r^{1/4}(\sigmax\sigma_C)^{1/2} \\
			& + 4\kappa (\sigmax \sigma_C)^{1/2}\int_0^\infty e^{-x^4}dx + 4\kappa\sigmax \int_0^\infty -x^2 dx\\
			\leq & C\left(\sigma_C + \kappa r^{1/4}(\sigmax \sigma_C)^{1/2} + \kappa \sigmax\sqrt{r}\right).
		\end{split}
	\end{equation*}
	We thus have finished the proof of \eqref{ineq:lm-ZV-2}. \end{proof}

\begin{Lemma}[Spectral Norm of Projected Random Matrix with independent~Sub-exponential Entries]\label{lm:poisson-projection}
	Suppose $E\in \mathbb{R}^{p_1\times p_2}$ has independent sub-exponential entries, $\Var(E_{ij}) = \sigma_{ij}^2$, $\sigma_C^2 = \max_j \sum_i\sigma_{ij}^2$, $\sigmax^2 = \max_{i,j}\sigma_{ij}^2$. Assume that 
	\[
	\||E_{ij}|/\sigma_{ij}\|_{\psi_1} = \max_{q\geq 1}q^{-1}\{\mathbb{E}(E_{ij}/\sigma_{ij})^q\}^{1/q} \leq C.
	\]
	Suppose $V\in \mathbb{O}_{p_2, r}$ is a fixed orthogonal matrix. Then,
	\begin{equation*}
		\mathbb{E}\|EV\|^2 \lesssim \sigma_C^2 + r^2\sigmax^2.
	\end{equation*}
\end{Lemma}
\begin{proof}[Proof of Lemma \ref{lm:poisson-projection}]
	We divide the proof into four steps. 
	\begin{itemize}[leftmargin=*]
		\item[Step 1] First, we introduce an $\varepsilon$-net to reduce the matrix concentration problem to a simpler vector one. Let $\mathcal{W}\subseteq \mathcal{B}_r = \{w\in \mathbb{R}^r: \|w\|_2\leq 1\}$ be the $\ell_2$ distance $\varepsilon$-net in $r$-dimensional space, such that $\varepsilon=1/2$ and $|\mathcal{W}|\leq (1+2/(1/2))^r = 5^r$ \citep[Lemma 2.5]{vershynin2011spectral}. Since $E$ is a random matrix with independent entries, for any fixed $w\in \mathcal{W}$, the vector $EVw$ has independent entries, 
		\begin{equation}\label{ineq:ZVw}
			\begin{split}
				& \mathbb{E}\|EVw\|_2^2 = \sum_{i=1}^{p_1} \mathbb{E}(EVw)_i^2 = \sum_{i=1}^{p_1} \Var((EVw)_i) \\
				= & \sum_{i=1}^{p_1}\sum_{j=1}^{p_2} \Var(E_{ij})\cdot (Vw)_j^2 \leq \sum_{j=1}^{p_2}\sigma_C^2 (Vw)_j^2 = \sigma_C^2.
			\end{split}
		\end{equation}
		\item[Step 2] Then we establish the concentration for each entry of $EVw$, say $(EVw)_i$. Denote $f_{ij} = \sigma_{ij} (Vw)_j$. We have
		\begin{equation}
			\mathbb{E} (EVw)_i^2 = \Var\left(\sum_{j=1}^{p_2} E_{ij}(Vw)_j\right) = \sum_{j=1}^{p_2} \sigma_{ij}^2(Vw)_j^2 = \sum_{j=1}^{p_2} f_{ij}^2 = \|f_{i\cdot}\|_2^2,
		\end{equation}
		\begin{equation}\label{ineq:max-i-f}
			\max_i \|f_{i\cdot}\|_2^2 = \max_i \sum_{j=1}^{p_2} \sigma^2_{ij} (Vw)_j^2 \leq \left(\max_{i,j}\sigma^2_{ij}\right)\cdot  \sum_{j=1}^{p_2}(Vw)_j^2 \leq \sigmax^2,
		\end{equation}
		\begin{equation}\label{ineq:sum-i-f}
			\sum_{i=1}^{p_1}\|f_{i\cdot}\|_2^2 = \sum_{i=1}^{p_1}\sum_{j=1}^{p_2} \sigma_{ij}^2 (Vw)_j^2 \leq \sum_{j=1}^{p_2} \sigma_C^2 (Vw)_j^2 \leq \sigma_C^2.
		\end{equation}
		Note that
		\begin{equation*}
			\sum_{j=1}^{p_2}\frac{E_{ij}}{\sigma_{ij}} \cdot \sigma_{ij}(Vw)_j = \sum_{j=1}^{p_2} E_{ij}(Vw)_j = (EVw)_i.
		\end{equation*} 
		By Bernstein-type concentration inequality (c.f., \cite[Proposition 5.16]{vershynin2010introduction}),
		\begin{equation}\label{eq:ZVw-tail}
			\begin{split}
				& \bbP\left(\left|(EVw)_i\right| \geq t \right) =  \bbP\left(\left|\sum_{j=1}^{p_2}\frac{E_{ij}}{\sigma_{ij}} \cdot \sigma_{ij}(Vw)_j\right| \geq t \right) \\ 
				\leq &  2\exp\left(-c\min\left\{\frac{t^2}{\|f_{i\cdot}\|_2^2}, \frac{t}{\|f_{i\cdot}\|_\infty}\right\}\right) \leq 2\exp\left(-c\min\left\{\frac{t^2}{\|f_{i\cdot}\|_2^2}, \frac{t}{\|f_{i\cdot}\|_2}\right\}\right)\\
				\leq & 2\exp\left(-c\min\left\{\frac{t}{\|f_{i\cdot}\|_2}-\frac{1}{4}, \frac{t}{\|f_{i\cdot}\|_2}\right\}\right) \leq 2\exp\left(\frac{c}{4} - ct/\|f_{i\cdot}\|_2\right).
			\end{split}
		\end{equation}
		Next, we consider $T_i \triangleq (EVw)_i^2 - \|f_{i\cdot}\|_2^2$, $i=1,\ldots, p_1$ and aim to establish the tail property of $T_i$. Suppose $C_1$ and $\widetilde{C}$ and two constants to be determined later. Then,
		\begin{equation*}
			\begin{split}
				& \mathbb{E}\exp\left(\frac{|T_i|^{1/2}}{C_1\|f_{i\cdot}\|_2}\right) = \mathbb{E}\exp\left(\frac{\left|(EVw)_i^2 - \|f_{i\cdot}\|_2^2\right|^{1/2}}{C_1\|f_{i\cdot}\|_2}\right)\\
				\leq & \mathbb{E}\exp\left(\frac{|(EVw)_i| + \|f_{i\cdot}\|_2}{C_1\|f_{i\cdot}\|_2}\right) \\
				= & \int_0^\infty \frac{\partial}{\partial t}\left(\mathbb{E}\exp\left(\frac{t + \|f_{i\cdot}\|_2}{C_1\|f_{i\cdot}\|_2}\right)\right) \bbP\left(|(EVw)_i|\geq t\right)dt\\
				\overset{\eqref{eq:ZVw-tail}}{\leq} & \int_0^{\widetilde{C}\|f_{i\cdot}\|_2} \frac{1}{C_1\|f_{i\cdot}\|_2}\exp\left(\frac{t + \|f_{i\cdot}\|_2}{C_1\|f_{i\cdot}\|_2}\right)dt \\
				& + \int_{\widetilde{C}\|f_{i\cdot}\|_2}^\infty \frac{1}{C_1\|f_{i\cdot}\|_2}\exp\left(\frac{t + \|f_{i\cdot}\|_2}{C_1\|f_{i\cdot}\|_2}\right) 2\exp\left(\frac{c}{4} - \frac{ct}{\|f_{i\cdot}\|_2}\right) dt\\
				\leq & \frac{\widetilde{C}}{C_1} \exp\left(\frac{\widetilde{C}+1}{C_1}\right) + \int_{\widetilde{C}\|f_{i\cdot}\|_2}^\infty \frac{2\exp(\frac{c}{4}+\frac{1}{C_1})}{C_1\|f_{i\cdot}\|_2}\exp\left(-\left(c-\frac{1}{C_1}\right)\frac{t}{\|f_{i\cdot}\|_2}\right)dt\\
				= & \frac{\widetilde{C}}{C_1}\exp\left(\frac{\widetilde{C}+1}{C_1}\right) + \frac{2\exp\left(\frac{c}{4} + \frac{1}{C_1} - \left(c-\frac{1}{C_1}\right)\widetilde{C}\right)}{cC_1 - 1}.
			\end{split}
		\end{equation*}
		Let $\widetilde{C} = \sqrt{C_1}$. We can see for large constant $C_1$, $\mathbb{E}\exp\left(\frac{|T_i|^{1/2}}{C_1\|f_{i\cdot}\|_2}\right) \leq 2$. Then,
		\begin{equation*}
			\left\|T_i\right\|_{\psi_{1/2}} \triangleq \inf\left\{\alpha > 0: \mathbb{E} \exp\left(|T_{i}/\alpha|^{1/2}\right) \leq 2\right\} \leq C_1^2 \|f_{i\cdot}\|_2^2.
		\end{equation*}
		
		\item[Step 3] In this step we establish the concentration inequality for the $\ell_2$ norm of the vector $EVw$. Noting that $\mathbb{E}T_i=0$, by the tail inequality for sum of heavy tail random variables (c.f., Lemma 6 in \cite{hao2020sparse}), we have for any $q\geq 2$, 
		\begin{equation*}
			\begin{split}
				& \left(\mathbb{E}\left|\|EVw\|_2^2 - \mathbb{E}\|EVw\|_2^2\right|^q\right)^{1/q} = \left(\mathbb{E}\left|\sum_{i=1}^{p_1} T_i\right|^q\right)^{1/q} \\
				\leq & C\left(\sqrt{q}\left(\sum_{i=1}^{p_1}\|f_{i\cdot}\|_2^4\right)^{1/2} + q^2\left(\sum_{i=1}^{p_1}\|f_{i\cdot}\|_2^{2q}\right)^{1/q}\right)\\
				\leq & C\Bigg(\sqrt{q}\left(\sum_{i=1}^{p_1}\|f_{i\cdot}\|_2^2 \cdot \max_i \|f_{i\cdot}\|_2^2\right)^{1/2}\\ 
				& + q^2 \left(\sum_{i=1}^{p_1}\|f_{i\cdot}\|_2^2\cdot \max_i \|f_{i\cdot}\|_2^{2q-2}\right)^{1/q}\Bigg)\\
				\overset{\eqref{ineq:max-i-f}\eqref{ineq:sum-i-f}}{\leq} & C\sqrt{q} \sigma_C\sigmax + Cq^2 \sigma_C^{2/q} \sigmax^{(2q-2)/q} = C\sqrt{q} \sigma_C\sigmax + Cq^2(\sigma_C/\sigmax)^{2/q} \sigmax.
			\end{split}
		\end{equation*}
		Set 
		\begin{equation}\label{eq:q}
			q = 2(r+1) + \log(\sigma_C/\sigmax),
		\end{equation} 
		we have
		\begin{equation*}
			\begin{split}
				& \left(\mathbb{E}\left|\|EVw\|_2^2 - \mathbb{E}\|EVw\|_2^2\right|^q\right)^{1/q} \leq C\sqrt{q}\sigma_C\sigmax + Cq^2\sigmax \triangleq G. 
			\end{split}
		\end{equation*}
		By Markov inequality,
		\begin{equation*}
			\begin{split}
				& \bbP\left(\left|\|EVw\|_2^2 - \mathbb{E}\|EVw\|_2^2\right| \geq t\right) = \bbP\left(\left|\|EVw\|_2^2 - \mathbb{E}\|EVw\|_2^2\right|^q \geq t^q\right) \\
				\leq & \frac{\mathbb{E}\left|\|EVw\|_2^2 - \mathbb{E}\|EVw\|_2^2\right|^q}{t^q} = \frac{G^q}{t^q}.
			\end{split}
		\end{equation*}
		
		\item[Step 4] Finally, we apply the $\varepsilon$-net technique to derive the upper bound for $\|EV\|_2^2 = \max_{\|w\|_2\leq 1} \|EVw\|_2^2$ from the concentration inequality of $\|EVw\|_2^2$ with fixed $w$. Applying the union bound, we have
		\begin{equation}\label{ineq:union-bound}
			\begin{split}
				& \bbP\left(\max_{w\in \mathcal{W}} \left|\|EVw\|_2^2 - \mathbb{E}\|EVw\|_2^2\right| \geq t\right)\\ 
				\leq & |\mathcal{W}|\cdot \bbP\left(\left|\|EVw\|_2^2 - \mathbb{E}\|EVw\|_2^2\right| \geq t\right) \leq \frac{G^q 5^r}{t^q}.
			\end{split}
		\end{equation}
		Suppose $u^\ast = \argmax_{\substack{u\in \mathbb{R}^r\\\|u\|_2\leq 1}} \|EVu\|_2$. By definition of $\varepsilon$-net, there exists $w\in \mathcal{W}$, such that $\|u^\ast - w\|_2\leq 1/2$. Then,
		$$\|EV\| = \|EVu^\ast\|_2 \leq \|EVw\|_2 + \|EV(u^\ast - w)\|_2 \leq  \max_{w\in \mathcal{W}}\|EVw\|_2 + \|EV\|/2 ,$$
		which means $\|EV\| \leq \max_{w\in \mathcal{W}}\|EVw\|_2/(1-1/2) = 2\max_{w\in \mathcal{W}}\|EVw\|_2$. Therefore,
		\begin{equation*}
			\begin{split}
				& \mathbb{E}\|EV\|_2^2 \leq 4\mathbb{E}\max_{w\in \mathcal{W}}\|EVw\|_2^2 \leq 4\max_{w\in \mathcal{W}}\left(\mathbb{E}\|EVw\|_2^2 + \left|\|EVw\|_2^2 - \mathbb{E}\|EVw\|_2^2\right|\right)\\
				\overset{\eqref{ineq:sum-i-f}}{\leq} & 4\sigma_C^2 + 4\int_0^\infty \bbP\left(\left|\|EVw\|_2^2 - \mathbb{E}\|EVw\|_2^2\right| \geq t\right)dt\\
				\leq & 4\sigma_C^2 + 4 \int_0^{5G} 1\cdot dt + 4 \int_{5G}^\infty \bbP\left(\left|\|EVw\|_2^2 - \mathbb{E}\|EVw\|_2^2\right| \geq t\right)dt\\
				\overset{\eqref{ineq:union-bound}}{\leq} & 4\sigma_C^2 + 20G + 4\int_{5G}^\infty \frac{G^q5^r}{t^q}dt = 4\sigma_C^2 + 20G + \frac{G(q-1)}{5^{q-1-r}} \\
				\overset{\eqref{eq:q}}{\leq} & 4\sigma_C^2 + 20G + \frac{G(q-1)}{5^{(q-1)/2}} \leq 4\sigma_C^2 + CG \\
				= & 4\sigma_C^2 + C\left(r + \log(\sigma_C/\sigmax)\right)\sigma_C\sigmax + C\left(r + \log(\sigma_C/\sigmax)\right)^2\sigmax^2.
			\end{split}
		\end{equation*}
		Finally, by arithmetic-geometric inequality,  
		\begin{equation*}
			\begin{split}
				& (r+\log(\sigma_C/\sigmax))\sigma_C\sigmax \leq \frac{1}{2}\sigma_C^2 + \frac{1}{2}(r+\log(\sigma_C/\sigmax))^2\sigmax^2\\ \lesssim & \sigma_C^2 + r^2 + \log^2(\sigma_C/\sigmax)\sigmax^2,
			\end{split}
		\end{equation*}
		\begin{equation*}
			\log^2(\sigma_C/\sigmax)\sigmax^2 \lesssim (\sigma_C/\sigmax)^2\sigmax^2 = \sigma_C^2,
		\end{equation*}
		we have
		\begin{equation*}
			\mathbb{E}\|EV\|_2^2 \lesssim \sigma_C^2 + r^2\sigmax^2.
		\end{equation*}
	\end{itemize}
\end{proof}

The following lemma provides a sharp bound for the operator norm of matrix sparsification. 
\begin{Lemma}\label{lm:diagonal-less-spectral-norm}
	If $M\in \mathbb{R}^{m_1\times m_2}$, $\rank(M) = r$, $\mathcal{G} \subseteq [m_1]\times [m_2]$, $\max_i|\{j: (i,j)\in \mathcal{G}\}|\leq b, \max_j|\{i: (i,j)\in \mathcal{G}\}|\leq b$, then we have
	$$\|G(M)\| \leq \sqrt{b\wedge r}\|M\|, \quad \|\Gamma(M)\| \leq (\sqrt{b\wedge r}+1)\|M\|.$$
	
	In particular, if $M\in \mathbb{R}^{p\times p}$ is any square matrix and $\Delta(M)$ is the matrix $M$ with diagonal entries set to 0, then 
	$$\|\Delta(M)\| \leq 2\|M\|.$$ 
	Here, the factor ``2" in the statement above cannot be improved.
\end{Lemma}

\begin{proof}[Proof of Lemma \ref{lm:diagonal-less-spectral-norm}] If $M\in \mathbb{R}^{m_1\times m_2}$, $M$ can be seen as a linear operator from $\mathbb{R}^{m_2}$ to $\mathbb{R}^{m_1}$. Note that
	\begin{equation*}
		\begin{split}
			\|G(M)\|_{\infty } \triangleq & \max_{x\in \mathbb{R}^{m_2}}\frac{\|G(M)x\|_\infty}{\|x\|_\infty} = \max_i \sum_{j=1}^{m_2} |G(M_{ij})| = \max_i \sum_{j: (i, j)\in \mathcal{G}} |M_{ij}| \\
			\leq & \max_i \sqrt{b}\left(\sum_{j: (i, j)\in \mathcal{G}} |M_{ij}|^2\right)^{1/2} \leq \sqrt{b}\max_i \|M_{i\cdot}\|_2 \leq \sqrt{b}\|M\|;\\
			\|G(M)\|_{1} \triangleq & \max_{x\in \mathbb{R}^{m_2}} \frac{\|G(M)x\|_1}{\|x\|_1} = \max_j \sum_{i=1}^{m_1}|G(M_{ij})| = \max_j \sum_{i: (i, j)\in \mathcal{G}}|M_{ij}| \\
			\leq & \max_j \sqrt{b}\left(\sum_{i: (i,j)\in \mathcal{G}}|M_{ij}|^2\right)^{1/2} \leq \sqrt{b}\max_j \|M_{\cdot j}\|_2 \leq \sqrt{b}\|M\|.
		\end{split}
	\end{equation*}
	By Riesz-Thorin interpolation theorem \cite[Chapter 4, Section 1.2]{katznelson2004introduction},
	\begin{equation*}
		\begin{split}
			\|G(M)\| \leq \left(\|G(M)\|_\infty \cdot \|G(M)\|_1\right)^{1/2} \leq \sqrt{b}\|M\|.
		\end{split}
	\end{equation*}
	Since $\rank(M) = r$, we also have
	\begin{equation*}
		\begin{split}
			\|G(M)\| \leq \|G(M)\|_F \leq \|M\|_F \leq \sqrt{r}\|M\|.
		\end{split}
	\end{equation*}
	The previous two inequalities yield
	\begin{equation*}
		\|G(M)\| \leq \sqrt{b\wedge r}\|M\|.
	\end{equation*}
	Finally,
	\begin{equation*}
		\|\Gamma(M)\| = \|M - G(M)\| \leq ( \sqrt{b\wedge r}+1)\|M\|.
	\end{equation*}
	
	In particular, note that $\Delta(M) = M - D(M)$, $\|D(M)\| = \max_{i} |M_{ii}| \leq \|M\|$, we have 
	\begin{equation*}
		\|\Delta M \| = \|M - D(M)\| \leq \|M\| + \|D(M)\|  \leq 2\|M\|.
	\end{equation*}
	
	Finally we provide an example to illustrate that the factor ``2" above is sharp. Suppose $p\geq 2$, $1_p$ is the $p$-dimensional all-one vector. Set $M = 1_p1_p^\top - \frac{p}{2}I_p$. Then, $\Delta(M) = 1_p 1_p^\top - I_p$. Since the eigenvalues of $1_p1_p^\top$ are $\{p, 0, \ldots, 0\}$, the eigenvalues of $(\Delta(M) = 1_p1_p^\top - I_p)$ and $(M = 1_p1_p^\top - \frac{p}{2}\cdot I_p)$ are $\{p-1, -1, \ldots, -1\}$ and $\{p/2, -p/2, \ldots, -p/2\}$, respectively. At this point,
	$$\frac{\|\Delta M\|}{\|M\|} = \frac{p-1}{p/2} = 2 - \frac{2}{p}.$$
	As $p\to \infty$, we can see the statement $\|\Delta M \| \leq (2-\varepsilon)\|M\|$ does not hold in general for any $\varepsilon>0$. \end{proof}

\begin{Lemma}\label{lm:epsilon}
	Suppose $E_{p_1}\subseteq \mathbb{S}^{p_1-1}$, $E_{p_2} \subseteq \mathbb{S}^{p_2-1}$ are $\varepsilon$-net in $p_1$- and $p_2$-dimensional spheres, $\varepsilon<1/2$, then for any symmetric matrix $A\in \mathbb{R}^{p_1\times p_1}$ and general $B\in \mathbb{R}^{p_1\times p_2}$,
	\begin{equation*}
		\|A\| \leq \frac{\max_{v \in E_{p_1}} |v^\top Av|}{1-2\varepsilon},\quad \|B\| \leq \frac{\max_{u\in E_{p_1}, v\in E_n} u^\top B v}{1-2\varepsilon}.
	\end{equation*}
\end{Lemma}

\begin{proof}[Proof of Lemma \ref{lm:epsilon}] Suppose $\widetilde{v}\in \mathbb{S}^{p_1-1}$ is the eigenvector of $A$ corresponding to the eigenvalue with largest absolute value, then $\widetilde{v}$ satisfies $\widetilde{v}^\top A\widetilde{v} = \|A\|$. Since $E_{p_1}$ is an $\varepsilon$-net of $\mathbb{S}^{p_1-1}$, there exists $u\in E_{p_1}$ such that $\|u - \widetilde{v}\| \leq \varepsilon$. Thus,
	\begin{equation*}
		\begin{split}
			\|A\| = & \left|\widetilde{v}^\top A\widetilde{v}\right| \leq \left|\widetilde{v}^\top A(\widetilde{v}-v)\right| + \left|(\widetilde{v} - v)^\top A v\right| + \left|v^\top Av\right|\\
			\leq & \|\widetilde{v}\|_2\cdot \|\widetilde{v}-v\|_2\cdot \|A\| + \|v\|_2\cdot \|\widetilde{v}-v\|_2\cdot \|A\| + \max_{v\in E_{p_1}}\left|v^\top Av\right|\\
			\leq & 2\varepsilon\|A\| + \max_{v\in E_{p_1}}\left|v^\top A v\right|,
		\end{split}
	\end{equation*}
	which implies $\|A\|\leq \frac{1}{1-2\varepsilon} \max_{v\in E_{p_1}} \left|v^\top Av\right|$. Similarly, suppose $\bar{u}$ and $\bar{v}\in \mathbb{S}^{p_1-1}$ are the left and right singular vectors of $B$ corresponding to its largest singular value. Then $B$ satisfies $\bar{u}^\top B\bar{v}_2 = \|B\|$, and there exists $u\in E_{p_1}$ and $v\in E_{p_2}$ such that $\|\bar{u}-u\|_2\leq \varepsilon$, $\|\bar{v} - u\|_2\leq \varepsilon$. Therefore,
	\begin{equation*}
		\begin{split}
			\|B\| = & \widetilde{u}^\top B \widetilde{v} \leq u^\top B v + (\widetilde{u}-u)^\top B v + \widetilde{u}^\top B (\widetilde{v} - v)\\
			\leq & \max_{u\in E_{p_1}, v\in E_{p_2}} u^\top Bv + \|\widetilde{u}-u\|_2 \|B\|\cdot \|v\| + \|\widetilde{u}\|_2\cdot \|B\|\cdot \|\widetilde{v} - v\|_2\\
			\leq & 2\varepsilon\|B\| + \max_{u\in E_{p_1}, v\in E_{p_2}} u^\top Bv,
		\end{split}
	\end{equation*}
	which implies $\|B\|\leq \frac{1}{1-2\varepsilon}\max_{u\in E_{p_1}, v\in E_{p_2}}u^\top B v$. \end{proof}

The following technical tool characterizes the spectral and Frobenius norm of projections after SVD. The proof is provided in \cite[Lemma 6]{zhang2018tensor}.
\begin{Lemma}\label{lm:P_hat_U M}
	Suppose $M, E\in \mathbb{R}^{p_1\times p_2}$, $\rank(M)=r$. If $\widehat{U} = \SVD_r(M+E)$ and $\widehat{U}_{\perp}$ is the orthogonal complement of $\widehat{U}$, then
	\begin{equation*}
		\left\|P_{\widehat{U}_\perp}M\right\| \leq 2\|E\|, \quad \left\|P_{\widehat{U}_\perp}M\right\|_F \leq 2\min\{\sqrt{r}\|E\|, \|E\|_F\}.
	\end{equation*}
\end{Lemma}

The following lemma gives a upper bound for $\mathbb{E}\exp(X^2/t)$ for sub-Gaussian distributed random variable $X$.
\begin{Lemma}\label{lm:sub-exponential}
	Suppose $X$ satisfies $\|X\|_{\psi_2} \leq B.$
	If $t \geq 4eB^2$, we have
	\begin{equation*}
		\begin{split}
			\mathbb{E}\exp(X^2/t) \leq 1+\sqrt{8/\pi} eB^2/t.
		\end{split}
	\end{equation*}
\end{Lemma}
\begin{proof}[Proof of Lemma \ref{lm:sub-exponential}] If $t \geq 4eB^2$,
	\begin{equation*}
		\begin{split}
			& \mathbb{E}\exp(X^2/t) \\
			= & 1 + \sum_{k=1}^\infty \mathbb{E}\frac{X^{2k}}{t^k k!} \leq 1 + \sum_{k=1}^\infty \frac{(2k)^k \cdot B^{2k}}{t^k \cdot \sqrt{2\pi} k^{k+.5} \cdot e^{-k}} \quad (\text{Stirling's Formula})\\
			\leq & 1 + \sum_{k=1}^\infty \left(\frac{2eB^2}{t}\right)^{k}\frac{1}{\sqrt{2\pi k}} \leq 1 + \frac{1}{\sqrt{2\pi}}\sum_{k=1}^\infty \left(\frac{2eB^2}{t}\right)^k\\
			\leq & 1 + \frac{2eB^2/t}{\sqrt{2\pi}(1 - 2eB^2/t)} \leq 1 + \sqrt{\frac{8}{\pi}} eB^2/t.
		\end{split}
	\end{equation*}
\end{proof}

The following lemma gives a simple construction of an orthogonal matrix of arbitrary dimension that satisfies the incoherence constraint.
\begin{Lemma}\label{lm:incoherence-orthogonoal-construction}
	Suppose $p \geq r \geq1$. There exists a $p$-by-$r$ matrix $Q$ with orthonormal columns, i.e., $Q \in \mathbb{O}_{p, r}$, such that 
	$$\max_{1\leq i \leq p} \|e_i^\top Q\|_2^2 \leq \frac{1}{\lfloor p/r\rfloor}.$$
\end{Lemma}
\begin{proof}
	Let $\alpha = \lfloor p/r\rfloor$, $\beta = p - \alpha r$. Construct
	\begin{equation*}
		Q = \begin{bmatrix}
			I_r\\
			\vdots \\
			I_r\\
			I_\beta ~~ 0_{\beta \times (p-\beta)}
		\end{bmatrix}R,
	\end{equation*}
	where the $I_r$ block is repeated for $\alpha$ times in $Q$; $R$ is the $r$-by-$r$ diagonal matrix with first $\beta$ diagonal entries equal $1/\sqrt{\alpha+1}$ and the other diagonal entries equal $1/\sqrt{\alpha}$. It is easy to check that all columns of $Q$ are orthonormal, i.e., $Q\in \mathbb{O}_{p, r}$. Moreover,
	\begin{equation*}
		\max_{1\leq i \leq p}\|e_i^\top Q\|_2^2 \leq \min_{1\leq i \leq r} R_{ii}^2 = \frac{1}{\alpha} = \frac{1}{\lfloor p/r \rfloor}.
	\end{equation*}
\end{proof}

\end{document}